\documentclass[11pt]{amsart}
\usepackage{amscd,amsxtra,amssymb, bbm}
\usepackage{tikz}
\usepackage[utf8]{inputenc}
\usetikzlibrary{calc, positioning, shapes,fit, matrix,decorations}
\usetikzlibrary{decorations.shapes}
\usetikzlibrary{decorations.pathreplacing}

\tikzset{
    point/.style={circle,inner sep=1pt,minimum size=3pt,fill=red},
}
\tikzset{
  decorate with/.style={decorate,decoration={shape backgrounds,shape=#1,shape size=1.5mm}},
   deco/.style={decorate with=dart},
   ordi/.style={draw,-stealth,  very thick},
   conj/.style={dashed, draw=red, thick},
   ve/.style={circle, draw, thick, fill=blue!20, inner sep=1pt, outer sep=2pt, minimum size=7pt},
   dv/.style={star,star points=5,star point ratio=2, draw, thick, fill=green!20, inner sep=1pt,outer sep=2pt,minimum size=7pt}
 }

\makeatletter
\newcommand*{\rom}[1]{\expandafter\@slowromancap\romannumeral #1@}
\makeatother

\usepackage{multirow} 
\usepackage{stmaryrd} 

\input xy
\xyoption{all}
\xyoption{arc}
\SelectTips{cm}{10}

\newtheorem{theorem}{Theorem}[section]
\newtheorem{corollary}[theorem]{Corollary}
\newtheorem{lemma}[theorem]{Lemma}
\newtheorem{proposition}[theorem]{Proposition}

\theoremstyle{definition}
\newtheorem{definition}[theorem]{Definition}
\newtheorem{remark}[theorem]{Remark}

\DeclareMathOperator{\Perf}{\mathsf{Perf}}

\newcommand{\tr}{\mathsf{tr}}

\DeclareMathOperator{\res}{\mathsf{res}}
\DeclareMathOperator{\ev}{\mathsf{ev}}

\DeclareMathOperator{\Coh}{\mathsf{Coh}}
\DeclareMathOperator{\TF}{\mathsf{TF}}

\DeclareMathOperator{\VB}{\mathsf{VB}}

\DeclareMathOperator{\Tri}{\mathsf{Tri}}

\DeclareMathOperator{\Pic}{\mathsf{Pic}}

\DeclareMathOperator{\Hom}{\mathsf{Hom}}

\DeclareMathOperator{\Ext}{\mathsf{Ext}}

\DeclareMathOperator{\Aut}{\mathsf{Aut}}
\DeclareMathOperator{\End}{\mathsf{End}}

\DeclareMathOperator{\Spec}{\mathsf{Spec}}

\DeclareMathOperator{\lieg}{\mathfrak{g}}
\DeclareMathOperator{\lieh}{\mathfrak{h}}

\newcommand{\bul}{\scriptstyle{\bullet}}

\newcommand{\kk}{\mathbbm{k}}
\newcommand{\CC}{\mathbb{C}}
\newcommand{\NN}{\mathbb{N}}

\newcommand{\FF}{\mathbb{F}}

\newcommand{\PP}{\mathbb{P}}

\newcommand{\Ad}{\mbox{\emph{Ad}}}

\newcommand{\kA}{\mathcal{A}}

\newcommand{\kB}{\mathcal{B}}
\newcommand{\kC}{\mathcal{C}}

\newcommand{\kF}{\mathcal{F}}
\newcommand{\kG}{\mathcal{G}}
\newcommand{\kH}{\mathcal{H}}
\newcommand{\kI}{\mathcal{I}}
\newcommand{\kO}{\mathcal{O}}

\newcommand{\kP}{\mathcal{P}}
\newcommand{\kQ}{\mathcal{Q}}

\newcommand{\kK}{\mathcal{K}}

\newcommand{\kV}{\mathcal{V}}

\newcommand{\kT}{\mathcal{T}}

\newcommand{\lar}{\longrightarrow}

\newcommand{\llbrace}{(\!(}
\newcommand{\rrbrace}{)\!)}

\begin{document}

\title[Torsion free sheaves  and Yang--Baxter equation]{Torsion free sheaves on Weierstra\ss{} cubic curves    and  the classical  Yang--Baxter equation}

\author{Igor Burban}
\address{
Mathematisches Institut,
Universit\"at zu K\"oln,
Weyertal 86--90,
D--50931 K\"oln,
Germany
}
\email{burban@math.uni-koeln.de}

\author{Lennart Galinat}
\address{Mathematisches Institut,
Universit\"at zu K\"oln,
Weyertal 86--90,
D--50931 K\"oln,
Germany
}
\email{lgalinat@math.uni-koeln.de}

\begin{abstract}
This work deals with an algebro--geometric theory of solutions of the classical Yang--Baxter equation  based on torsion free coherent sheaves of Lie algebras on Weierstra\ss{} cubic curves.
\end{abstract}

\maketitle
\section{Introduction}

\noindent
This paper is devoted to  a study of  the classical Yang--Baxter equation (CYBE)
\begin{equation*}
\bigl[r^{12}(x_1, x_2), r^{13}(x_1, x_3)\bigr] + \bigl[r^{13}(x_1,x_3), r^{23}(x_2,x_3)\bigr] +
\bigl[r^{12}(x_1, x_2),
r^{23}(x_2, x_3)\bigr] = 0
\end{equation*}
from an algebro--geometric perspective.
Here,  $\mathfrak{g}$ is  a finite dimensional complex simple Lie algebra and  $r: (\CC^2, 0) \lar \mathfrak{g} \otimes \mathfrak{g}$  is the germ of a meromorphic function. Solutions of CYBE (also called $r$--matrices) have numerous applications in mathematical physics, most notably in the modern theory of classical integrable systems; see for example  \cite{FaddeevTakhtajan, ReymanST2}. The simplest solution of CYBE is the Yang's $r$--matrix
$r(x, y) = \dfrac{\gamma}{y-x}$, where $\gamma \in \lieg \otimes \lieg$ is the Casimir element.

Belavin and Drinfeld  proved that any non--degenerate skew--symmetric solution of CYBE is (up to a certain equivalence relation)  either elliptic, trigonometric or rational \cite{BelavinDrinfeld, BelavinDrinfeld2}. Elliptic solutions
of CYBE exist only for the Lie algebra $\lieg = \mathfrak{sl}_n(\CC)$; explicit formulae for them were found by Belavin \cite{Belavin}. In \cite{BelavinDrinfeld} the authors showed  that Belavin's list of elliptic $r$--matrices is in fact exhaustive. Moreover, Belavin and Drinfeld  classified all  trigonometric solutions of CYBE; the underlying combinatorial pattern (including both discrete and continuous parameters) turned out to be  rather complicated \cite{BelavinDrinfeld}. The theory of rational solutions of CYBE was developed by Stolin in \cite{Stolin, Stolin2, Stolin3}.

Let $E = \overline{V\bigl(u^2 - 4v^3 + g_2 v + g_3\bigr)} \subset \PP^2$ be a Weierstra\ss{} cubic curve, where $g_2, g_3 \in \CC$. It is well--known that  $E$ is singular  if and only if $g_2^3  -  27 g_3^2 = 0$; in this case $E$ has  a unique singular point $s$, which is a cusp for $g_2 = 0 = g_3$ and a node otherwise.
Let  $\breve{E}$ be the regular part of $E$,  $K$ the field of meromorphic functions on $E$ and $\omega$ a  regular differential
one--form on $E$ (taken in the Rosenlicht sense if  $E$ is singular; see e.g.~\cite[Section II.6]{BarthHulekPetersVen}).

\smallskip
\noindent
Next,  consider a  coherent sheaf of Lie algebras $\kA$ on $E$ such that:
\begin{enumerate}
\item $H^0(E, \kA) = 0 = H^1(E, \kA)$;
\item $\kA$ is weakly $\lieg$--locally free on $\breve{E}$, i.e.~$\kA\big|_x \cong \lieg$ for all $x \in \breve{E}$.
\end{enumerate}
The first  assumption implies that the sheaf $\kA$ is torsion free (in particular, locally free when $E$ is smooth). A consequence  of  the second assumption is that
 the space  $A_K$ of global sections of the rational envelope of $\kA$ is  a simple Lie algebra over $K$.
In the case the curve $E$ is singular, consider the $\CC$--bilinear pairing $\kappa^\omega$ given as the composition
$$
A_K \times A_K \stackrel{\kappa}\lar K \stackrel{\res_s^\omega}\lar \CC,
$$
where $\kappa$ is the Killing form of $A_K$ and $\res_s^\omega(f) = \res_s(f\omega)$ for $f \in K$ (the residue of the meromorphic one--form $f \omega$
at the point $s \in E$ is taken in the Rosenlicht sense). We impose the following additional requirement (which is automatically fulfilled  provided  the sheaf $\kA$ is locally free) on the germ $A_s$ of the sheaf $\kA$ at the singular point $s$:
\begin{enumerate}
  \setcounter{enumi}{2}
  \item $A_s$ is a coisotropic Lie subalgebra of $A_K$ with respect to the pairing $\kappa^\omega$.
\end{enumerate}

\smallskip
\noindent
The first main result of this article is the following.

\smallskip
\noindent
\textbf{Result A}. Consider a pair $(E, \kA)$, where $E$ is a Weierstra\ss{} cubic and $\kA$ a sheaf of Lie algebras on $E$ satisfying  the properties $(1)-(3)$  above.
$$
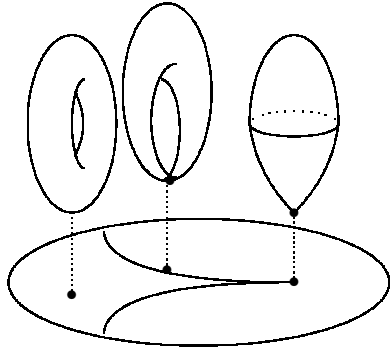
{\xy 0;/r0.20pc/:
{(30,0)\ellipse(30,10){-}},
{(10,25)\ellipse(7,14){-}},
{(10,25)\ellipse(7,14){-}},
{(12,25)\ellipse(2,7)^,=:a(180){-}},
{\ar@/^3pt/@{-}(10.5,30);(10.5,20)},
{(25,30)\ellipse(7,14){-}},
{(26.5,25.5)\ellipse(4,9)_,=:a(180){-}},
{(24,24)\ellipse(3,8)^,=:a(180){-}},
\POS(25.5,16)*{\bul};
{(45,25)\ellipse(7,14)^^,=:a(90){-}},
{(45,25)\ellipse(7,14)__,=:a(90){-}},
{(45,25)\ellipse(7,2)__,=:a(180){.}},
{(45,25)\ellipse(7,2)^_,=:a(180){-}},
\POS(52,25); \POS(45,11);
**\crv{(51,17)};
\POS(38,25); \POS(45,11);
**\crv{(39,17)};
\POS(52,25);
\POS(15,8); \POS(45,0)*{\bul} ;
**\crv{(16,0)};
\POS(15,-8); \POS(45,0);
**\crv{(16,0)};
\POS(15,8);
\POS(45,11)*{\bul} ;
{\ar@{.}(45,11); (45,0) }
\POS(25,2)*{\bul};
{\ar@{.}(25,2); (25,16) }
\POS(10,-2)*{\bul};
{\ar@{.}(10,-2); (10,11) }
\endxy}
$$
Then there exists a distinguished regular section $\rho \in \Gamma\bigl(\breve{E} \times \breve{E}\setminus \Delta, \kA \boxtimes \kA\bigr)$
(called \emph{geometric r--matrix}) which is a non--degenerate skew--symmetric solution of CYBE:
$$
\bigl[\rho^{12}, \rho^{13}\bigr] + \bigl[\rho^{12}, \rho^{23}\bigr] + \bigl[\rho^{13}, \rho^{23}\bigr] = 0,
$$
where $\Delta \subset \breve{E} \times \breve{E}$ is the diagonal and both sides of the above  equality are viewed as meromorphic sections of   $\kA \boxtimes \kA \boxtimes \kA$ over the triple product $\breve{E} \times \breve{E} \times \breve{E}$; see Theorem \ref{T:CYBEtorsfree}. Trivializing the sheaf $\kA$ over some open subset of $\breve{E}$, we get an elliptic solution of CYBE in the case $E$ is elliptic, a trigonometric solution in the case $E$ is nodal and a rational solution in the case $E$ is cuspidal; see Subsection \ref{SS:Liebialgebras}. It turns out  that  at least all \emph{rational} and \emph{elliptic} solutions of CYBE arise from an appropriate pair $(E, \kA)$; see Theorem \ref{T:geometrizationRational} and Remark \ref{R:ellipticsolutions}.

\smallskip
The idea to express  solutions of CYBE in algebro--geometric terms  was suggested for the first time by   Cherednik \cite{Cherednik}. Going into a different direction, Reyman and Semenov--Tian--Shansky observed in \cite{ReymanST} that all elliptic solutions of CYBE can be obtained from  appropriate  decompositions of the Lie algebra $\lieg\llbrace z\rrbrace$ into a direct sum of Lagrangian subalgebras
$
\lieg \llbrace z\rrbrace = \lieg \llbracket z\rrbracket \dotplus W,
$
where $\lieg \llbrace z\rrbrace$ is equipped with  the pairing $\lieg \llbrace z\rrbrace \times \lieg \llbrace z\rrbrace \lar \CC$  given by the formula $(f, g)\mapsto \res_0\bigl(\kappa(f, g)dz\bigr)$; in this case, $\bigl(\lieg \llbrace z\rrbrace, \lieg \llbracket z\rrbracket,  W\bigr)$  is a so--called \emph{Manin triple}. See also \cite[Lecture 6]{EtingofSchiffmann} for an  account of the theory of Manin triples, where some basic trigonometric and rational solutions were described in such a  way.
The idea to realize Lagrangian decompositions of $\lieg\llbrace z \rrbrace$  geometrically was suggested by Drinfeld \cite{Drinfeld}. In fact, starting with a pair $(E, \kA)$ as in Result A,  we get a \emph{canonical} Lagrangian decomposition
$
Q(\widehat{A}_p) = \widehat{A}_p \dotplus \Gamma\bigl(E\setminus \{p\}, \kA\bigr),
$
where $p = (0:1:0)$ is the ``infinite'' point of the Weierstra\ss{} cubic  $E$, the Lie algebra $\widehat{A}_p$ is the completion of the germ of $\kA$ at $p$, $Q(\widehat{A}_p)$ is the rational envelope of $\widehat{A}_p$ and the pairing on $Q(\widehat{A}_p)$ is given by the  \emph{global} differential one--form $\omega$ on $E$ (which is unique up to a scalar); see Proposition \ref{P:canManinTriples}.
We show in Theorem \ref{T:TwoWaysCYBE} that both geometric approaches to CYBE match with each other and lead to  the same power series expansion for a solution of CYBE.

The interest to the geometric theory of Yang--Baxter equations was revitalized by Polishchuk \cite{Polishchuk1}. He discovered that appropriately interpreted  triple Massey products in the derived category $D^b\bigl(\Coh(X)\bigr)$ of coherent sheaves on a  \emph{Calabi--Yau curve} $X$ (a reduced  projective curve with trivial canonical sheaf)  yield  solutions of the so--called associative Yang--Baxter equation (AYBE) with spectral parameters.
This new type of Yang--Baxter equations turned out to be closely related both to CYBE for the Lie algebra $\mathfrak{sl}_n(\CC)$ as well as to the quantum Yang--Baxter equation \cite{Polishchuk2}. Let   $\kF$ be  a \emph{simple}  locally free sheaf  of rank $n \ge 2$  on
$X$.  Then the sheaf $\kA = \Ad(\kF)$ of traceless endomorphisms of $\kF$  is a locally free sheaf of Lie algebras on $X$ with vanishing cohomology, whose fibers are isomorphic to the Lie algebra $\mathfrak{sl}_n(\CC)$. Based on the description of simple vector bundles on  Kodaira cycles of projective lines given in \cite{OldSurvey}, Polishchuk
computed the corresponding trigonometric solutions of AYBE \cite{Polishchuk2}. Recently, Lekili and Polishchuk  realized  all trigonometric solutions of AYBE in terms of certain triple Massey products in the Fukaya category of an appropriate punctured Riemann surface \cite{LekiliPolishchuk}.

Polishchuk's approach was extended in the joint works of the first--named author with Kreu\ss{}ler \cite{BK4} and Henrich \cite{BH}. In particular, the expression of  triple Massey products in geometric terms   was elaborated in detail and extended to a relative setting of genus one fibrations;  more complicated degenerations of elliptic curves (like the cuspidal Weierstra\ss{} cubic) leading to rational solutions, were included into the picture.

The main novelty of this paper is the appearance of torsion free sheaves in the theory of CYBE. It turns out that precisely torsion free sheaves on Weierstra\ss{} cubics \emph{which  are not locally free} underly  the combinatorial complexity  of the entire panorama of trigonometric and rational solutions.
 Let $X \stackrel{\mu}\lar E$ be the  \emph{Weierstra\ss{} model} of a Calabi--Yau curve $X$ and $\kF$ a simple vector bundle  on $X$. We prove that the  torsion free sheaf  $\kA:= \mu_*\bigl(Ad(\kF)\bigr)$  satisfies all  properties  which are necessary to apply Result A to the pair $(E, \kA)$; see Proposition \ref{P:minimalmodelsheaves}. As a consequence, the approach of
\cite{Polishchuk1, Polishchuk2, BH} to the study of  solutions of CYBE  for the Lie algebra $\mathfrak{sl}_n(\CC)$, based on a computation of triple Massey products for simple  vector bundles over  non--integral Calabi--Yau curves, can be  naturally included into the framework, where only Weierstra\ss{} curves
are necessary; see Theorem \ref{T:CYBEandWeierModel}.

It is a natural question to ask, to what extend the  computations with torsion free sheaves which are not locally free,  could be made constructive.
The theory of torsion free sheaves on singular curves of arithmetic genus one is nowadays relatively well--understood, thanks to the combination of methods of the theory   of matrix problems and the technique of Fourier--Mukai transforms;   see for instance \cite{DrozdGreuel, OldSurvey, Thesis, Burban1, BodnarchukDrozd, Duke, Survey, BK3, BodnarchukDrozd2,BodDroGre}.

We illustrate Result A  by the following example.
Let $\lieg$ be a  simple Lie algebra of Dynkin type $A_{n-1}$, realized in the usual way as the algebra  $\mathfrak{sl}_n(\CC)$ of traceless square matrices.  We denote by $\lieg = \lieg_+ \oplus \lieh \oplus \lieg_-$ the conventional triangular
decomposition of $\lieg$  into the direct sum of the Lie algebras of strictly upper triangular, diagonal and strictly lower triangular matrices. Let $\Phi_{\pm}$ be  the set of postive/negative roots of $\lieg$. Concretely, $\Phi_+ = \bigl\{(i, j) \in \mathbb{N}^2 \big| 1 \le i < j \le n\bigr\}$; for $\alpha = (i, j)  \in \Phi_\pm$ we write
$e_\alpha = e_{i,j}$.
Let $\xi = \exp\Bigl(\dfrac{2\pi i}{n}\Bigr)$ be a primitive $n$--th root of $1$ and $\zeta_j = \xi^j$. Then we  have the following  ``natural'' bases of the Cartan subalgebra  $\lieh$:
\begin{itemize}
\item $\bigl(g_1, \dots, g_{n-1}\bigr)$ with  $g_j = \mathsf{diag}\bigl(1, \zeta_j, \dots, \zeta_j^{n-1}\bigr)$ for
$1 \le j \le n-1$.
\item $\bigl(h_1, \dots, h_{n-1}\bigr)$ with  $h_j = \mathsf{diag}\bigl(0, \dots, 0, 1, -1, 0, \dots, 0\bigr)$, where $1$ stands at the $j$-th entry for $1 \le j \le n-1$.
\item In both cases, $\bigl(g_1^\ast, \dots, g_{n-1}^\ast\bigr)$ and $\bigl(h_1^\ast, \dots, h_{n-1}^\ast\bigr)$ denote the dual bases of $\lieh$.
\end{itemize}

\smallskip
\noindent
\textbf{Result B}. Let $E$ be a singular Weierstra\ss{} cubic and $\kQ$  a  simple torsion free and not locally free sheaf of rank $n \ge 2$ with $\chi(\kQ) = 1$ (such $\kQ$ is unique up to an isomorphism; see e.g.~\cite{BK3}). Let  $\kA$ be the sheaf of Lie algebras on $E$ defined  by the short exact sequence
$$
0 \lar \kA \lar \mbox{\it End}_E(\kQ) \stackrel{\mathsf{tr}}\lar \widetilde{\kO} \lar 0,
$$
where $\widetilde{\kO}$ is the direct image of $\kO_{\PP^1}$ under the normalization map $\PP^1 \lar E$ and  $\mathsf{tr}$ is  the trace map.
 Then the pair $(E, \kA)$ satisfies all necessary assumptions to apply Result A and gives the following solutions of CYBE for  $\lieg = \mathfrak{sl}_n(\CC)$ (see Theorem \ref{T:ExplicitSolFromTF}).

\smallskip
\noindent
1.~If  $E$ is nodal then
$
r(x, y) = r_{\mathsf{st}}(x,y) + r_{\lieh} + r_{\mathsf{sp}},
$
where
$$
r_{\mathsf{st}}(x,y) = \frac{x}{y-x}\gamma + \Bigl(\frac{1}{2}\sum\limits_{j=1}^{n-1} g_j^* \otimes g_j + \sum\limits_{\alpha \in \Phi_+} e_{-\alpha} \otimes e_{\alpha}\Bigr)
$$
is the standard quasi--trigonometric $r$--matrix \cite{BelavinDrinfeld, KPSST}, whereas
$$
r_{\lieh} = \frac{1}{2n}\sum\limits_{j = 1}^{n-1} \Bigl(\frac{1 + \zeta_j}{1 - \zeta_j}\Bigr) g_{n-j} \otimes g_j \quad \mbox{\rm and} \quad
 r_{\mathsf{sp}} = \sum\limits_{\alpha \in \Phi_+} e_{-\alpha} \wedge \Bigl(\sum\limits_{k = 1}^{p(\alpha)} e_{\tau^k(\alpha)}\Bigr).
$$
Here,  for $\alpha = (i, j) \in \Phi_+$ we put: $p(\alpha) := i-1$ and $\tau(\alpha) := (i-1, j-1) \in \Phi_+$ provided  $i \ge 2$, whereas
$a \wedge b := a\otimes b - b \otimes a$ for $a, b \in \lieg$.

\smallskip
\noindent
2.~If  $E$ is cuspidal  then
$$
r(x, y) = \frac{\gamma}{y-x} + \sum\limits_{k = 1}^{n-1} h_k^* \wedge e_{k+1, k} + \sum \limits_{k \ge l+2} \bigl(\sum\limits_{j= 0}^{l-1} e_{l-j, k - j-1}\bigr)\wedge e_{k,l}.
$$

\smallskip
\noindent
\textit{List of notations}. For convenience of the reader we introduce now  the most important  notations used in this paper.

\smallskip
\noindent
$-$ $\kk$ denotes an algebraically closed field
of characteristic zero. From the point of view of applications to CYBE, the case $\kk = \CC$ is of  main interest.

\smallskip
\noindent
$-$ We work in the category of algebraic varieties (respectively, algebraic schemes) over $\kk$. The usage of the K\"unneth isomorphism is one of the reasons for this choice. However, for $\kk = \CC$ it is possible to pass to the complex--analytic category at most places. In particular, we talk about \emph{meromorphic functions} on $X$,  whereas the terminology \emph{rational functions} would be more consistent.

\smallskip
\noindent
$-$ Given an  algebraic variety (or  scheme) $X$, $\breve{X}$ always stands for the regular part of $X$. Next, $\Coh(X)$ (respectively,
$\VB(X), \TF(X)$) denotes the category of coherent sheaves (respectively, locally free sheaves, torsion free sheaves) on $X$;
$D^b\bigl(\Coh(X)\bigr)$ is the derived category of $\Coh(X)$.
 We denote
by $\kO_X$ the structure sheaf of $X$. If $X$ is Cohen--Macaulay then   $\Omega_X$ denotes the dualizing sheaf of $X$.

\smallskip
\noindent
$-$
We  write
$\Hom$ and  $\End$ when working with global morphisms  between
coherent sheaves whereas  $\mbox{\textit{Hom}}$ and $\mbox{\textit{End}}$ stand for interior Homs in $\Coh(X)$. If $\kF$ is a locally free sheaf on $X$ then $\Ad(\kF)$ denotes   the sheaf of traceless endomorphisms of  $\kF$.

\smallskip
\noindent
$-$ For a coherent sheaf $\kF$ on $X$ and a point $x  \in X$,  we denote by
$\kF\big|_x$ the fiber of $\kF$ over $x$. We use all possible notations $\Gamma(X, \kF), H^0(X, \kF)$ and $\kF(X)$ for the space of global sections of $\kF$. If $X$ is projective then $\chi(\kF)$ denotes  the Euler characteristic of $\kF$.

\smallskip
\noindent
$-$
 A  \emph{Calabi--Yau curve}  $X$ is a reduced projective Gorenstein curve with trivial canonical  sheaf.  The irreducible Calabi--Yau curves are the \emph{Weierstra\ss{} cubics} $$wv^2 = 4 u^3 -  g_2 uw^2 - g_3 w^3,$$
  where $g_1, g_2 \in \kk$.
The complete list of Calabi--Yau  curves is  known; see for
    example \cite[Section 3]{Smyth}. Namely, a Calabi--Yau curve  $X$ is either
 \begin{itemize}
    \item an elliptic curve;
    \item a  cycle of $m \ge 1 $ projective lines (Kodaira fiber  $\mathrm{I}_m$; for  $m = 1$ it is a  nodal Weierstra\ss{} cubic);
    \item a cuspidal Weierstra\ss{} cubic (Kodaira fiber II), a tachnode plane cubic curve (Kodaira fiber III) or a generic configuration of $n$ concurrent lines in $\mathbb{P}^{m-1}$ for $m \ge 3$.
        \end{itemize}

\smallskip
\noindent
$-$ In this work, $\lieg$ always denotes a finite dimensional complex simple Lie algebra; $\lieg \times \lieg \stackrel{\kappa}\lar \CC$ is the Killing form of $\lieg$ and
$\gamma \in \lieg \otimes \lieg$ is the Casimir element.

\medskip
\noindent
\emph{Acknowledgement}. This  work  was supported  by the DFG project Bu--1866/3--1. We are grateful    to   Alexander Stolin for helpful discussions.

\section{Survey on the classical Yang--Baxter equation}

\noindent
Let $\mathfrak{g}$ be a finite dimensional simple complex  Lie algebra and $r: (\CC^2, 0) \lar \mathfrak{g} \otimes \mathfrak{g}$  the germ of a meromorphic function.
We begin with the so--called  \emph{generalized  classical Yang--Baxter equation} (GCYBE), which   is the following system of constraints on the coefficients of $r$:
\begin{equation}\label{E:GCYBE}
\bigl[r^{12}(x_1, x_2), r^{13}(x_1, x_3)\bigr] + \bigl[r^{13}(x_1, x_3), r^{23}(x_2,x_3)\bigr] +
\bigl[r^{32}(x_3, x_2), r^{12}(x_1, x_2)\bigr] = 0.
\end{equation}
The upper indices in this equation   indicate various embeddings of
$\mathfrak{g}\otimes \mathfrak{g}$ into $U \otimes U \otimes U$, where $U$ is an arbitrary associative algebra containing $\lieg$ (for example, the universal enveloping algebra of $\lieg$).  The function $r^{13}$ is defined as
$$
r^{13}: \mathbb{C}^2 \stackrel{r}\lar \mathfrak{g}\otimes \mathfrak{g}
\stackrel{\imath_{13}}\lar U \otimes U  \otimes U,
$$
where $\imath_{13}(a\otimes b) = a \otimes {1} \otimes b$. The other
maps $r^{12}$, $r^{23}$ and $r^{32}$ have a similar meaning. Note that for instance
$$
\bigl[(a\otimes b)^{12}, (c \otimes d)^{13}\bigr] = \bigl[a\otimes b \otimes 1, c \otimes 1 \otimes d\bigr] = [a, c] \otimes b \otimes d \in \lieg \otimes \lieg \otimes \lieg
$$
for arbitrary $a, b, c, d \in \lieg$. Hence, the left--hand side of (\ref{E:GCYBE}) does not depend on the embedding of $\lieg$ into an ambient associative algebra $U$.

\smallskip
\noindent
The Killing form $\lieg \times \lieg \stackrel{\kappa}\lar \CC$ induces an isomorphism of vector spaces
\begin{equation}\label{E:KillingIsom}
\mathfrak{g} \otimes \mathfrak{g} \stackrel{\widetilde\kappa}\lar \End(\mathfrak{g}), \quad a \otimes b \mapsto \bigl(c \mapsto
\kappa(ac) \cdot b\bigr).
\end{equation}
A solution $r$ of (\ref{E:GCYBE}) is called \emph{non--degenerate}, if for some (hence for a generic) point $(x_1, x_2)$ in the domain of definition of $r$, the linear
map $\widetilde\kappa\bigl({r}(x_1, x_2)\bigr) \in \End(\lieg)$ is an isomorphism.

\smallskip
\noindent
One can perform the following transformations with a solution $r$  of (\ref{E:GCYBE}).
\begin{itemize}
\item \emph{Gauge transformations}. This means that for any holomorphic  germ
$(\mathbb{C},0) \stackrel{\phi}\lar {\mathsf{Aut}}(\mathfrak{g})$,  the  function
\begin{equation}\label{E:gaugeequiv}
 \widetilde{r}(x_1, x_2) :=
\bigl(\phi(x_1) \otimes \phi(x_2)\bigr) r(x_1,x_2).
\end{equation}
is again a  solution of  (\ref{E:GCYBE}).
\item \emph{Rescaling}. For any  germ of a non--zero  meromorphic function $(\CC, 0) \stackrel{u}\lar \CC$
\begin{equation}\label{E:equiv}
 \check{r}(x_1, x_2) :=
u(x_2) r(x_1,x_2)
\end{equation}
is again a  solution of  (\ref{E:GCYBE}).
\item \emph{Change of variables}. Let $(\CC, 0) \stackrel{\eta_i}\lar (\CC, 0)$ be a germ automorphism for $i = 1, 2$. Then the  function
\begin{equation}\label{E:equiv2}
 \hat{r}(x_1, x_2) :=
r\bigl(\eta_1(x_1), \eta_2(x_2)\bigr).
\end{equation}
is again a  solution of  (\ref{E:GCYBE}).
\end{itemize}
It is easy to see that all these transformations preserve non--degeneracy of a solution of (\ref{E:GCYBE}).
Next, a solution $r$ of (\ref{E:GCYBE}) is  called \emph{skew--symmetric}
if
\begin{equation}\label{E:skewsymmetry}
r(x_1, x_2) = - \sigma\bigl(r(x_2, x_1)\bigr),
\end{equation}
where
$\mathfrak{g} \otimes \mathfrak{g} \stackrel{\sigma}\lar \lieg \otimes \lieg$ is given by the rule $a\otimes b \mapsto b \otimes a$. For skew--symmetric solutions,
the generalized classical Yang--Baxter equation (\ref{E:GCYBE}) takes the form
\begin{equation}\label{E:CYBE}
\bigl[r^{12}(x_1, x_2), r^{13}(x_1, x_3)\bigr] + \bigl[r^{13}(x_1,x_3), r^{23}(x_2,x_3)\bigr] +
\bigl[r^{12}(x_1, x_2),
r^{23}(x_2, x_3)\bigr] = 0.
\end{equation}
This is the conventional \emph{classical Yang--Baxter equation} with \emph{two} spectral variables (CYBE). Automorphisms of skew--symmetric solutions  (\ref{E:GCYBE}) preserving the skew--symmetry  are more restrictive. Namely, we are allowed to  rescale a solution only by a constant function;  a change  of variables has to be performed by the same automorphism $\eta = \eta_1 = \eta_2$. However, the skew--symmetry  is preserved under arbitrary gauge transformations.

By a result of   Belavin and Drinfeld \cite{BelavinDrinfeld2},  any \emph{non--degenerate unitary}
solution of  (\ref{E:CYBE})  is equivalent (after a change of variables and a gauge transformation) to a solution
 $r(x_1, x_2) =  \varrho(x_2-x_1)$,  depending just on the difference (or  the quotient) of spectral
parameters. In other words, (\ref{E:CYBE}) reduces to the equation
\begin{equation}\label{E:CYBE1}
\bigl[\varrho^{12}(x), \varrho^{13}(x+y)\bigr] + \bigl[\varrho^{13}(x+y), \varrho^{23}(y)\bigr] +
\bigl[\varrho^{12}(x), \varrho^{23}(y)\bigr] = 0
\end{equation}
for the  germ of a meromorphic function $(\CC, 0) \stackrel{\varrho}\lar \lieg \otimes \lieg$ (the so--called CYBE with \emph{one} spectral parameter).
According to  Belavin and Drinfeld \cite{BelavinDrinfeld}, any  non--degenerate
solution of (\ref{E:CYBE1}) is automatically skew--symmetric and  has a simple pole at $0$ with  residue equal
to a multiple of the Casimir element $\gamma \in \lieg \otimes \lieg$. This means that the theory of skew--symmetric solutions of  the classical Yang--Baxter equation with two spectral parameters  (\ref{E:CYBE}) basically reduces to a description  of solutions of (\ref{E:CYBE1}). For some applications, dependence on the difference of spectral parameters is crucial, but from the theoretical  perspective,  the version  of the classical Yang--Baxter equation with two spectral parameters (\ref{E:CYBE}) seems to be  more natural.

A few words about  the history of the subject. The concept on an $R$--matrix controlling exact solvability of certain models of mathematical physics appeared in works of Yang and Baxter in 60's--70's. The development  of the underlying  algebraic theory, including in  particular  the precise formulation of   the quantum Yang--Baxter equation itself, was initiated by Faddeev and his school; see e.g.~\cite{FaddeevTakhtajan}.
The classical Yang--Baxter equation (\ref{E:CYBE1}) was  introduced by Sklyanin \cite{Sklyanin} and Belavin \cite{Belavin}. Nowadays, it  plays a central role in the modern theory of classical integrable systems via an appropriate version of the so--called Adler--Kostant--Symes scheme; see e.g.~\cite{FaddeevTakhtajan, ReymanST2}.
The generalized classical Yang--Baxter equation (\ref{E:GCYBE}) was discovered by Cherednik \cite{Cherednik2}. Solutions of this equation are also interesting from the point of view of  applications in the mathematical physics \cite{AvanTalon, Maillet}. New non--skew--symmetric solutions of (\ref{E:GCYBE}) of all types (elliptic, trigonometric, rational) were constructed   by Skrypnyk \cite{Skrypnyk1, Skrypnyk}; applications of (\ref{E:GCYBE})  both to classical and quantum integrable systems were extensively studied by this author; see e.g.~\cite{Skrypnyk1,Skrypnyk,  Skrypnyk2}.

\section{Geometric solutions of the generalized classical Yang--Baxter equation}\label{S:GCYBE}

\subsection{The residue sequence} In this subsection, let  $X$ be an algebraic curve over $\kk$ (not necessarily integral), $C \subset X$ a connected  non--empty \emph{smooth} open subset, $C \stackrel{\delta}\lar  X \times C$ the diagonal embedding and $\Delta = \mathsf{Im}(\delta)$.

\begin{proposition}\label{P:ResidueSeq} There exists the following short exact sequence of coherent sheaves on the algebraic surface $X \times C$ (called  \emph{residue sequence}):
\begin{equation}\label{E:Residue}
0 \lar \kO_{X \times C} \lar \kO_{X \times C}(\Delta) \xrightarrow{\res_\Delta} \delta_*\bigl(\mbox{\textit{Hom}}_C(\Omega_C, \kO_C)\bigr) \lar 0.
\end{equation}
\end{proposition}

\begin{proof} We have to define the morphism of sheaves $\res_\Delta$. For any point $p \in C$, let $U \subset X \times C$ be an open neighbourhood of $(p, p) \in X \times C$ and
$V := \delta^{-1}\bigl(U \cap \Delta)$. Let $f \in \Gamma\bigl(U, \kO_{X \times C}(\Delta)\bigr)$ and $\omega \in \Gamma(V, \Omega_C)$. Diminishing if necessary the open neighbourhood $U$, we may write in some local coordinates:
\begin{equation}\label{E:localforms}
f(x, t) = \frac{g(x, t)}{x-t} \quad \mbox{and} \quad \omega = \psi(t)dt
\end{equation}
for some regular functions $g \in \Gamma \bigl(U, \kO_{X \times C}\bigr)$ and $\psi \in \Gamma(V, C)$. Then we put:
\begin{equation}\label{E:DefResidue}
\bigl(\res_\Delta(f)\bigr)(\omega) := \bigl(V \ni y \mapsto \psi(y) g(y, y\bigr).
\end{equation}
The   map $\Gamma\bigl(U, \kO_{X \times C}(\Delta)\bigr) \lar \Hom_{\kO_C(V)}\bigl(\Omega_C(V), \kO_C(V)\bigr)$ defined in this way, is independent of all the choices made and  determines   the morphism of coherent sheaves $\res_\Delta$. The exactness of (\ref{E:Residue}) is then obvious.
\end{proof}

\smallskip
\noindent
Assume that the open subset $C \subset X$ is such that there exists a nowhere vanishing regular differential form $\omega = \tau(z) dz \in \Gamma(V, \Omega_C)$, defining a trivialization of $\Omega_C$.  Then the short exact sequence (\ref{E:Residue}) can be rewritten in the following way:
\begin{equation}\label{P:ResidueSeqInduced}
0 \lar \kO_{X \times C} \lar \kO_{X \times C}(\Delta) \xrightarrow{\res_\Delta^\omega} \kO_\Delta \lar 0.
\end{equation}
For further applications, let us respell the definition of $\res_\Delta^\omega$ explicitly.
As above, for  any point $p \in C$, consider  an open neighbourhood $U \subset X \times C$  of  $(p, p)$, set
$V := \delta^{-1}\bigl(U \cap \Delta)$ and  take any  $f \in \Gamma\bigl(U, \kO_{X \times C}(\Delta)\bigr)$. Then we put:
$
\res_\Delta^\omega(f) := \bigl(V \ni y \mapsto \tau(y) g(y, y)\bigr),
$
where
$
f(x, t) = \dfrac{g(x, t)}{x-t}$.
The proof of the next result is straightforward.

\begin{lemma}\label{L:PropertiesResSeq}  The short exact sequence (\ref{P:ResidueSeqInduced}) has the following functorial properties with respect to restrictions:
\begin{itemize}
\item If $C' \subset C$ is an open subset and $\omega' := \omega\big|_{C'}$ then the restriction of (\ref{P:ResidueSeqInduced}) on the open subset $X \times C' \subseteq X \times C$ can be canonically identified with  the residue sequence given by the differential form $\omega'$.
\item For any point $p \in C$, the restriction of (\ref{P:ResidueSeqInduced}) on the closed subset $X \times \{p\} \subset X \times C$ can be canonically identified with the short exact sequence
\begin{equation}\label{E:ResSequence}
0 \lar \kO_X \lar \kO_X(p) \xrightarrow{{\res}_p^{\omega}} \kk_p \lar 0,
\end{equation}
where $\res_p^\omega(h) := \res_p(h \omega)$ for any local meromorphic    function $h$ in a neighbourhood of $p$ with at most simple pole at $p$.
\end{itemize}
\end{lemma}

\subsection{Some basic facts  about sheaves of Lie algebras} In this subsection, $C$ is  an integral curve over $\kk$ and  $\lieg$  a simple Lie algebra over $\kk$.

\begin{definition}\label{D:locfreeLieAlg}
Let  $\kA$ be a coherent sheaf of Lie algebras on $C$. We say that $\kA$ is \emph{weakly $\lieg$--locally free} if
$\kA\big|_{x} \cong \lieg$ for any $x \in C$.
\end{definition}

\begin{remark}
Note that Definition \ref{D:locfreeLieAlg} \emph{does not imply} that for any $p \in C$ there exists an open neighbourhood $p \in V \subset C$ such that
$\Gamma(V, \kA) \cong \lieg \otimes_\kk \Gamma(V, \kO_C)$ as Lie algebras over the ring $\Gamma(V, \kO_C)$. However, because of the vanishing
$H^2(\lieg; \lieg) = 0$, one can conclude that $\widehat{A}_p \cong \lieg \otimes_{\kk} \widehat{O}_p$ for any $p \in C$, where $\widehat{A}_p$ and $\widehat{O}_p$ are the completions of the germs of $\kA$ and $\kO_C$ at the point $p$; see for example \cite[Section 2]{Fialowski} for a proof.
\end{remark}

\begin{proposition} Let $\kA$ be a weakly $\lieg$--locally free sheaf of Lie algebras on $C$. Then the following results are true.
\begin{itemize}
\item $\kA$ is locally free on $C$ viewed as a coherent sheaf.
\item Let $C = V_1 \cup \dots \cup V_n$ be an affine open covering  such that the Lie algebra $\Gamma(V_i, \kA)$ is free viewed as  $\Gamma(V_i, \kO_C)$--module
for any $1 \le i \le n$. Then the local Killing--forms
$$
\Gamma(V_i, \kA) \times \Gamma(V_i, \kA) \stackrel{\kappa_i}\lar \Gamma(V_i, \kO_C), \quad (f, g) \mapsto \tr\bigl(\mathsf{ad}(f) \mathsf{ad}(g)\bigr)
$$
coincide on the intersections $V_i \cap V_j$ for all $1 \le i \ne j \le n$ and define a global isomorphism of $\kO_C$--modules
\begin{equation}
\kA \otimes  \kA \stackrel{\widetilde\kappa}\lar \mbox{\textit{End}}_C(\kA).
\end{equation}
\item Let $\gamma \in \Gamma(C, \kA \otimes \kA)$ be the \emph{Casimir element} (i.e.~the preimage of the identity map  under
$\Gamma(C, \kA \otimes \kA) \stackrel{\widetilde\kappa}\lar \End_C(\kA)$). Then we have:
    \begin{itemize}
        \item For any open subset $V \subset C$ and $f \in \Gamma(V, \kA)$ it holds:
\begin{equation}\label{E:Casimir}
\bigl[\gamma\big|_V, f \otimes 1 + 1 \otimes f\bigr] = 0 \in \Gamma(V, \kA \otimes \kA).
\end{equation}
\item For any $x \in C$, the element of the fiber  $\gamma\big|_x \in \kA\big|_x \otimes_{\kk} \kA\big|_x$ is the Casimir element of the simple Lie algebra $\kA\big|_x \cong \lieg$.
    \end{itemize}
\item Let $\kK$ be the sheaf  of meromorphic functions on $C$, $\kA_\kK := \kA \otimes \kK$ be the rational envelope of $\kA$ and $K := \Gamma(C, \kK)$. Then  $A_K := \Gamma(C, \kA_\kK)$ is a simple Lie algebra over the field $K$ (however, in general  $A_K \not\cong \lieg \otimes_\kk K$).
\end{itemize}
\end{proposition}

\begin{proof}
The first statement  follows from the assumption that $C$ is reduced. The proof of the second result is straightforward. The equality (\ref{E:Casimir}) follows from the fact that the isomorphism of $\kO(V)$--modules
$
\Gamma(V, \kA) \otimes_{\kO(V)} \Gamma(V, \kA) \stackrel{\widetilde\kappa}\lar \End_{\kO(V)}\bigl(\Gamma(V, \kA)\bigr)
$
is also an isomorphism of $\kO(V)$--linear representations of $\Gamma(V, \kA)$. The second  property of $\Omega$ is obvious. The last statement follows from the fact that the Killing form $A_K \times A_K \lar K$
is non--degenerate, which  insures the semi--simplicity of $A_K$.  Now assume that $A_K \cong B_1 \times \dots \times B_m$, where $m \ge 2$ and $B_1, \dots, B_m$ are some simple Lie algebras over $K$. Then there exists an open set $U \subset C$ and Lie algebras $A_1, \dots, A_m$ over the ring $O= \kO(U)$ such that
each $A_i$ is a finitely generated free  $O$--module satisfying $\kA_i \otimes_O K \cong B_i$ and $\Gamma(U, \kA) \cong A_1 \times \dots \times A_m$.
But then for any $x \in U$, the fiber $\kA\big|_x$ splits into  a non--trivial product of  Lie algebras, contradiction.
\end{proof}

\subsection{Algebro--geometric approach to GCYBE}\label{SS:GeometrizationGCYBE}
We start with the following algebro--geometric datum $\bigl((X, \kA), (C, \omega)\bigr)$, where
\begin{itemize}
\item $X$ is a Cohen--Macaulay projective curve over $\kk$ (which need not be integral in general).
\item $\kA$ is a coherent sheaf of Lie algebras on $X$ such that $H^0(X, \kA) = 0 = H^1(X, \kA)$.
\item $C \subset X$ is a non--empty smooth \emph{affine} subset such that $\kA^\circ := \kA\big|_C$ is weakly $\lieg$--locally free on $C$.
\item $\omega \in \Gamma(C, \Omega_C)$ is a nowhere vanishing differential one--form.
\end{itemize}
We shall use the following notation: $A = \Gamma(C, \kA^\circ)$ and $O = \Gamma(C, \kO_C)$; $C \stackrel{\delta}\lar  X \times C$ is the diagonal
embedding and
$\Delta = \mathsf{Im}(\delta)$.
\begin{proposition}\label{P:BasicsGeomdata} Let $\bigl((X, \kA), (C, \omega)\bigr)$ be a datum as above.
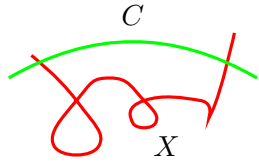
\begin{center}
\begin{tikzpicture}[scale=0.3]
\draw [red,very thick, xshift=-4cm] plot [smooth, tension=1] coordinates { (-1,2) (1,0) (2,-2) (0,-2) (1,0)( 2.5,1) (4,0) (4.5,-1)(3.5,-1)(4,0)(6.5,0) (7,-0.5)(8,3)};
\draw  [green, very thick, xshift=-4cm] (-2,1)to  [bend left,->] node[above= 2 pt, black]{$C$} (9,1);
\node at (1,-2){$X$} ;
\end{tikzpicture}
\end{center}
Then the following results are true.
\begin{itemize}
\item The coherent sheaf $\kA$ is Cohen--Macaulay.
\item The residue sequence (\ref{E:ResSequence}) induces a short exact sequence
\begin{equation}\label{E:ResidueSheafA}
0 \lar \kA \boxtimes \kA^\circ \lar \kA \boxtimes \kA^\circ(\Delta) \xrightarrow{\res_\Delta^\omega} \delta_*\bigl(\kA^\circ \otimes \kA^\circ \bigr) \lar 0.
\end{equation}
\item We have the vanishing $H^i\bigl(X \times C, \kA \boxtimes \kA^\circ\bigr) = 0$ for all $i \in \mathbb{N}_0$.
\end{itemize}
\end{proposition}
\begin{proof}
From  $\Gamma(X, \kA) = 0$ we conclude  that $\kA$ is torsion free, hence Cohen--Macaulay.

\smallskip
\noindent
Let $X \stackrel{p_1}\longleftarrow X \times C \stackrel{p_2}\lar C$ be the canonical projections.
Since $\kA^\circ$ is locally free over $C$, the induced sequence
$$
0 \lar p_2^* \kA^\circ \lar p_2^* \kA^\circ(\Delta) \lar   p_2^* \kA^\circ \otimes \kO_\Delta\lar 0
$$
is still exact. Next, the sheaf $\kO_\Delta$ is flat over $X$, hence $\mbox{\textit{Tor}}_1^{X \times C}\bigl(p_1^*(\kA), p_2^* \kA^\circ \otimes \kO_\Delta\bigr) = 0$. This proves  exactness of (\ref{E:ResidueSheafA}). According to  K\"unneth formula, we have
a  quasi--isomorphism of complexes of vector spaces:
$\mathrm{R\Gamma}\bigl(X \times C, \kA \boxtimes \kA^\circ\bigr) \cong \mathrm{R\Gamma}\bigl(X, \kA) \stackrel{\mathbbm{L}}\otimes \mathrm{R\Gamma}\bigl(C, \kA^\circ) \cong 0.
$
This implies the third statement.
\end{proof}
\begin{corollary}\label{C:residue}
Applying the functor $\Gamma\bigl(X \times C, \,-\,\bigr)$ to the short exact sequence (\ref{E:ResidueSheafA}), we get an isomorphism of vector spaces:
$
\Gamma\bigl(X \times C, \kA \boxtimes \kA^\circ(\Delta)\bigr) \xrightarrow{\res_\Delta^\omega} \Gamma(C, \kA^\circ \otimes \kA^\circ).
$
\end{corollary}
\begin{definition}
Consider the following commutative diagram of vector spaces over $\kk$:
\begin{equation}
\begin{array}{c}
\xymatrix{
\Gamma\bigl(C, \kA^\circ \otimes \kA^\circ\bigr) \ar[dd]_-{\widetilde\kappa} & & \Gamma\bigl(X \times C, \kA \boxtimes \kA^\circ(\Delta)\bigr) \ar[ll]_-{\res_\Delta^\omega} \ar@{^{(}->}[d] \\
 && \Gamma\bigl(X \times C \setminus \Delta, \kA \boxtimes \kA^\circ\bigr) \ar[d] \\
 \End_C(\kA^\circ) \ar[rr]^-{T_\omega} & & \Gamma\bigl(C \times C \setminus \Delta, \kA^\circ \boxtimes \kA^\circ\bigr).
}
\end{array}
\end{equation}
Here, both right vertical maps are the canonical restrictions, $\widetilde\kappa$ is the isomorphism induced by the Killing form,  $\res_\Delta^\omega$
is the isomorphism from Corollary \ref{C:residue} and $T_\omega$ is defined
as the composition of known maps. The   section
\begin{equation}\label{E:SzegoKernel}
\rho = \rho_\omega := T_\omega\bigl(\mathrm{id}_{\kA^\circ}\bigr) \in \Gamma\bigl(C \times C \setminus \Delta, \kA^\circ \boxtimes \kA^\circ\bigr)
 \end{equation}
 is canonical up to a choice of the differential form $\omega$ and
 is called \emph{geometric $r$--matrix} attached to the datum $\bigl((X, \kA), (C, \omega)\bigr)$.
\end{definition}

\begin{lemma}  Assume that in some local coordinates on the curve $C$,  $\omega = \dfrac{1}{\nu(x)} dx$, where $\nu\in \Gamma(C, \kO_C)$ is nowhere vanishing. Then there exists a lift
$\widetilde{\gamma} \in \Gamma\bigl(C \times C, \kA^\circ \boxtimes \kA^\circ\bigr)$  of the Casimir element $\gamma$ under the canonical surjective map
$\Gamma\bigl(C  \times C, \kA^\circ \boxtimes \kA^\circ\bigr) \stackrel{\delta^*}\lar \Gamma\bigl(C, \kA^\circ \otimes \kA^\circ\bigr)$ and
 $\upsilon \in \Gamma\bigl(C \times C, \kA^\circ \boxtimes \kA^\circ\bigr)$ such that for any
point $(x_1, x_2) \in (C \times C)\setminus \Delta$ we have:
\begin{equation}\label{E:structureRmat}
\rho(x_1, x_2) = \frac{\nu(x_2)}{x_1 - x_2} \widetilde{\gamma} + \upsilon(x_1, x_2).
\end{equation}
\end{lemma}

\begin{proof}
First observe that the following diagram of vector spaces
$$
\xymatrix{
\Gamma\bigl(C \times C, \kA^\circ \boxtimes \kA^\circ\bigr) \ar[d]_-\cong  \ar[rr]^-{\delta^*} & & \Gamma\bigl(C, \kA^\circ \otimes \kA^\circ\bigr)
\ar[d]^-\cong \\
\Gamma(C, \kA^\circ) \otimes_\kk \Gamma(C, \kA^\circ) \ar[rr] & & \Gamma(C, \kA^\circ) \otimes_O\Gamma(C, \kA^\circ)
}
$$
is commutative (here, all maps are the canonical ones). Since the lower horizontal map is obviously surjective, the map $\delta^*$ is surjective, too.
Let $\widetilde{\gamma}$ be any preimage of $\gamma$ under $\delta^*$.
Consider the following commutative diagram:
\begin{equation*}
\xymatrix{
 & & \Gamma\bigl(X \times C, \kA \boxtimes \kA^\circ(\Delta)\bigr) \ar[d]_-{\mathsf{restr}} \ar[rd]^-{\res_\Delta^\omega} & & & \\
 0  \ar[r] & \Gamma\bigl(C \times C, \kA^\circ \boxtimes \kA^\circ\bigr) \ar[r] & \Gamma\bigl(C \times C, \kA^\circ \boxtimes \kA^\circ(\Delta)\bigr)
 \ar[r]^-{\underline{\res}_\Delta^\omega} & \Gamma(C, \kA^\circ \otimes \kA^\circ) \ar[r] & 0,
}
\end{equation*}
where the lower sequence is induced by the residue sequence on $C \times C$ and $\mathsf{restr}$ is the restriction map. By Corollary \ref{C:residue}, the linear map $\res_\Delta^\omega$
is an isomorphism. Let
\begin{equation}\label{E:CanRmatrix}
\varrho  := \bigl(\res_\Delta^\omega\bigr)^{-1}(\gamma) \in \Gamma\bigl(X \times C, \kA^\circ \boxtimes \kA^\circ(\Delta)\bigr),
\end{equation}
then  by the definition we have:  $\rho = \mathsf{restr}(\varrho)$. On the other hand, consider the   section $\widetilde\rho \in \Gamma\bigl(C \times C,
\kA^\circ \boxtimes \kA^\circ(\Delta)\bigr)$ given for any $x_1 \ne x_2 \in C$ by the expression $\widetilde{\rho}(x_1, x_2) = \dfrac{\tau(x_2)}{x_1-x_2}\widetilde{\gamma}$. If $\upsilon:= \rho - \widetilde{\rho}$
then
$
\underline{\res}_\Delta^\omega(\upsilon)  = \gamma - \gamma = 0.
$
This implies that $\upsilon  \in \Gamma\bigl(C \times C, \kA^\circ \boxtimes \kA^\circ\bigr)$.
\end{proof}

\begin{remark}
The element $\widetilde{\gamma}\in \Gamma\bigl(C, \kA^\circ\bigr) \otimes_{\kk} \Gamma\bigl(C, \kA^\circ\bigr)\cong  \Gamma\bigl(C \times C, \kA^\circ \boxtimes \kA^\circ\bigr)$  in the presentation (\ref{E:structureRmat}) is unique only up to transformations of the form $\widetilde\gamma \mapsto \widetilde\gamma + \eta$, where $\eta \in \Gamma\bigl(C \times C, \kA^\circ \boxtimes \kA^\circ\bigr)$ is any section such that $\eta\big|_{\Delta} = 0$.
\end{remark}

\begin{theorem}\label{T:GCYBE} The section $\rho$ satisfies  the generalized classical Yang--Baxter equation:
\begin{equation}\label{E:geomGCYBE}
\bigl\llbracket \rho, \rho\bigr\rrbracket := \bigl[\rho^{12}, \rho^{13}\bigr] + \bigl[\rho^{12}, \rho^{23}\bigr] + \bigl[\rho^{32}, \rho^{13}\bigr] = 0
\end{equation}
where both sides are viewed as meromorphic   sections of the sheaf $\kA^\circ \boxtimes \kA^\circ \boxtimes \kA^\circ$ on $C \times C \times C$.
\end{theorem}

\begin{proof} We first explain the main steps of the proof.
\begin{itemize}
\item Let $\varrho \in \Gamma\bigl(X \times C, \kA \boxtimes \kA^\circ\bigr)$ be the section given by (\ref{E:CanRmatrix}). We shall  prove that
\begin{equation}\label{E:geomGCYBE2}
\bigl\llbracket \varrho, \varrho\bigr\rrbracket := \bigl[\varrho^{12}, \varrho^{13}\bigr] + \bigl[\varrho^{12}, \varrho^{23}\bigr] + \bigl[\varrho^{32}, \varrho^{13}\bigr] = 0,
\end{equation}
where both sides are viewed as meromorphic   sections of the sheaf $\kA \boxtimes \kA^\circ \boxtimes \kA^\circ$ on $X \times C \times C$. The equality (\ref{E:geomGCYBE}) is obtained by restricting (\ref{E:geomGCYBE2}) on $C \times C \times C$.
\item  Let $\Sigma_{12}, \Sigma_{13}, \Sigma_{23} \subset X \times C \times C$ be the divisors given by the rules
$$\Sigma_{12} := \bigl\{(t_1, t_1, t_2)\big| (t_1, t_2) \in C \times C\bigr\}$$ etc. We put $\Sigma  = \Sigma_{12} + \Sigma_{13} + \Sigma_{23} = \Sigma_{12} \cup \Sigma_{13} \cup \Sigma_{23}$ (viewing $\Sigma$ as a divisor respectively as a subvariety of $X \times C \times C$). Then we have:
$$\llbracket \varrho, \varrho\rrbracket \in \Gamma\bigl(X\times C \times C, \kA \boxtimes \kA^\circ \boxtimes \kA^\circ(\Sigma)\bigr).$$
In other words, $\llbracket \varrho, \varrho\rrbracket$ is a rational section of $\kA \boxtimes \kA^\circ \boxtimes \kA^\circ$ having at most simple poles along the divisors $\Sigma_{12}, \Sigma_{13}$ and  $\Sigma_{23}$.
\item Computing the residues, we show that the section $\llbracket \varrho, \varrho\rrbracket$ has \emph{no poles} meaning that
$$\llbracket \varrho, \varrho\rrbracket \in \mathsf{Im}\Bigl(\Gamma\bigl(X\times C \times C, \kA \boxtimes \kA^\circ \boxtimes \kA^\circ\bigr) \lar  \Gamma\bigl(X\times C \times C, \kA \boxtimes \kA^\circ \boxtimes \kA^\circ(\Sigma)\bigr)\bigr).$$
\item By K\"unneth formula, $\Gamma\bigl(X \times C \times C, \kA \boxtimes \kA^\circ \boxtimes \kA^\circ\bigr) = 0$, implying  (\ref{E:geomGCYBE}).
\end{itemize}
Now we proceed with  the details. To simplify the notation, we shall write: $T := C \times C$; let  $T \stackrel{\pi}\lar C$ be the canonical projection on the first component and  $Y = X \times T$. Next, consider the map  $T \stackrel{\sigma}\lar Y$ given by the rule: $(t_1, t_2) \mapsto
(t_1, t_1, t_2)$; let  $\Theta := \mathsf{Im}(\sigma) = \Sigma_{12}$. Finally,  put  $\kB := \kA \boxtimes \kA^\circ \boxtimes \kA^\circ$.

\smallskip
\noindent
Consider  the following short  exact sequence of coherent sheaves on $Y$
$$
0 \lar \kO_Y \lar \kO_Y(\Theta) \xrightarrow{\res_\Theta} \sigma_*\bigl(\mbox{\textit{Hom}}_T(\pi^*\Omega_C, \kO_T)\bigr) \lar 0,
$$
which is an analogue of the residue sequence (\ref{E:Residue}).
The morphism  $\res_\Theta$ is given by the following rule: if $f = f(x, z_1, z_2) = \dfrac{g(x, z_1, z_2)}{x-z_1}$ is a local section of $\kO_Y(\Sigma)$ and $\alpha = \psi(z_1, z_2) dz_1$ is a local section of $\pi^*\Omega_C$ (in some local coordinates) then
$$
\bigl(\res_{\Sigma}(f)\bigr)(\alpha) = \bigl(T \ni (t_1, t_2) \mapsto \psi(t_1, t_2) g(t_1, t_1, t_2)\bigr).
$$
The nowhere vanishing differential form $\omega = \dfrac{1}{\nu(z)} dz \in \Gamma\bigl(C, \Omega_C\bigr)$ yields a trivialization  $\kO_C \stackrel{\omega}\lar \Omega_C$. We get  an induced short exact sequence of coherent sheaves
\begin{equation}\label{E:res}
0 \lar \kO_Y \lar \kO_Y(\Theta) \stackrel{\res_\Theta^\omega}\lar \sigma_*\bigl(\kO_T\bigr) \lar 0,
\end{equation}
where $\res_\Theta^\omega(f) = \Bigl(T \ni (t_1, t_2) \mapsto \dfrac{g(t_1, t_1, t_2)}{\tau(t_1)}\Bigr)$ for a local function $f$ as above.
Arguing as  in the proof of Proposition \ref{P:BasicsGeomdata}, we get a short exact sequence
$$
0 \lar \kB(\Sigma - \Theta) \lar \kB(\Sigma) \stackrel{\res_\Theta^\omega}\lar \kB(\Sigma - \Theta\bigr)\big|_{\Theta} \lar 0,
$$
Note that $\kB(\Sigma - \Theta\bigr)\big|_{\Theta} \cong \overline\kB(2 \Delta)$, where $\overline\kB = (\kA^\circ \otimes \kA^\circ) \boxtimes \kA^\circ$ and $\Delta \subset T$ is the diagonal. Consider the following commutative diagram:
\begin{equation*}
\xymatrix{
& \Gamma\bigl(T, \overline\kB(2 \Delta)\bigr) \ar@{^{(}->}[r]&   \Gamma\bigl(T \setminus \Delta, \overline\kB\bigr) \ar[r]^-\cong & \bigl(A \otimes_O A \otimes_\kk A\bigr)\otimes_\kk S\\
\Gamma\bigl(X \times T, \kB(\Sigma)\bigr) \ar[ru]^-{\res_{\Theta}^\omega} \ar[rd]_-{\mathsf{restr}} & &  & & \\
& \Gamma\bigl(C \times T, \kB(\Sigma)\bigr) \ar[uu]_-{\underline{\res}_{\Theta}^\omega} \ar@{^{(}->}[r]& \Gamma\bigl(C \times T\setminus \Sigma , \kB \bigr) \ar[r]^-\cong &
 \bigl(A \otimes_\kk A \otimes_\kk A\bigr)\otimes_\kk R,
}
\end{equation*}
where $R := \Gamma\bigl(C \times C \times C\setminus \Sigma, \kO_{C \times C \times C}\bigr)$ and $S := \Gamma\bigl(C \times C \setminus \Delta, \kO_{C \times C}\bigr)$. Our goal is to show that $\res_{\Theta}^\omega\bigl(\llbracket \varrho, \varrho\rrbracket\bigr) = 0$. It is equivalent to the statement that $\underline{\res}_{\Theta}^\omega\bigl(\llbracket \rho, \rho\rrbracket\bigr) = 0$.

\smallskip
\noindent
Recall that $\rho \in \Gamma\bigl(C \times C \setminus \Delta, \kA^\circ \boxtimes \kA^\circ\bigr)$ admits a presentation (\ref{E:structureRmat}). Then we get: $\llbracket \rho, \rho\rrbracket =$
$$
\;\left[
\frac{\nu(x_2)}{x_1 - x_2} \widetilde{\gamma}^{12} + \upsilon^{12}(x_1, x_2), \frac{\nu(x_3)}{x_1 - x_3} \widetilde{\gamma}^{13} + \upsilon^{13}(x_1, x_3)
\right]+
$$
$$
\;\left[
\frac{\nu(x_2)}{x_1 - x_2} \widetilde{\gamma}^{12} + \upsilon^{12}(x_1, x_2), \frac{\nu(x_3)}{x_2 - x_3} \widetilde{\gamma}^{23} + \upsilon^{23}(x_2, x_3)
\right]+
$$
$$
\left[
\frac{\nu(x_2)}{x_3 - x_2} \widetilde{\gamma}^{32} + \upsilon^{32}(x_3, x_2), \frac{\nu(x_3)}{x_1 - x_3} \widetilde{\gamma}^{13} + \upsilon^{13}(x_1, x_3)
\right],
$$
where the entire expression is viewed as an element of $\bigl(A \otimes_\kk A \otimes_\kk A\bigr)\otimes_\kk R$. Then  $\underline{\res}_{\Theta}^\omega\bigl(\llbracket \rho, \rho\rrbracket\bigr)$ viewed as an element of $\bigl((A \otimes_O A) \otimes_\kk A\bigr)\otimes_\kk R,$
is given by the formula:
$$\underline{\res}_{\Theta}^\omega\bigl(\llbracket \rho, \rho\rrbracket\bigr) =
\nu(z_2) \Bigl[\gamma^{12}, \overline\gamma^{13} + \overline\gamma^{23}\Bigr] + \Bigl[\gamma^{12}, \overline{\upsilon}^{13}(z_1, z_2) + \overline{\upsilon}^{23}(z_1, z_2)\Bigr].
$$
Here we apply the canonical morphism $A \otimes_\kk A \otimes_\kk A \lar (A \otimes_O A) \otimes_\kk A$ to define all the terms. Now, the equality
(\ref{E:Casimir}) implies that both summands of the above expression for  $\underline{\res}_{\Theta}^\omega\bigl(\llbracket \rho, \rho\rrbracket\bigr)$ are actually zero.

In a similar way, one can prove that the section $\llbracket \varrho, \varrho\rrbracket$ has no poles along the  other two divisors $\Sigma_{13}$ and $\Sigma_{23}$, which finishes the proof of the identity (\ref{E:geomGCYBE}).
\end{proof}

\begin{remark}
The idea to construct solutions of CYBE from an algebro--geometric data was suggested for the first time by Cherednik \cite{Cherednik}, who  stated
a version of  Theorem \ref{T:GCYBE} in the case of smooth curves and briefly outlined  an idea of the proof.
\end{remark}

\begin{definition} Let $x \ne y \in C$. Then we have  two canonical maps of vector spaces
\begin{itemize}
\item  the \emph{residue isomorphism} $\Gamma\bigl(X, \kA(x)\bigr) \stackrel{\res_x^\omega}\lar \kA\big|_x$
\item the \emph{evaluation map} $\Gamma\bigl(X, \kA(x)\bigr) \stackrel{\ev_y}\lar \kA\big|_y$
\end{itemize}
defined as follows. Tensoring  (\ref{E:ResSequence}) with $\kA$, we get a short exact sequence
\begin{equation}\label{E:sheafresidA}
0 \lar \kA \lar \kA(x) \stackrel{\res_x^\omega}\lar \kA\big|_x \lar 0.
\end{equation}
Since $H^0(X, \kA) = 0 =  H^1(X, \kA)$, applying to (\ref{E:sheafresidA}) the functor of global sections $\Gamma\bigl(X, \,-\,\bigr)$, we get the residue isomorphism $\res_x^{\omega}$ (we use  the same notation for a morphism of sheaves and the map of global sections). The
map $\ev_y$ is defined via the commutative diagram
$$
\xymatrix{
& \Gamma\bigl(X, \kA(x)\bigr) \ar[rd]^-{\ev'_y} \ar[ld]_{\ev_y}     &  \\
\kA\big|_y \ar@{^{(}->}[rr]^{\cong} & & \kA(x)\big|_y
}
$$
where $\ev'_y$ assigns to a global section of $\kA(x)$ its value in the fiber over $y$.
\end{definition}

\begin{lemma}\label{L:rmatrandmaps} Let $\bigl((X, \kA), (C, \omega)\bigr)$ be a geometric datum as in the beginning of this subsection.
For any $x_1 \ne x_2 \in C$, consider the linear map $\kA\big|_{x_2} \xrightarrow{\rho^{\sharp}(x_1, x_2)} \kA\big|_{x_1}$ making the following diagram  of vector spaces
\begin{equation}\label{E:rmatrixandmaps}
\begin{array}{c}
\xymatrix{
 & \Gamma\bigl(X, \kA(x_2)\bigr) \ar[ld]_-{\res_{x_2}^\omega} \ar[rd]^-{\ev_{x_1}}& \\
 \kA\big|_{x_2} \ar[rr]^-{\rho^{\sharp}(x_1, x_2)} & & \kA\big|_{x_1}
}
\end{array}
\end{equation}
commutative.
Then the  following results are true.
\begin{itemize}
\item
The  tensor $\rho(x_1, x_2) \in \kA\big|_{x_1} \otimes \kA\big|_{x_2}$ (which is the value of the geometric r--matrix $\rho$  at the point $(x_1, x_2)$) is the image of the linear map $\rho^{\sharp}(x_1, x_2)$ under the isomorphism
$\Hom_{\kk}\bigl(\kA\big|_{x_2}, \kA\big|_{x_1}\bigr) \lar \kA\big|_{x_1} \otimes \kA\big|_{x_2}$ induced by the Killing form on $\kA\big|_{x_2}$.
\item There exists a non--empty open subset $C^\circ \subset C$ such that for all  $x_1 \ne x_2 \in C^\circ$, the linear map $\rho^{\sharp}(x_1, x_2)$ is an isomorphism. In particular, the geometric $r$--matrix $\rho$ given by (\ref{E:geomGCYBE}) leads to  \emph{non--degenerate} solution of the generalized classical Yang--Baxter equation (\ref{E:GCYBE}).
\end{itemize}
\end{lemma}

\begin{proof}
The first result is a consequence of the functoriality of the residue sequence (\ref{P:ResidueSeqInduced}); see Lemma \ref{L:PropertiesResSeq}. To prove the second statement, let $x_1 \ne x_2 \in C$ be arbitrary points. Consider the short exact sequence
\begin{equation}\label{E:evaluation}
0 \lar \kA(x_2 -x_1) \lar \kA(x_2) \lar \kA\big|_{x_1} \lar 0
\end{equation}
 as well as the  coherent sheaf $\kC : = p^*\bigl(\kA(-x_1)\bigr) \otimes \kO_{X \times C}(\Delta)$ on $X \times C$, where $X \times C \stackrel{p}\lar X$ is the canonical projection. Obviously, $\kC$ is flat over $C$ and $\kB\big|_{X \times \{x\}} \cong \kA\bigl(x - x_1)$ for all $x \in C$. Therefore,
$
\chi\Bigl(\kB\big|_{X \times \{x\}}\Bigr) = \chi\Bigl(\kB\big|_{X \times \{x_1\}}\Bigr) = \chi(\kA) = 0
$
for all $x \in C$. Since $\Gamma\bigl(X, \kB\big|_{X \times \{x_1\}}\bigr) = \Gamma(X, \kA) = 0$, by semi--continuity there exists
an open subset $C^\circ \subseteq C$ containing $x_1$ such that $\Gamma\bigl(X, \kB\big|_{X \times \{x\}}\bigr)  = 0$ for all $x \in C^\circ$. Applying to (\ref{E:evaluation}) the functor of global sections $\Gamma(X, \,-\,)$, we conclude that the linear map $\ev_{x_1}$ is an isomorphism, which implies  the statement.
\end{proof}

\begin{remark}
Assume that a geometric datum $\bigl((X, \kA), (C, \omega)\bigr)$ is such that there exists  an $O$--linear isomorphism of Lie algebras $\Gamma(C, \kA) \stackrel{\xi}\lar \lieg \otimes_\kk O$. This trivialization $\xi$ allows to present the geometric $r$--matrix $\rho$  as  a meromorphic  function
$r = \rho^\xi: C \times C\setminus\Delta \lar \lieg \otimes \lieg$, which is a non--degenerate solution of the generalized classical Yang--Baxter equation (\ref{E:GCYBE}). This tensor--valued function can be written in the form:
$$
r(x_1, x_2) = \frac{\nu(x_2)}{x_1-x_2} \widetilde{\gamma} + \upsilon(x_1, x_2),
$$
where $C \times C \stackrel{\upsilon}\lar \lieg \otimes \lieg$ is a regular function.
\begin{itemize}
\item Let us take any other  $O$--linear isomorphism of Lie algebras $\Gamma(C, \kA^\circ) \stackrel{\zeta}\lar \lieg \otimes_\kk O$. Then we get an automorphism  $\phi$ of $\lieg \otimes_\kk O$ defined via the following commutative diagram of Lie algebras:
    $$
    \xymatrix{
    & \Gamma(C, \kA^\circ) \ar[ld]_-\xi \ar[rd]^-\zeta & \\
 \lieg \otimes_\kk O \ar[rr]^-\phi & &    \lieg \otimes_\kk O
    }
    $$
This automorphism $\phi$ establishes  a gauge equivalence between the solutions $\rho^\xi$ and $\rho^\zeta$ of (GCYBE).
\item If we replace the differential form $\omega$ by any multiple $\widetilde{\omega} = \dfrac{1}{\alpha} \omega$, where $\alpha$ is some meromorphic function on $C$, then the corresponding solution $r(x_1, x_2)$ will be transformed into the equivalent solution $\alpha(x_2) r(x_1, x_2)$.
\end{itemize}
\end{remark}

\begin{lemma}
Let $\bigl((X, \kA), (C, \omega)\bigr)$ be a geometric datum as at the beginning of this section and $Z$ be the irreducible component of $X$ containing $C$. Then the \emph{geometric genus} of $Z$ is at most one.
\end{lemma}
\begin{proof} We first show the following

\smallskip
\noindent
Claim. Assume that we have a smooth connected projective curve $Y$ and a coherent sheaf of Lie algebras $\kB$ on $Y$ such that
\begin{itemize}
\item $\kB$ is locally free and $H^1(Y, \kB) = 0$;
\item there exists a non--empty open subset $Y^\circ \subset Y$ such that $\kB\big|_y$ is a semi--simple Lie algebra for all $y \in Y^\circ$.
\end{itemize}
Then the genus of $Y$ is at most one.

\smallskip
\noindent
\emph{Proof of the claim}.
Indeed, since $\kB$ is locally free, for any open subset $V \subseteq Y$, we have a well--defined Killing--form $\Gamma(V, \kB) \times \Gamma(V, \kB) \stackrel{\kappa}\lar
\Gamma(V, \kO_Y)$. It defines a morphism of  sheaves $\kB \stackrel{\bar\kappa}\lar \kB^\vee$, which is moreover an isomorphism when restricted to $Y^\circ$. Since $\kB$ and $\kB^\vee$ have  the same rank, we conclude that the morphism $\bar\kappa$ is a monomorphism, whose  cokernel  is a torsion sheaf. From  $H^1(Y, \kB) = 0$ it follows that $H^1(Y, \kB^\vee) = 0$, too. By the Riemann--Roch formula, we get two inequalities:
$$
0 \le \chi(\kB) = \mathsf{deg}(\kB) + (1-g)\mathsf{rk}(\kB) \quad \mbox{and} \quad 0 \le \chi(\kB^\vee) = - \mathsf{deg}(\kB) + (1-g)\mathsf{rk}(\kB).
$$
Adding them, we conclude that $g  \le 1$, as claimed.

\smallskip
Now we are prepared to prove the statement of the lemma. Consider the torsion free sheaf of Lie algebras $\kC := \kA\big|_{Z}$ on the integral projective curve $Z$.
From the vanishing $H^1(X, \kA) = 0$ one can deduce that $H^1(Z, \kC) = 0$, too. Let $Y \stackrel{\pi}\lar Z$ be the normalization of $Z$ and
$\kB := \pi^*\kC/\mathsf{tor}(\pi^*\kC\bigr)$. Then the pair $(Y, \kB)$ satisfies the assumptions of the claim we started with, finishing the proof of the lemma.
\end{proof}

\section{On the skew--symmetry of the geometric r--matrix}

\noindent
In Section \ref{S:GCYBE}, we introduced a certain geometric datum $\bigl((X, \kA), (C, \omega)\bigr)$ and then assigned to it a section
$\rho \in \Gamma\bigl(C \times C \setminus \Delta, \kA^\circ \boxtimes \kA^\circ \bigr)$.

\begin{definition} The geometric $r$--matrix $\rho$ is called \emph{skew--symmetric}, if
$$
\sigma(\rho) = - \imath^*(\rho) \in \Gamma\bigl(C \times C \setminus \Delta, \kA^\circ \boxtimes \kA^\circ \bigr),
$$
where $\imath$ is the automorphism of the surface $C \times C$ permuting both components and $\sigma$ is an automorphism of the coherent sheaf $\kA^\circ \boxtimes \kA^\circ$ permuting  both factors on the level of appropriate local sections.
\end{definition}

\noindent
Note that after passing to a local  trivialization of $\kA$, we obtain the conventional notion
of skew--symmetry given by (\ref{E:skewsymmetry}).
The goal of  this section is to  answer the following natural

\smallskip
\noindent
\textbf{Question}.
What are  additional assumptions on the geometric datum $\bigl((X, \kA), (C, \omega)\bigr)$, which  insure that  the corresponding geometric $r$--matrix $\rho$ is skew--symmetric?

\subsection{Criterion of skew--symmetricity}
Let $x \ne y \in C$ and $U = C \setminus \{x, y\}$. Consider the following $\kk$--bilinear pairing
$$
\varpi: \;
\Gamma\bigl(X, \kA(x)\bigr) \times \Gamma\bigl(X, \kA(y)\bigr) \stackrel{\mathsf{can}}\lar \Gamma(U, \kA) \times \Gamma(U, \kA) \stackrel{\kappa}\lar  \Gamma(U, \kO_C)
$$

\begin{proposition}\label{P:skewsymm} We have: $\rho^{12}(x, y) = - \rho^{21}(y, x)$ if and only if
\begin{equation}\label{E:critskewsymm}
\res_x\bigl(\varpi(a, b) \omega\bigr) + \res_y\bigl(\varpi(a, b)\omega\bigr) = 0 \; \mbox{\rm for all} \; (a, b) \in \Gamma\bigl(X, \kA(x)\bigr) \times \Gamma\bigl(X, \kA(y)\bigr).
\end{equation}
\end{proposition}
\begin{proof} It follows from Lemma \ref{L:rmatrandmaps} that the equality $\rho^{12}(x, y) + \rho^{21}(y, x) = 0$ can be rewritten in the following form:
\begin{equation}\label{E:unitarity1}
\kappa_y\bigl(\bigl(\rho^{\sharp}(x, y)\bigr)(\alpha), \beta\bigr) + \kappa_x\bigl(\alpha, \bigl(\rho^{\sharp}(y, x)\bigr)(\beta)\bigr) = 0 \; \mbox{\rm for all} \;
(\alpha, \beta) \in \kA\big|_x \times \kA\big|_y,
\end{equation}
where $\kappa_x$ (respectively,  $\kappa_y$)  is  the Killing forms of the Lie algebra $\kA\big|_x$ (respectively, $\kA\big|_y$).  Given $(\alpha, \beta) \in \kA\big|_x \times \kA\big|_y$, let $(a, b) \in \Gamma\bigl(X, \kA(x)\bigr) \times \Gamma\bigl(X, \kA(y)\bigr)$ be the element, uniquely determined by the properties  $\alpha = \res^\omega_x(a)$ and $\beta = \res^\omega_y(b)$. Taking into account that the Killing form $\kappa_y$ is symmetric, we can rewrite (\ref{E:unitarity1}) as follows:
\begin{equation}\label{E:unitarity2}
 \kappa_x\bigl(\res^\omega_x(a), \ev_x(b)\bigr) +  \kappa_y\bigl(\res^\omega_y(b), \ev_y(b)\bigr) = 0 \; \mbox{\rm for all} \;
(a, b) \in \Gamma\bigl(X, \kA(x)\bigr) \times \Gamma\bigl(X, \kA(y)\bigr).
\end{equation}
Let $U \stackrel{\jmath}\lar X$ be the canonical embedding. Then the following diagram
$$
\xymatrix{
\kA(x) \otimes \kA(y) \ar[rr] \ar[d]_-{\res_x^\omega \otimes \ev_x} & & \jmath_*\bigl(\kA\big|_U) \otimes \jmath_*\bigl(\kA\big|_U) \ar[d] \\
\kA\big|_x \otimes \kA\big|_x \ar[d]_-{\widetilde{\kappa}_x} & & \jmath_*\bigl(\kA\big|_U  \otimes \kA\big|_U\bigr) \ar[d]^-{\widetilde{\kappa}_U} \\
\CC_x & & \jmath_*\bigl(\kO_U\bigr) \ar[ll]_-{\res_x^\omega}
}
$$
of quasi--coherent sheaves on $X$ is commutative. From this diagram we conclude that $\kappa_x\bigl(\res^\omega_x(a), \ev_x(b)\bigr) = \res_x\bigl( \varpi(a, b) \omega\bigr)$, implying the statement.
\end{proof}

\subsection{Geometric r--matrices  arising from Calabi--Yau curves}
Belavin--Drinfeld trichotomy result \cite[Theorem 1.1]{BelavinDrinfeld} on the classification of non--degenerate skew--symmetric  solutions of (\ref{E:CYBE}) suggests to restrict our attention on those geometric
data $\bigl((X, \kA), (C, \omega)\bigr)$, for which  $X$ is a  Calabi--Yau curve and $C$ is a smooth part of an irreducible component of $X$. This case is ``deterministic'' since  there exists a unique (up to a scalar) natural choice of the differential form $\omega$:  a generator of the vector space
$\Gamma(X, \Omega_X) \cong \kk$.

\smallskip
Let $E = E_{g_2, g_3} := \overline{V\bigl(u^2 - 4v^3 + g_2 v + g_3\bigr)} \subset \PP^2$ be a Weierstra\ss{} cubic curve, where $g_2, g_3 \in \kk$. Weierstra\ss{} cubics are precisely the integral Calabi--Yau curves.
The curve $E_{g_2, g_3}$ is singular  if and only
if $g_2^3  =  27 g_3^2$. In this case, $E$ has a unique singular point $s = (\lambda, 0) = (\lambda: 0: 1)$ with
 \begin{equation*}
 \lambda =
 \left\{
 \begin{array}{cl}
  \dfrac{3g_3}{2g_2} & s \mbox{\rm \; is a node}, \\
  0 & s \mbox{\rm \; is a cusp}.
 \end{array}
 \right.
 \end{equation*}
We  may put:
\begin{equation}\label{E:regular}
\omega := \frac{du}{12 v^2 - g_2} = \frac{dv}{2 u}.
\end{equation}
 In the case $E$ is smooth, the differential form $\omega$ is nowhere vanishing. If $E$ is singular, $\omega$ is a regular
  Rosenlicht one--form \cite[Section II.6]{BarthHulekPetersVen}.
Let $\PP^1 \stackrel{\pi}\lar  E$ be the normalization map.
 Denote by   $\Omega^m_{\PP^1}$ and $\Omega^m_{E}$
the sheaves  of meromorphic differential  1--forms on $\PP^1$ and $E$
respectively; note  that
$\Omega^m_E = \pi_*\bigl(\Omega^m_{\PP^1}).
$
 Then the canonical sheaf  of $E$  can be realized as   the subsheaf $\Omega_E$ of $\Omega^m_E$ such that for
any open subset $U \subseteq E$ one has
$$
\Omega_E(U) = \left\{\alpha   \in \Omega^{m}_{\PP^1}(V)\left|
\mbox{for all\;} p \in U\; \mbox{and}\;  f \in \kO_{E}(U):\;
\sum\limits_{q:\; \pi(q) = p} \res_{q}\bigl((f\circ \pi)\alpha\bigr) = 0
\right.\right\},
$$
where $V = \pi^{-1}(U) \subseteq \PP^1$. Then the  differential form $\omega$, given by (\ref{E:regular}),  generates  $\Gamma(E, \Omega_E)$.

\smallskip
\noindent
Alternatively,  we can choose homogeneous coordinates $(z_0: z_1)$  on $\PP^1$ such that
\begin{equation}\label{E:PreimageSing}
\pi^{-1}(s) =
 \left\{
 \begin{array}{cl}
  \bigl\{0, \infty\bigr\} & E \mbox{\rm \; is nodal}, \\
  \bigl\{\infty\bigr\}  &   E \mbox{\rm \; is cuspidal},
 \end{array}
 \right.
\end{equation}
where $0 = (1: 0)$ and $\infty = (0: 1)$. Then  $z = \dfrac{z_1}{z_0}$ is a local coordinate in a neighbourhood of $0$ and we can define a generator of
$\Gamma(E, \Omega_E)$ by the formula:
\begin{equation}\label{E:DiffForm}
\omega =
 \left\{
 \begin{array}{cl}
  \dfrac{dz}{z}  & E \mbox{\rm \; is nodal}, \\
   dz   &   E \mbox{\rm \; is cuspidal}.
 \end{array}
 \right.
\end{equation}
\begin{theorem}\label{T:CYBEtorsfree} Let $E$ be a Weierstra\ss{} cubic, $\omega$ a generator of $\Gamma(E, \Omega_E)$  and  $\kA$ be a coherent sheaf of Lie algebras on $E$ satisfying the following conditions:
\begin{enumerate}
\item $H^0(E, \kA) = 0 = H^1(E, \kA)$.
\item $\kA$ is weakly $\lieg$--locally free over $\breve{E}$.
\end{enumerate}
If $E$ is singular, we impose a third condition on the germ $A_s$ of the sheaf $\kA$ at the singular point $s$.
Consider the $\CC$--bilinear pairing
\begin{equation}\label{E:canpairing}
\kappa^\omega: \quad A_K \times A_K \stackrel{\kappa}\lar K \stackrel{\res_s^\omega}\lar \CC,
\end{equation}
where $\kappa$ is the Killing form of $A_K$ and
$$\res_s^\omega(f) = \sum\limits_{q \in \PP^1: \pi(q) = s} \res_q(f\omega)$$ for $f \in K$. We require that
\begin{enumerate}
  \setcounter{enumi}{2}
  \item  $A_s$ is a coisotropic Lie subalgebra of $A_K$ with respect to the pairing $\kappa^\omega$, i.e.
  $$\kappa^\omega(f, g) = 0 \;\mbox{\rm for all}\; f, g \in A_s.$$
\end{enumerate}
Then the geometric $r$--matrix $\rho \in \Gamma\bigl(\breve{E} \times \breve{E}\setminus \Delta, \kA \boxtimes \kA\bigr)$, corresponding to the datum $\bigl((E, \kA), (\breve{E}, \omega)\bigr)$, is a skew--symmetric solution of the classical Yang--Baxter equation
$$
\bigl[\rho^{12}, \rho^{13}\bigr] + \bigl[\rho^{12}, \rho^{23}\bigr] + \bigl[\rho^{13}, \rho^{23}\bigr] = 0.
$$
\end{theorem}

\begin{remark}
Note that the condition (3) is automatically satisfied provided the sheaf $\kA$ is locally free.
\end{remark}

\begin{proof}
According to Theorem \ref{T:GCYBE}, the section $\rho$ satisfies GCYBE. Hence, it is sufficient to prove skew--symmetry of $\rho$. Let $x \ne y \in \breve{E}$ and $U := E \setminus \{x, y\}$. Consider the following canonical map
$$
\varpi: \Gamma\bigl(E, \kA(x)\bigr) \times \Gamma\bigl(E, \kA(y)\bigr) \lar \Gamma(U, \kA)  \times \Gamma(U, \kA) \lar A_K \times A_K
\stackrel{\kappa}\lar K.
$$
Then for any $(f, g) \in \Gamma\bigl(E, \kA(x)\bigr) \times \Gamma\bigl(E, \kA(y)\bigr)$ we have: $\varpi(f, g) \in K \cap \Gamma\bigl(\breve{E} \setminus \{x, y\}, \kO_E\bigr)$. By  the  residue theorem, we have the identity:
$$
0 = \sum\limits_{q \in E} \res_q\bigl(\varpi(f, g)\omega\bigr) = \res_s\bigl(\varpi(f, g)\omega\bigr) + \res_x\bigl(\varpi(f, g)\omega\bigr) + \res_y\bigl(\varpi(f, g)\omega\bigr),
$$
since $\res_q\bigl(\varpi(f, g)\omega\bigr) = 0$ for $q \notin \bigl\{s, x, y\bigr\}$. However, $\res_s\bigl(\varpi(f, g)\omega\bigr) = 0$ since
$A_s \subset A_K$ is coisotropic. As a consequence:
$$
\res_x\bigl(\varpi(f, g)\omega\bigr) + \res_y\bigl(\varpi(f, g)\omega\bigr) = 0 \; \mbox{for all} \; (f, g) \in \Gamma\bigl(E, \kA(x)\bigr) \times \Gamma\bigl(E, \kA(y)\bigr),
$$
what is equivalent to the skew--symmetry $\rho^{12}(x, y) = - \rho^{21}(y, x)$ due to Proposition \ref{P:skewsymm}.
\end{proof}

\begin{remark}\label{R:Yangian}
Let $E$ be a singular Weierstra\ss{} curve, $\PP^1 \stackrel{\nu}\lar E$ its
 normalization, $q \in \PP^1$ such that $\nu(q) = s$ and $\kI$ be the ideal sheaf of $q$.
Consider the sheaf of Lie algebras
$$
\kA := \nu_\ast
\left(\mbox{\it sl}\left(
\begin{array}{cccc}
\kI & \kI & \dots & \kI\\
\kI & \kI & \dots & \kI \\
\vdots & \vdots & \ddots & \vdots \\
\kI & \kI & \dots & \kI\\
\end{array}
\right)
\right).
$$
Obviously, we have: $H^0(E, \kA) = 0 = H^1(E, \kA)$ and $\Gamma(\breve{E}, \kA) \cong \lieg \otimes
\Gamma\bigl(\breve{E}, \kO_E\bigr)$.
Moreover,
\begin{itemize}
\item If $E$ is cuspidal, then $A_s$ is a coisotropic Lie subalgebra of $A_K$ and the corresponding solution of (\ref{E:CYBE}) is the solution of Yang: $r(x, y) = \dfrac{1}{y-x} \gamma$.
\item If $E$ is nodal then $A_s$ is not coisotropic in  $A_K$. The corresponding geometric $r$--matrix is $r(x, y) = \dfrac{x}{y-x} \gamma$. It satisfies the generalized classical Yang--Baxter equation
    (\ref{E:GCYBE}) but does not satisfy (\ref{E:CYBE}).
\end{itemize}
\end{remark}

\begin{theorem}\label{T:CYBEfromVB}
Let $X$ be a  Calabi--Yau curve and $\kA$ a coherent sheaf of Lie algebras on $X$ such that
\begin{enumerate}
\item $H^0(X, \kA) = 0 = H^1(X, \kA)$.
\item The sheaf $\kA$ is weakly $\lieg$--locally free on $X$.
\end{enumerate}
Let $\omega$ be a generator of $\Gamma(X, \Omega_X)$ and $C = Z \cap \breve{X}$, where $Z$ is an irreducible component of $X$ and $\breve{X}$ is the regular part of $X$. The the geometric r--matrix $\rho \in \Gamma\bigl(C \times C\setminus \Delta, \kA \boxtimes \kA\bigr)$, attached to the datum
$\bigl((X, \kA), (C, \omega)\bigr)$, is a skew--symmetric solution
of \textrm{CYBE}.
\end{theorem}

\begin{proof}
  Here we basically give a replica of the proof of Theorem \ref{T:CYBEtorsfree}. We know that  $\rho$ satisfies GCYBE, thus it suffices to prove the skew--symmetry of $\rho$. Let $x \ne y \in C$ and $U := X \setminus \{x, y\}$. Consider the following canonical map
$$
\varpi: \Gamma\bigl(X, \kA(x)\bigr) \times \Gamma\bigl(X, \kA(y)\bigr) \lar \Gamma(U, \kA)  \times \Gamma(U, \kA) \lar A_K \times A_K
\stackrel{\kappa}\lar K.
$$
Then for any $(f, g) \in \Gamma\bigl(X, \kA(x)\bigr) \times \Gamma\bigl(X, \kA(y)\bigr)$ we have: $\varpi(f, g) \in  \Gamma\bigl(X \setminus \{x, y\}, \kO_X\bigr) \subset K$. The residue theorem implies that
$$
0 = \sum\limits_{q \in X} \res_q\bigl(\varpi(f, g)\omega\bigr) =  \res_x\bigl(\varpi(f, g)\omega\bigr) + \res_y\bigl(\varpi(f, g)\omega\bigr),
$$
since $\res_q\bigl(\varpi(f, g)\omega\bigr) = 0$ for  any smooth point $q \notin \bigl\{x, y\bigr\}$ as well as for any singular point of $X$ (at this place, the assumption that the sheaf $\kA$ is locally free, plays a crucial role). By Proposition \ref{P:skewsymm} we get:  $\rho^{12}(x, y) = - \rho^{21}(y, x)$.
\end{proof}

\begin{remark}\label{R:AdP}
Let $X$ be a Calabi--Yau curve and $\kF$ a simple locally free sheaf on $X$  (i.e.~$\End_X(\kF) = \kk$) of rank $n \ge 2$ and multi--degree $\mathbbm{d}$.  Consider the  sheaf  $\kA:= {\mbox{\textit{Ad}}}_X(\kF)$ of  traceless endomorphisms of $\kF$,
defined through  the  short exact sequence
\begin{equation}\label{E:shortExactAd}
0 \lar \kA \lar \mbox{\emph{End}}_X(\kF) \stackrel{\tr}{\lar} \kO_X \lar 0,
\end{equation}
where $\tr$ is the  trace map. Clearly, $\kA$ is $\mathfrak{sl}_n(\CC)$--locally free and $H^0(X, \kA) = 0 = H^1(X, \kA)$. Moreover, at least for elliptic curves and Kodaira fibers of types $\mathrm{I}_m \; (m \ge 1)$, $\mathrm{II}, \mathrm{III}$  and $\mathrm{IV}$ it is known that the  isomorphy class of $\kA$ is determined  by \emph{purely discrete} data: the rank $n$ and the multi--degree of  $\mathbbm{d}$ of  $\kF$; see \cite[Proposition 2.19]{BH} and references therein.
\end{remark}

\subsection{Minimal model of a geometric datum} The setting of Theorem \ref{T:CYBEfromVB} seems to be more general than in the case of Theorem
\ref{T:CYBEtorsfree}, because of the occurrence of non--integral Calabi--Yau curves. However,  the contrary happens to be true.
\begin{theorem}
Let $X$ be a singular Calabi--Yau curve and $Y$ an irreducible component of $X$. Then the following results are true.
\begin{itemize}
\item There exists a Weierstra\ss{} curve $E$ and a morphism $X \stackrel{\mu}\lar E$ (called \emph{minimal model} of $X$ relative to $Y$) collapsing all irreducible components of $X$ different to  $Y$ to the singular point of $E$ and mapping $Y$ surjectively on $E$.
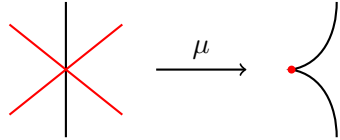
\begin{center}
\begin{tikzpicture}[scale=0.3]
    \draw [thick] (2,3)
    to [out=270,in=10] (0,0)
      to [out=-10,in=90]
	 ( 2,-3);
\node[point] at (0,0){};
   \draw[thick,->](-6,0) to node[above]{$\mu$} (- 2,0);
 \draw[thick] (-10, -3 ) to (-10, 3);
 \draw[red,thick] (-12.5, 2) to  (-7.5,-2);
 \draw[red, thick] (-7.5, 2) to  (-12.5,-2);
\end{tikzpicture}
\end{center}
    \item If $\omega$ is a generator of  $\Gamma(X, \Omega_X)$, then $\mu_*\bigl(\omega\big|_Y\bigr)$ is a generator of $\Gamma(E, \Omega_E)$.
\item The canonical adjunction morphism
$
\kO_E \stackrel{\zeta}\lar R\mu_*\bigl(\kO_X\bigr)
$
in the derived category $D^b\bigl(\Coh(E)\bigr)$ is an isomorphism.
\end{itemize}
\end{theorem}

\begin{proof} The list of Calabi--Yau curves is actually known, see e.g.~\cite[Section 3]{Smyth}. For the proof of the first statement, see \cite[Proposition 3.8]{Silverman}. The second result is a straightforward computation.

To prove the third statement, consider the right adjoint functor $D^b\bigl(\Coh(E)\bigr) \stackrel{\mu^{!}}\lar D^b\bigl(\Coh(X)\bigr)$ to the functor $R\mu_*$, which exists according   to the Grothendieck duality \cite{RD, Neeman}. We have another canonical adjunction morphism
$
R\mu_* \mu^{!}\bigl(\kO_E\bigr) \stackrel{\xi}\lar \kO_E,
$
which is compatible with restrictions on open sets in $E$. Since both curves $E$ and $X$ are Calabi--Yau, and $\mu^{!}(\Omega_E) \cong \Omega_X$, we get a morphism of sheaves
$
\mu_* \bigl(\kO_X\bigr) \stackrel{\bar\xi}\lar \kO_E,
$
which is an isomorphism when restricted to the open set $\breve{E}$. On the other hand, we have another  canonical morphism of sheaves
$\kO_E \stackrel{\bar\zeta}\lar \mu_*\bigl(\kO_X\bigr)$. It is clear that $\mu_*\bigl(\kO_X\bigr)$ is a sheaf of Cohen--Macaulay rings, finite as
a module over $\kO_E$. Therefore, $\mu_*\bigl(\kO_X\bigr)$ is torsion free viewed a $\kO_E$--module, thus $\bar\xi$ is a monomorphism. Since
$\Gamma\bigl(E, \mu_*\bigl(\kO_X\bigr)\bigr) \cong \kk$, we conclude that $\bar\xi$ is an isomorphism. But it implies  that $\bar\zeta$ is an isomorphism, too.

Since the fiber of the morphism $\mu$ over the point $s$  has dimension one, the cohomology sheaves $R^{i}\mu_*\bigl(\kO_X\bigr)$ are automatically zero  unless
$i = 0$ or $1$. We have proved that $R^{0}\mu_*\bigl(\kO_X\bigr) \cong \kO_E$.  The sheaf $\kT := R^{1}\mu_*\bigl(\kO_X\bigr)$ is automatically torsion.
 Assume that $\kT$ is non--zero. Then its   support must be  the singular point $s$. We have an exact triangle
$$
\kO_E \stackrel{\zeta}\lar R\mu_*\bigl(\kO_X\bigr) \lar \kT[-1] \lar \kO_E[1]
$$
in the derived category $D^b\bigl(\Coh(E)\bigr)$. Applying the derived functor of global sections $R\Gamma\bigl(E, \,-\,\bigr)$, we get an exact triangle
$$
R\Gamma\bigl(E, \kO_E\bigr) \xrightarrow{R\Gamma(\zeta)} R\Gamma\bigl(E, R\mu_*\bigl(\kO_X\bigr)\bigr) \lar R\Gamma(E,\kT)[-1] \lar R\Gamma\bigl(E, \kO_E\bigr)[1]
$$
in the derived category $D^b\bigl(\mathsf{Vect}(\kk)\bigr)$ of vector spaces over $\kk$. We have a canonical isomorphism of complexes $R\Gamma\Bigl(E, R\mu_*\bigl(\kO_X\bigr)\Bigr) \cong R\Gamma\bigl(X,  \kO_X\bigr)$ and $H^i\bigl(X, \kO_X\bigr) \cong \kk \cong  H^i\bigl(E, \kO_E\bigr)$ for $i = 0, 1$. It implies that the class
 of the complex $R\Gamma(E,\kT)$ in the Grothendieck group of $D^b\bigl(\mathsf{Vect}(\kk)\bigr)$ is zero. On the other hand,
$R^i \Gamma(E, \kT) = H^i(E, \kT) = 0$ for $i \ne 0$, whereas $\Gamma(E, \kT) \ne 0$, contradiction. Hence, the morphism $\zeta$ is an isomorphism, as claimed.
\end{proof}

\begin{proposition}\label{P:minimalmodelsheaves}
Let $X \stackrel{\mu}\lar E$ be the minimal model of a Calabi--Yau curve $X$ (with respect to some irreducible component). Then the following results are true.
\begin{itemize}
\item The functor $\Perf(E) \stackrel{L\mu^*}\lar D^b\bigl(\Coh(X)\bigr)$ is fully faithful.
\item Let $\kF$ be a coherent sheaf on $X$ such that $H^0(X, \kF) = 0 = H^1(X, \kF)$. Then $R^1 \mu_*(\kF) = 0$ and $H^0(X, \kG) = 0 = H^1(X, \kG)$ for
$\kG := \mu_*(\kF)$.
\end{itemize}
\end{proposition}
\begin{proof}
By the projection formula, for any complex $\kH^{\bul}$ from $\Perf(E)$ we have:
$$
R\mu_* \bigl(L\mu^*(\kH^{\bul})\bigr) \cong \kH^{\bul} \stackrel{\mathbb{L}}\otimes R\mu_* \bigl(\kO_X\bigr) \cong \kH^{\bul},
$$
which  implies the first statement. To prove the second result, let $\kT := R^1 \mu_*(\kF)$. Then $\kT$ is a torsion sheaf and we have an exact triangle
\begin{equation}\label{E:triangle}
\kT[-2] \lar \kG[0] \stackrel{\imath}\lar R\mu_*(\kF) \lar \kT[-1],
\end{equation}
where the morphism of sheaves $\kG \xrightarrow{\mathcal{H}^0(\imath)}  R^0\mu_*(\kF)$ is an isomorphism. Because of the vanishing $H^0(X, \kF) = 0 = H^1(X, \kF)$ we have: $R\Gamma\bigl(E, R\mu_*(\kF)\bigr) \cong 0$ in $D^b\bigl(\mathsf{Vect}(\kk)\bigr)$. Applying the functor $R\Gamma(E, \,-\,)$ to  the exact triangle (\ref{E:triangle}), we conclude that  $H^1(E, \kG) = 0$ and $H^0(E, \kT) = 0$. The last vanishing implies that $\kT \cong 0$.
\end{proof}

\begin{theorem}\label{T:CYBEandWeierModel}
Let $X$ be a Calabi--Yau curve, $Y$ an irreducible component of $X$ and $C = Y \cap \breve{X}$, where $\breve{X}$ is the regular part of $X$. Let  $X \stackrel{\mu}\lar E$ be the  minimal model of $X$ relative to $Y$ and $\widetilde\omega$ the restriction of a generator of $\Gamma(X, \Omega_X)$ on $C$.  Let $\kB$ be a sheaf of Lie algebras on $X$ such that
\begin{enumerate}
\item $H^0(X, \kB) = 0 = H^1(X, \kB)$.
\item $\kB$ is weakly $\lieg$--locally free on $X$.
\end{enumerate}
Let $\kA := \mu_*(\kB)$ and $\omega = \mu_*(\widetilde\omega)$. Then the geometric $r$--matrices, corresponding to the data $\bigl((X, \kB), (C, \widetilde\omega)\bigr)$ and
$\bigl((E, \kA), (\breve{E}, \omega)\bigr)$, lead to the same solution of (\ref{E:CYBE}).
\end{theorem}

\begin{proof}
Let $\tilde{x} \ne  \tilde{y}\in C$ are such that $H^i\bigl(X, \kB(\tilde{x} - \tilde{y})\bigr) = 0 = H^i\bigl(X, \kO_X(\tilde{x} - \tilde{y})\bigr)$ for $i = 0, 1$; let
  $x := \mu(\tilde{x})$ and $y := \mu(\tilde{y})$. Then the following diagram  of coherent sheaves on $E$
$$
\xymatrix{
0 \ar[r] & \kO_E \ar[r] \ar[d]_{\cong} & \kO_E(x) \ar[r]^-{\res_x^\omega} \ar[d]^-{\cong} & \CC_x  \ar[r] \ar[d]^-{\cong} &  0 \\
0 \ar[r] & \mu_*\bigl(\kO_X\bigr) \ar[r]  & \mu_*\bigl(\kO_X(\widetilde{x})\bigr) \ar[r]^-{\mu_*(\res_{\widetilde{x}}^{\widetilde\omega})}  & \mu_*\bigl(\CC_{\tilde{x}}\bigr)  \ar[r]  &  0 \\
}
$$
is commutative, where all vertical maps are the canonical ones (the exactness of the lower sequence follows from the fact that $R^1 \mu_*(\kO_X) = 0$). Analogously, because of the vanishings $R^1 \mu_*(\kB) = 0 = R^1 \mu_*\bigl(\kB(\tilde{x} - \tilde{y})\bigr)$, we get   commutative diagrams
\begin{equation}\label{E:DiagResid}
\begin{tabular}{c}
\xymatrix{
0 \ar[r] & \kA \ar[r] \ar[d]_{\cong} & \kA(x) \ar[r]^-{\res_x^\omega} \ar[d]^-{\cong} & \kA\big|_x  \ar[r] \ar[d]^-{\cong} &  0 \\
0 \ar[r] & \mu_*(\kB) \ar[r]  & \mu_*\bigl(\kB(\tilde{x})\bigr) \ar[r]^-{\mu_*(\res_{\widetilde{x}}^{\widetilde\omega})}  & \mu_*\bigl(\kB\big|_{\tilde{x}}\bigr)  \ar[r]  &  0 \\
}
\end{tabular}
\end{equation}
and
\begin{equation}\label{E:DiagEv}
\begin{tabular}{c}
\xymatrix{
0 \ar[r] & \kA(x-y) \ar[r] \ar[d]_{\cong} & \kA(x) \ar[r]^-{\ev_y} \ar[d]^-{\cong} & \kA\big|_y  \ar[r] \ar[d]^-{\cong} &  0 \\
0 \ar[r] & \mu_*\bigl(\kB(\tilde{x}-\tilde{y})\bigr) \ar[r]  & \mu_*\bigl(\kB(\tilde{x})\bigr) \ar[r]^-{\mu_*(\ev_{\tilde{y}})}  & \mu_*\bigl(\kB\big|_{\tilde{y}}\bigr)  \ar[r]  &  0.
}
\end{tabular}
\end{equation}
Applying the functor of global sections $\Gamma(E, \,-\,)$ to both diagrams (\ref{E:DiagResid}) and (\ref{E:DiagEv}), we get a commutative diagram of vector spaces
$$
\xymatrix{
\kA\big|_x \ar[d]_-\cong  & &  \ar[ll]_-{\res_x^\omega} \Gamma\bigl(E, \kA(x)\bigr) \ar[rr]^-{\ev_y} \ar[d]^-\cong  & & \kA\big|_y \ar[d]^-\cong \\
\kB\big|_{\tilde{x}}  & &  \ar[ll]_-{\res_{\tilde{x}}^{\widetilde\omega}} \Gamma\bigl(X, \kB(\tilde{x})\bigr) \ar[rr]^-{\ev_{\tilde{y}}}  & & \kB\big|_{\tilde{y}}
}
$$
where all vertical isomorphisms are the canonical ones. The statement of the theorem follows now from Lemma \ref{L:rmatrandmaps}.
\end{proof}

\begin{remark}\label{E:CYnonreduced}
So far, a Calabi--Yau curve by definition was assumed to be \emph{reduced}. However, this assumption can be weakened. Namely,   it would be
 sufficient to consider
a Gorenstein projective curve $X$ with trivial dualizing sheaf, which has at least one reduced component $Y$ (e.g.~the  plane cubic $PV(uv^2) \subset \PP^2_\kk$).
If $\kF$ is a simple locally free sheaf on $X$ then the sheaf $\kB:= {\mbox{\textit{Ad}}}_X(\kF)$ of  traceless endomorphisms of $\kF$ has vanishing cohomology.
Let $X \stackrel{\mu}\lar E$ be the minimal model of $X$ relative to $Y$
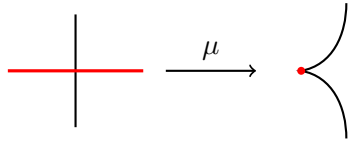
\begin{center}
\begin{tikzpicture}[scale=0.3]
    \draw [thick] (2,3)
    to [out=270,in=10] (0,0)
      to [out=-10,in=90]
	 ( 2,-3);
\node[point] at (0,0){};
   \draw[thick,->](-6,0) to node[above]{$\mu$} (- 2,0);
 \draw[thick] (-10, -2.5 ) to (-10, 2.5);
 \draw[red,  very thick] (-7, 0) to  (-13,-0);
\end{tikzpicture}
\end{center}
and $\omega$ be a generator of the vector space $\Gamma\bigl(E, \Omega_E\bigr)$.
Let $\widetilde{\omega}$ be a meromorphic 1--form on $C = Y \cap \breve{X}$ such that $\omega = \mu_*(\widetilde\omega)$ and  $\kA := \mu_*(\kB)$.
Then $\bigl((X, \kB), (C, \widetilde\omega)\bigr)$ is a well--defined geometric datum in the sense of Subsection \ref{SS:GeometrizationGCYBE}.
Moreover, $\kO_E \cong \mu_*\bigl(\kO_X\bigr)$  and $H^i(E, \kA) = 0$ for $i = 0, 1$, since the proof of Proposition \ref{P:minimalmodelsheaves} does not use the assumption that $X$ is reduced.
By Theorem \ref{T:CYBEandWeierModel},
the $r$--matrices, corresponding to the geometric data  $\bigl((X, \kB), (C, \widetilde\omega)\bigr)$ and
$\bigl((E, \kA), (\breve{E}, \omega)\bigr)$, are equivalent. Moreover, the obtained solution of GCYBE is also skew--symmetric. Indeed, for any open subset  $U \subset E$,  the following  diagram of vector spaces
$$
\xymatrix{
\Gamma(U, \kA) \times \Gamma(U, \kA) \ar@{^{(}->}[r] \ar[dd]_-{\cong} & A_K \times A_K  \ar[r]^-{\kappa}  & K \\
                                                     & & \Gamma\bigl(U, \kO_E\bigr) \ar[d]^-{\cong} \ar@{_{(}->}[u]\\
\Gamma\bigl(\mu^{-1}(U), \kB\bigr) \times \Gamma\bigl(\mu^{-1}(U), \kB\bigr)  \ar[rr]^-{\kappa}  & & \Gamma\bigl(\mu^{-1}(U), \kO_X\bigr),
}
$$
is commutative, implying that $\mathsf{Im}\bigl(\Gamma(U, \kA) \times \Gamma(U, \kA) \stackrel{\kappa}\lar K\bigr) = \Gamma(U, \kO_E)$.
Hence, $A_s$ is a coisotropic Lie subalgebra of $A_K$ with respect to the pairing $\kappa^\omega$ given by (\ref{E:canpairing}).
 By Theorem \ref{T:CYBEtorsfree},
the geometric $r$--matrix corresponding to $\bigl((E, \kA), (\breve{E}, \omega)\bigr)$ is skew--symmetric. It is an interesting open problem to classify  simple locally
free sheaves on non--reduced Calabi--Yau curves and compute the corresponding solutions of CYBE.
\end{remark}

\subsection{Geometric $r$--matrices and bialgebra structures}\label{SS:Liebialgebras} In this subsection, let $E$ be a Weierstra\ss{} cubic, $p \in E$ its ``infinite point'' and
$s \in E$ its singular point provided $E$ is singular.  Let  $\kA$  be a coherent sheaf of Lie algebras on $E$, satisfying all  properties  from  Theorem \ref{T:CYBEtorsfree}. Consider the smooth affine curve
$$
E^\circ :=
 \left\{
 \begin{array}{cl}
  E \setminus \{p\}  & E \mbox{\rm \; is smooth}, \\
  E \setminus \{s\} &   E \mbox{\rm \; is singular}
 \end{array}
 \right.
$$
and put $A:= \Gamma\bigl(E^\circ, \kA\bigr)$. According to Theorem \ref{T:CYBEtorsfree}, the geometric $r$--matrix
$\rho \in \Gamma\bigl(E^\circ \times E^\circ\setminus \Delta, \kA \boxtimes \kA\bigr)$ is a skew--symmetric solution of the classical Yang--Baxter equation.

\begin{proposition}
The $\kk$--linear map
$
A \stackrel{\theta}\lar A \otimes_\kk A$ given by the formula $f \mapsto \bigl[f\otimes 1 + 1 \otimes f, \rho\bigr]$, is  skew--symmetric and  satisfies the so--called \emph{co--Jacobi identity}
\begin{equation}\label{E:coJacobi}
\tau \circ \bigl((\theta \otimes \mathsf{id}) \otimes \theta\bigr)= 0, \quad \mbox{\rm where}\quad \tau(a_1 \otimes a_2 \otimes a_3) =
a_1 \otimes a_2 \otimes a_3 +
a_3 \otimes a_1 \otimes a_2 + a_2 \otimes a_3 \otimes a_1
\end{equation}
for any simple tensor $a_1 \otimes a_2 \otimes a_3 \in A \otimes_\kk A \otimes_\kk A$. In other words,
the pair $(A, \theta)$ is a \emph{Lie bialgebra}.
\end{proposition}

\begin{proof}
For any $f \in A$ we have:
$
g:= \theta(f) =  \bigl[f\otimes 1 + 1 \otimes f, \rho\bigr] \in \Gamma\bigl(E^\circ \times E^\circ\setminus \Delta, \kA \boxtimes \kA\bigr).
$
We first have to show  that   $g$ has no pole along
$\Delta$, i.e.~$g\in \Gamma\bigl(E^\circ \times E^\circ, \kA \boxtimes \kA\bigr)$. For any $x_1 \ne x_2 \in
E^\circ$ we have:
$
g(x_1, x_2) = \Bigl[f(x_1)\otimes 1 + 1 \otimes f(x_2), \dfrac{\nu(x_2)}{x_1-x_2}\widetilde{\gamma} + \upsilon(x_1, x_2)\Bigr],
$
where a local presentation (\ref{E:structureRmat}) of $\rho$ is used. Next, we have:
$$
g(x_1, x_2) =
 \dfrac{\nu(x_2)}{x_1-x_2} \cdot\Bigl(\Bigl[f(x_1) \otimes 1 + 1 \otimes f(x_1), \widetilde{\gamma}\Bigr] +
\Bigl[1 \otimes \bigl(f(x_1)-f(x_2)\bigr), \widetilde{\gamma}\Bigr]\Bigr) + $$
$$\Bigl[f(x_1)\otimes 1 + 1 \otimes f(x_2), \upsilon(x_1, x_2)\Bigr].
$$
Hence, the section $g$ is indeed regular, as claimed. The co--Jacobi identity (\ref{E:coJacobi}) can be proven along the same lines as in \cite[Lemma 2.1.3]{ChariPressley}.
\end{proof}

\begin{remark}\label{R:ellipticsolutions}
Let $E$ be an elliptic curve. All indecomposable (in particular, all simple) locally free sheaves  on $E$ were described by Atiyah \cite{Atiyah}.
According to his classification, a simple
locally free sheaf  $\kF$ on $E$ is uniquely determined by its rank $n$, degree $d$ and determinant
$
\mathsf{det}(\kF) \cong \kO_E\bigl([q] + (d-1)[p]\bigr)\in \Pic^d(E) \cong E;
$
in this case we automatically have: $\mathsf{gcd}(n, d) = 1$. It follows from \cite{Atiyah}, that the sheaf of Lie algebras $\kA_{(n,d)} = \Ad(\kF)$ does not depend on the moduli parameter $q \in E$ and is determined by the pair $(n, d)$, which can be without  loss of generality normalized by the condition  $0 \le  d < n$.  It is  easy  to show, that $\kA_{(n, d)} \cong \kA_{(n, n-d)}$ as sheaves of Lie algebras, see \cite[Proposition 2.14]{BH}.

The geometric $r$--matrix, attached to the pair $(E, \kA_{n,d})$, gives precisely an elliptic solution of Belavin; see \cite[Theorem 5.5]{BH}. In particular, all elliptic solutions of (\ref{E:CYBE1}) can be  realized geometrically. The Lie
algebra $A_{(2,1)} = \Gamma\bigl(E^\circ, \kA_{(2,1)}\bigr)$  (expressed in terms of generators and relations) appeared for the first time  in a work of Golod \cite{Golod}.
Combining  \cite[Proposition 5.1]{BH} and \cite[Theorem 3.1]{SkrypnykEllAlg}, one gets a description of the Lie algebra $A_{(n,d)}$ in terms of generators and relations for general mutually prime $(n, d)$.
Poisson structures, related with the Lie bialgebra structure of $A_{(n,d)}$ as well as applications to the theory of classical integrable systems
 were studied  in works of Reyman and Semenov--Tian--Shansky \cite{ReymanST} and Hurtubise and Markman \cite{HurtubiseMarkman}.  See also \cite{KurokiTakebe} for yet another appearance  of this Lie bialgebra in conformal field theory. In
 \cite[Section 5]{GinzburgKapranovVasserot}, Ginzburg, Kapranov and Vasserot studied   $A_{(n, d)}$ with tools   of the geometric representation theory.
\end{remark}

\begin{remark} Let $(E, \kA)$ be a pair from Theorem \ref{T:CYBEtorsfree}. Assume additionally that $E$ is singular and $A_K \cong \lieg \otimes_\kk K$.
In this case, it can be shown\footnote{We are thankful to Oksana Yakimova for explaining this fact to us.}, that $A \cong \lieg \otimes_\kk \Gamma\bigl(E^\circ, \kO_E\bigr).$ Let $\PP^1 \stackrel{\nu}\lar E$ be the normalization map.
Then the following results are true.

\smallskip
\noindent
1.~Assume $E$ is nodal. We can choose homogeneous  coordinates on $E$ in such a way that $\nu^{-1}(s) = \bigl\{(0:1), (1: 0)\bigr\}$. Then we can put
$\omega = \dfrac{dz}{z}$ and write $\Gamma\bigl(E^{\circ}, \kO_E\bigr)\cong \CC[z, z^{-1}]$.  Then there exists an isomorphism
$A \stackrel{\xi}\lar  \lieg[z, z^{-1}]$ and the corresponding skew--symmetric solution $r = \rho^\xi$ of (\ref{E:CYBE}) takes  the form
\begin{equation}\label{E:genquasitrig}
r(x_1, x_2) = \frac{x_2}{x_1 - x_2} \gamma + v(x_1, x_2),
\end{equation}
where $v \in \CC[x_1^\pm, x_2^\pm]$. In other words, under these assumptions, we get a bialgebra structure on the Lie algebra $\lieg[z, z^{-1}]$.
By the  same argument as in  \cite[Theorem 19]{KPSST}, one can prove that $r$ is gauge--equivalent to a trigonometric solution of (\ref{E:CYBE1}).

Now, assume $X$ is a cycle of projective lines, $\kF$ a simple vector bundle on $X$ and $Y$ an irreducible component of $X$. Let
$X \stackrel{\mu}\lar E$ be the minimal model of $X$ relative to $Y$ and $\kA:= \mu_*\bigl(\Ad(\kF)\bigr)$. In this case it  can be  shown  that
$v \in \CC[x_1, x_2]$, i.e.~the solution $r$ is \emph{quasi--trigonometric} in the sense of
\cite{KPSST}.

\smallskip
\noindent
2.~If  $E$ is cuspidal, then we can choose homogeneous  coordinates on $E$ in such a way that $\nu^{-1}(s) = (1: 0)$. Then take
$\omega = dz$ and write $\Gamma\bigl(E^{\circ}, \kO_E\bigr)\cong \CC[z]$.  Similarly to the previous case,  there exists an isomorphism
$A \stackrel{\xi}\lar  \lieg[z]$ and the corresponding skew--symmetric solution $r = \rho^\xi$ of (\ref{E:CYBE}) has the form
\begin{equation}\label{E:genrat}
r(x_1, x_2) = \frac{1}{x_1 - x_2} \gamma + v(x_1, x_2),
\end{equation}
where $v \in \CC[x_1, x_2]$. In other words, $r$ is a rational skew--symmetric solution of (\ref{E:CYBE}) in the sense of Stolin \cite{Stolin}.

All Lie bialgebra structures on the Lie algebras $\lieg\llbracket z\rrbracket$ and $\lieg[z]$ were classified  by Montaner, Stolin and Zelmanov  in \cite{MSZ}. It is an interesting open problem to extend their methods  to  the case of the Lie algebra $\lieg[z, z^{-1}]$ as well as  to relate their classification to the geometric approach to CYBE, developed in this paper.
\end{remark}

\begin{remark} Let  $X$ be a  Kodaira cycle of projective lines, $C$ the regular part of its irreducible component  and $\kF$ a simple locally free sheaf on $X$ of rank $n \ge 2$. According to  Polishchuk \cite{Polishchuk1}, the geometric $r$--matrices  arising  from a datum  $\bigl((X, \Ad(\kF)), (C, \omega)\bigr)$ satisfy  a \emph{stronger} version
of CYBE:
\begin{equation}\label{E:sCYBE}
\left\{
\begin{array}{l}
\pi^{\otimes 3}\bigl(r(x_1,x_2)^{12} r(x_2,x_3)^{23} - r(x_1,x_3)^{13} r(x_1,x_2)^{12}  -
r(x_2,x_3)^{23} r(x_1,x_3)^{13}\bigr) = 0  \\
\pi^{\otimes 3}\bigl(r(x_1,x_2)^{12} r(x_1,x_3)^{13} + r(x_1,x_3)^{13} r(x_2,x_3)^{23} - r(x_2,x_3)^{23}
r(x_1,x_2)^{12}\bigr) = 0,
\end{array}
\right.
\end{equation}
where $\mathfrak{gl}_n(\CC)  \stackrel{\pi}\longrightarrow \mathfrak{pgl}_n(\CC)$ is the canonical projection. Assume that $r$ is a quasi--constant (quasi--)trigonometric solution of CYBE (i.e.~the function $v$ in (\ref{E:genquasitrig}) is constant). It has been proven by Schedler in \cite[Theorem 3.4]{Schedler} that the constraints  (\ref{E:sCYBE}) uniquely determine the ``continuous part'' of a solution for  a given   Belavin--Drinfeld triple. Already this observation shows that the class of pairs $\bigl((X, \Ad(\kF)), (C, \omega)\bigr)$ is too narrow to realize all quasi--trigonometric solutions of CYBE in a geometric way.
\end{remark}

\section{Geometrization of the rational solutions}
The goal of this section is to show that any non--degenerated skew--symmetric rational solution of the classical Yang--Baxter (\ref{E:CYBE}) arises from a pair $(E, \kA)$ as in Theorem \ref{T:CYBEtorsfree}, where $E$ is the cuspidal Weierstra\ss{} cubic.

\subsection{Lie algebra over the affine cuspidal curve from a rational solution of CYBE}
In what follows, $\lieg$ is a finite dimensional simple Lie algebra over $\CC$ of dimension $d$ and
$\CC^2  \stackrel{r}\lar \lieg \otimes \lieg$ a non--degenerate skew--symmetric  solution of CYBE of the form (\ref{E:genrat}).
\begin{itemize}
\item Consider the following symmetric non--degenerate  pairing of the Lie algebra $\lieg\llbrace z\rrbrace$:
\begin{equation}\label{E:RatPairing1}
\lieg\llbrace z\rrbrace \times \lieg\llbrace z\rrbrace \stackrel{\bar\kappa}\lar \CC, \quad \bar\kappa(az^l, bz^k) := \res_0\bigl(\kappa(a, b)z^{l+k}dz\bigr) \; \mbox{\rm for any}\; a, b \in \lieg.
\end{equation}
\item Let $e_1, \dots, e_d$ be some orthonormal basis of $\lieg$ with respect to the Killing form $\kappa$. Obviously, we have:
$\gamma = \sum_{l=1}^d e_l \otimes e_l$. Moreover, we have a formal power series expansion
\begin{equation}\label{E:expRatSol}
r(x, y) = \frac{1}{y - x} \gamma + v(x, y) = \sum\limits_{k=0}^\infty \sum\limits_{l = 1}^d \bigl(y^{-k-1} e_l + p_{k, l}(y)\bigr) \otimes x^k e_l,
\end{equation}
where $p_{k, l} \in \lieg[z]$ for any $k \in \NN_0$ and $1 \le l \le d$. Clearly, $p_{k, l} = 0$ for any $k \ge t+1$, where $t = \mathsf{deg}_x(v)$.
\item For any pair $(k, l)$ as above, we put $f_{k, l}:= e_l z^{-k-1} + p_{k, l} \in \lieg[z, z^{-1}]$ and
denote
\begin{equation}
W:= \bigl\langle f_{k, l} \; \big| \; k \in \NN_0, 1 \le l \le d\bigr\rangle_{\CC} \subset \lieg[z, z^{-1}] \subset \lieg\llbrace z\rrbrace.
\end{equation}
\end{itemize}
In order to show, how a rational solution $r$ defines a sheaf of Lie algebras on the cuspidal Weierstra\ss{} curve,  we begin with
 the following statement.
\begin{theorem} Let $r$ be a non--degenerate skew--symmetric rational solution of (\ref{E:CYBE}). Then the  following results are true.
\begin{itemize}
\item We have a direct sum decomposition
\begin{equation}\label{E:LagrDecRational}
\lieg\llbrace z\rrbrace = \lieg\llbracket z\rrbracket \dotplus W,
\end{equation}
i.e.~$\lieg\llbracket z\rrbracket +  W = \lieg\llbrace z\rrbrace$ and $\lieg\llbracket z\rrbracket \cap W = 0$.
\item The vector space $W$ is a Lagrangian Lie subalgebra of $\lieg\llbrace z\rrbrace$, i.e.~$W^\perp = W$ with respect to the pairing $\bar\kappa$. In other words, the decomposition (\ref{E:expRatSol}) is a
\emph{Manin triple}.
\item Let $S := \CC\bigl[z^{-2}, z^{-3}\bigr]$. Then we have: $S \cdot W = W$. Moreover, $W$ is a finitely generated torsion free $S$--module such that
$\CC\llbrace z\rrbrace \otimes_S W \cong \lieg\llbrace z\rrbrace$.
\end{itemize}
\end{theorem}

\begin{proof} It follows from the definition of the vector space $W$ that  any $0 \ne w \in W$ has a non--trivial Laurent part. This implies that
$\lieg\llbracket z\rrbracket \cap W = 0$. Next, for any $k \in \NN_0$ and  $1 \le l \le d$ we have:
$e_l z^{-k-1} = f_{k, l} - p_{k, l} \in  W + \lieg\llbracket z\rrbracket$. Therefore, $z^{-1} \lieg[z^{-1}] \subseteq W + \lieg\llbracket z\rrbracket$, hence $\lieg\llbracket z\rrbracket +  W = \lieg\llbrace z\rrbrace$ as claimed.

Next, let us prove that $W$ is a  coisotropic subspace of $\lieg\llbrace z\rrbrace$, i.e.~$W \subseteq W^\perp$. We have to show that
$\bar\kappa(w_1, w_2) = 0$ for any $w_1, w_2 \in W$. To do this, it is sufficient to take elements of the basis of $W$. We have:
\begin{equation*}
\bar\kappa\bigl(f_{k,l}, f_{r,t}\bigr) =
\bar\kappa\bigl(e_l z^{-k-1} + p_{k, l}, e_t z^{-r-1} + p_{r, t}\bigr) = \bar\kappa\bigl(e_l z^{-k-1}, p_{r, t}\bigr) + \bar\kappa\bigl( p_{k, l},
e_t z^{-r-1}\bigr).
\end{equation*}
Let us write
$
p_{k,l}(z) = \sum\limits_{m = 0}^\infty \sum\limits_{j = 1}^d \alpha_{m,j}^{k,l} \bigl(e_j z^m\bigr)$,  where $\alpha_{m,j}^{k,l} \in \CC$.
Then the statement that the subspace $W$ is coisotropic, is equivalent to
\begin{equation}\label{E:coisotropskewsymm}
0 = \bar\kappa\bigl(f_{k,l}, f_{r,t}\bigr) = \alpha_{r,t}^{k,l} + \alpha_{k,l}^{r,t} \; \mbox{\rm for any}\; k, r \in \NN_0 \;
\mbox{\rm and}\; 1 \le l, t \le d.
\end{equation}
It is not difficult to show, that the equality (\ref{E:coisotropskewsymm}) is equivalent to the \emph{skew--symmetry} of the solution $r$. Moreover, it is easy to see that $\lieg\llbracket z\rrbracket^\perp = \lieg\llbracket z\rrbracket$. Since we already proved that $\lieg\llbrace z\rrbrace = \lieg\llbracket z\rrbracket \dotplus W$, the non--degeneracy  of the pairing  $\bar\kappa$  implies  that  $W = W^\perp$.

The statement   that $W$ is a Lie subalgebra of $\lieg\llbrace z\rrbrace$ is a consequence of the assumption that $r$ satisfies the Yang--Baxter equation
(\ref{E:CYBE}); see for example \cite[Proposition 6.2]{EtingofSchiffmann}.

Now, let us prove that the Lie subalgebra $W$ is stable under the multiplication by the elements of the ring  $S$. We start with the observation that
$z^{-t-1} \lieg\bigl[z^{-1}\bigr] \subseteq W$, thus
$$
W = W^\perp \subseteq \bigl(z^{-t-1} \lieg\bigl[z^{-1}\bigr]\bigr)^\perp = z^{t-1} \lieg\bigl[z^{-1}\bigr].
$$
Therefore, we have shown that
\begin{equation}\label{E:BoundsSpaceW}
z^{-t-1} \lieg\bigl[z^{-1}\bigr] \subseteq W \subseteq z^{t-1} \lieg\bigl[z^{-1}\bigr].
\end{equation}
Consider the $\CC$--algebra  $T := \CC\bigl[z^{-2t}, z^{-2t-1}, \dots \bigr] \subset \CC\bigl[z^{-1}\bigr]$. Since,
$$T \cdot \bigl(z^{t-1} \lieg\bigl[z^{-1}\bigr]\bigr) = z^{-t-1} \lieg\bigl[z^{-1}\bigr],$$ it follows
 from (\ref{E:BoundsSpaceW}) that $T \cdot W = W$ and moreover, that $W$ is a torsion free finitely generated module over  $T$.
 Consider the \emph{ring of multiplicators} of the Lie subalgebra $W$:
 $$
 B := \bigl\{f \in \CC\llbrace z\rrbrace \big| f\cdot W \subseteq W\bigr\}.
 $$

 \smallskip
 \noindent
 \underline{Claim}. We have: $T \subseteq B \subseteq \CC\bigl[z^{-1}\bigr]$.

\smallskip
 \noindent Indeed, the inclusion $T \subseteq B$ is clear. Let us first show that $B \subseteq \CC\bigl[z, z^{-1}\bigr]$.  Assume there exists $f \in B \cap \bigl(\CC\llbrace z\rrbrace \setminus  \CC\bigl[z, z^{-1}\bigr]\bigr)$ (in other words, $f$ is a multiplicator of $W$ which is  an \emph{infinite} Laurent series). Since  $z^{-t-1} \lieg\bigl[z^{-1}\bigr] \subseteq W$, we have:  $a (z^{-t-1} f) \in W$ for any $a \in \lieg$. But obviously, $ a(z^{-t-1} f) \notin  \lieg\bigl[z, z^{-1}\bigr]$, which  contradicts (\ref{E:BoundsSpaceW}).

\smallskip
\noindent
 Now, let us prove that $B \subseteq \CC\bigl[z^{-1}\bigr]$. Put $\bar{t}:= \mathsf{max}\{\mathsf{deg}(w) \,\big| \, w \in W\}$. According to
(\ref{E:BoundsSpaceW}) we have: $\bar{t} \le t-1$. Let $w_0 \in W$ be such that $\mathsf{deg}(w_0) = \bar{t}$. If there exists a Laurent polynomial
$f \in B$ with
$\mathsf{deg}(f)> 0$ then $\mathsf{deg}(f\cdot w_0) \ge \bar{t}+1$, contradiction.

\smallskip
\noindent
Denote now $Y := \Spec(T)$. Then $Y$ is an affine curve with a unique singular point $s$ corresponding to the maximal ideal
$\mathfrak{m} = \langle z^{-2t},  z^{-2t-1}, \dots \rangle_\CC$ in $T$. The normalization map $\mathbb{A}^1 \stackrel{\pi}\lar Y$ is bijective, implying that $s$ is a (higher) cusp. Now, we  may  pass to the completions at the point $s$:
\begin{itemize}
\item $\widehat{\kO}_{Y, s} \cong \CC\bigl\llbracket z^{-2t}, z^{-2t-1}, \dots\bigr\rrbracket$.
\item $Q\bigl(\widehat{\kO}_{Y, s}\bigr) \cong \CC\llbrace z^{-1}\rrbrace$ (the quotient field of $\widehat{\kO}_{Y, s}$).
\item $\widehat{W}$ is the $\mathfrak{m}$--adic completion of the $T$--module $W$.
\end{itemize}
Because of (\ref{E:BoundsSpaceW}) we have:
$
\widehat{W} = z^{-t-1} \lieg\llbracket z^{-1}\rrbracket + \bigl\langle f_{k,l} \; \big|\, 0 \le k \le t, 1 \le i \le d\bigr\rangle_{\CC}.
$
From this description of $\widehat{W}$ we derive that
$W = \widehat{W} \cap \lieg\bigl[z, z^{-1}\bigr]$.
Furthermore, the following results are true.
\begin{itemize}
\item The symmetric $\CC$--bilinear pairing
\begin{equation}\label{E:RatPairing2}
\lieg\llbrace z^{-1}\rrbrace \times \lieg\llbrace z^{-1}\rrbrace \stackrel{\widehat{\kappa}}\lar \CC, \; \mbox{\rm where}\; \widehat\kappa(az^l, bz^k) := \res_0\bigl(\kappa(a, b)z^{l+k}dz\bigr) \; \mbox{\rm for}\; a, b \in \lieg
\end{equation}
is well--defined and non--degenerate.
\item  We have a Lagrangian decomposition (with respect to the pairing $\widehat{\kappa}$):
\begin{equation}\label{E:LagrDecRational2}
\lieg\llbrace z^{-1}\rrbrace = \lieg[z] \dotplus \widehat{W},
\end{equation}
\end{itemize}
Actually, the Lagrangian decomposition (\ref{E:LagrDecRational2}) is precisely the one occurring in Stolin's theory of rational solutions of CYBE.
According to results of \cite{Stolin, Stolin2}, there exists an automorphism  $\varphi \in \Aut_{\CC[z]}\bigl(\lieg[z]\bigr)$ such that
$\widehat{W}' := \varphi(\widehat{W}) \subseteq z \lieg\llbracket z^{-1}\rrbracket$. Therefore,
$$
\bigl(z \lieg\llbracket z^{-1}\rrbracket\bigr)^\perp = z^{-1} \lieg\llbracket z^{-1}\rrbracket \subseteq \bigl(\widehat{W}'\bigr)^\perp = \widehat{W}'
\subseteq z \lieg\llbracket z^{-1}\rrbracket.
$$
Hence, the ring $S = \CC[z^{-2}, z^{-3}]$ stabilizes $\widehat{W}'$. Hence, it stabilizes $\widehat{W}$, too.
Since $W = \widehat{W} \cap \lieg\bigl[z, z^{-1}\bigr]$, we have: $S \cdot W = W$, as claimed.
Since $W$ is a finitely generated torsion free module over $T$, it is finitely generated and torsion free over $S$, too. Theorem is proven.
\end{proof}

\subsection{Spectral data and Beauville--Laszlo construction}
The affine cuspidal cubic $E_\circ := \Spec(S)$ admits a one--point--compactification to a projective cuspidal cubic
$E = \overline{V(v^2 - u^3)} \subset \PP^2$. Let $p \in E$ be the ``infinite'' point of $E$. Let $\widehat{O} = \widehat{\kO}_{E, p}$ be the completion of the local ring of $E$ at $p$ and $\widehat{Q}$ be the quotient field of $\widehat{O}$. Then we have a canonical map
$\Gamma(E_\circ, \kO_E) \stackrel{l_p}\lar \widehat{Q}$, assigning to a regular function on $E_\circ$ its formal Laurent expansion
at the point $p$. We have the following commutative diagram of $\CC$--algebras:
\begin{equation}\label{E:BLcusp}
\begin{array}{c}
\xymatrix{
\Gamma(E_\circ, \kO_E) \ar@{^{(}->}[rr]^-{l_p} & & \widehat{Q} & & \widehat{O}\ar@{_{(}->}[ll] \\
 \CC[z^{-2}, z^{-3}]\ar@{^{(}->}[rr] \ar[rr] \ar[u]^-{\cong} &  & \CC\llbrace z\rrbrace  \ar[u]^-{\cong}  & & \ar@{_{(}->}[ll] \CC\llbracket z\rrbracket
 \ar[u]_-{\cong}
}
\end{array}
\end{equation}
To simplify the notation, we shall view the vertical isomorphisms as the identity maps.
Let $V := z^{-t-1} \lieg\llbracket z^{-1}\rrbracket$. According to (\ref{E:BoundsSpaceW}) we have: $V \subseteq W$. Moreover, it is not difficult to see that the quotient vector space $W/V$ is finite dimensional. Therefore, the induced map $\widehat{Q} \otimes_S V \lar \widehat{Q} \otimes_S W$ is an isomorphism.
Let $F := \lieg\llbracket z \rrbracket$. Let $\phi$ be the  $\widehat{Q}$--linear isomorphism of Lie algebras making the following diagram
\begin{equation*}
\begin{array}{c}
\xymatrix{
\widehat{Q} \otimes_S W
 \ar[d]_-{\phi}& & \ar[ll]_-{\cong} \widehat{Q} \otimes_S V \ar[rr]^-{\mathsf{mult}} & & \lieg\llbrace z\rrbrace \ar[d]^{=} \\
\widehat{Q} \otimes_{\widehat{O}} F \ar[rrrr]^-{\mathsf{mult}} & & & & \lieg\llbrace z\rrbrace
}
\end{array}
\end{equation*}
commutative.   According to a construction of Beauville and Laszlo \cite{BL}, the datum  $(W, F, \phi)$ defines a sheaf of Lie algebras $\kA$ on the projective curve $E$, together with isomorphisms of Lie algebras
$
\Gamma\bigl(E_\circ, \kA\bigr) \stackrel{\alpha}\lar W$ and $\widehat{A}_p \stackrel{\beta}\lar W
$
such that the following diagram
\begin{equation}\label{E:morphismBLtriples}
\begin{array}{c}
\xymatrix{
\widehat{Q} \otimes_{S}  \Gamma\bigl(E_\circ, \kA\bigr)
 \ar[d]_{\mathsf{id} \otimes \alpha} \ar[rr]^-{\mathsf{can}}  & & \widehat{Q}\otimes_{\widehat{O}} \widehat{A}_p  \ar[d]^-{\mathsf{id} \otimes \beta} \\
\widehat{Q}  \otimes_S W \ar[rr]^-{\phi} & &  \widehat{Q}\otimes_{\widehat{O}}  F
}
\end{array}
\end{equation}
is commutative. Here, the Lie algebra $A_p$ is the germ of the sheaf $\kA$ at the point $p$ and $\widehat{A}_p$ is its completion.
An interested reader might look for a more detailed exposition in \cite[Section 1.7]{BurbanZheglov}, where the Beauville--Laszlo construction occurred in another setting.
The following sequence of vector spaces (which is a version  of the Mayer--Vietoris sequence)
\begin{equation}\label{E:MVsequence}
0 \lar H^0(E, \kA) \lar \Gamma(E_\circ, \kA) \oplus \widehat{A}_p \lar Q\bigl(\widehat{A}_p\bigr) \lar H^1(E, \kA) \lar 0
\end{equation}
is exact; see e.g.~\cite[Proposition 3]{Parshin}.
Since $W + F = \lieg\llbrace z\rrbrace$ and $W \cap  F = 0$, we may conclude that
$H^0(E, \kA) = 0 = H^1(E, \kA)$.

\begin{lemma}\label{L:locfreeRatSol}
As usual, let $\breve{E}$ be the regular part of $E$. Then we have:
$$
\Gamma(\breve{E}, \kA) \cong \lieg \otimes_{\CC}  \Gamma(\breve{E}, \kO_E) \cong \lieg[z].
$$
\begin{proof}
It follows from (\ref{E:BoundsSpaceW}) that $\Gamma\bigl(E_0\setminus\{s\}, \kA\bigr) \cong \lieg[z, z^{-1}]$. Again, we have  the Mayer--Vietoris exact sequence
$
0 \lar \Gamma\bigl(\breve{E}, \kA\bigr) \lar \Gamma\bigl(\breve{E}\setminus\{p\}, \kA\bigr)\oplus \widehat{A}_p \lar Q\bigl(\widehat{A}_p\bigr) \lar 0.
$
Since $\lieg[z, z^{-1}] \cap \lieg\llbracket z\rrbracket = \lieg[z]$, we get the result.
\end{proof}
\end{lemma}

\subsection{Comparison theorem} Let us first summarize the results and constructions made in the previous two subsections.
We have started with a non--degenerate skew--symmetric rational solution $r$ of CYBE (\ref{E:genrat}). Then we attached to it a Lagrangian Lie subalgebra
$W \subset \lieg\llbrace z\rrbrace$  with respect to the pairing (\ref{E:RatPairing1}). Next, we proved that the ring $S = \CC[z^{-2}, z^{-3}]$ stabilizes
$W$ and that $W$ is a finitely generated torsion free $S$--module. Hence, we get a sheaf of Lie algebras on the affine cuspidal curve
$E_\circ = \Spec(S)$. The curve $E_\circ$ can be completed by a single smooth point $p$ to a projective curve $E$, which is isomorphic to the  Weierstra\ss{} cubic. As usual, we denote by $s$ the unique singular point of $E$ and by $\breve{E}$ the regular part of $E$.
  Using a construction of Beauville and Laszlo \cite{BL}, we extended the $S$--module $W$ to a torsion free coherent sheaf of Lie algebras on $E$ such that
\begin{itemize}
\item there exists an $S$--linear isomorphism of Lie algebras $\Gamma\bigl(E_\circ, \kA\bigr) \stackrel{\alpha}\lar W$ and a $\CC\llbrace z\rrbrace$--linear isomorphism of Lie algebras $\widehat{A}_p \stackrel{\beta}\lar \lieg\llbrace z\rrbrace$, satisfying the  condition
    (\ref{E:morphismBLtriples});
\item we have: $H^0(E, \kA) = 0 = H^1(E, \kA)$.
\end{itemize}
According to Lemma \ref{L:locfreeRatSol}, we have an isomorphism $
\Gamma(\breve{E}, \kA) \stackrel{\vartheta}\lar  \lieg[z]
$
induced by the isomorphisms $\alpha$ and $\beta$. In particular, the sheaf $\kA$ is weakly $\lieg$--locally free over $\breve{E}$.
In these notations, $z$ is a local parameter of the point $p$ and $\omega = dz$ spans the vector space
$\Gamma\bigl(E, \Omega_E\bigr)$. Since the Lie algebra $\widehat{W} \subset \lieg\llbrace z^{-1}\rrbrace$ is Lagrangian with respect to the pairing
(\ref{E:LagrDecRational2}), the Lie algebra $A_s$ is a Lagrangian subalgebra of $A_K$. Summing up, $(E, \kA)$ is a pair satisfying all conditions, which are necessary to apply Theorem \ref{T:CYBEtorsfree}. Now we are prepared to prove the main result of this section.

\begin{theorem}\label{T:geometrizationRational} Let $r$ be a non--degenerate skew--symmetric rational solution of CYBE and $\kA$  the sheaf of Lie algebras on the cuspidal cubic
$E$, described above. Let $\rho \in \Gamma\bigl(\breve{E}\times \breve{E}\setminus \Delta, \kA \boxtimes \kA\bigr)$ be the geometric $r$--matrix, attached to the datum $\bigl((E, \kA), (\breve{E}, \omega)\bigr)$ and $\widetilde{r} := \rho^{\vartheta}: \CC^2 \lar \lieg \otimes \lieg$ be the corresponding solution of CYBE, where $\Gamma(\breve{E}, \kA) \stackrel{\vartheta}\lar  \lieg[z]$ is the trivialization described above. Then for any point
$(x_1, x_2) \in \breve{E} \times \breve{E} \setminus \Delta \cong \CC^2\setminus \Delta$ we have: $\widetilde{r}(x_1, x_2) = r(x_2, x_1) = - r(x_1, x_2)$, i.e.~we essentially get the same  solution of CYBE we started with.
\end{theorem}
\begin{proof}
For any $n \in \NN_0$, consider the scheme $P_n := \Spec\bigl(\CC[x]/x^{n+1}\bigr)$ and the morphism $P_n \stackrel{\jmath_n}\lar \breve{E}$ mapping the closed point of $P_n$ to the ``infinite'' point $p$ of $\breve{E}$. Note that $\bigl(E_\circ \times \{p\}\bigr) \cap \Delta = \emptyset$, where the intersection is taken inside of $E_\circ \times \breve{E}$. Therefore, we have a morphism of schemes
$
E_\circ \times P_n \xrightarrow{\mathsf{id}\times \jmath_n} \bigl(E_\circ \times \breve{E}\bigr)\setminus \Delta.
$
Consider now the following commutative diagram
\begin{equation}
\begin{array}{c}
\xymatrix{
\Gamma\bigl(E \times \breve{E}, \kA \boxtimes \kA(\Delta)\bigr) \ar@{_{(}->}[d] \ar@{^{(}->}[drr] & & \\
\Gamma\bigl(E \times \breve{E} \setminus \Delta, \kA \boxtimes \kA\bigr) \ar@{^{(}->}[rr] \ar[d]_-{(\mathsf{id}\times \jmath_n)^\ast} & & \Gamma\bigl(\breve{E} \times \breve{E} \setminus \Delta, \kA \boxtimes \kA\bigr) \ar[d]^-{(\mathsf{id}\times \jmath_n)^\ast} \\
\Gamma\bigl(E_\circ, \kA\bigr)\otimes \Gamma\bigl(P_n, \jmath_n^*\kA\bigr) \ar@{^{(}->}[rr] \ar[d]_-{\vartheta \otimes \bar\beta} & & \Gamma\bigl(\breve{E}_\circ, \kA\bigr)\otimes \Gamma\bigl(P_n, \jmath_n^*\kA\bigr) \ar[d]^-{\alpha_{\mid} \otimes \bar\beta} \\
W \otimes \bigl(\lieg[x]/x^{n+1}\bigr) \ar[rr]^-{\tau_n} & & \bigl(\lieg \otimes \lieg\bigr)\bigl[x, y,y^{-1}\bigr]/(x^{n+1}).
}
\end{array}
\end{equation}
Here, $\breve{E}_\circ = E \setminus \{s, p\}$ and  $\bar\beta$ is   the morphism induced by the trivialization $\widehat{A}_p \stackrel{\beta}\lar \lieg \llbracket z\rrbracket$, whereas $\tau_n$ is the \emph{isomorphism} of vector spaces induced by the  isomorphism
$W\big|_{\breve{E}_\circ} \cong (\lieg \otimes \lieg)[z, z^{-1}]$, which is obtained by localizing (\ref{E:BoundsSpaceW}) with respect to
the multiplicative subset $\bigl\{z^{-k} \big| k \in \NN_0 \setminus \{1\}\bigr\} \subset S$.  Recall that we  have a distinguished element $\varrho \in \Gamma\bigl(E \times \breve{E}, \kA \boxtimes \kA(\Delta)\bigr)$, which restricts to the geometric $r$--matrix $\rho \in \Gamma\bigl(\breve{E} \times \breve{E} \setminus \Delta, \kA \boxtimes \kA\bigr)$. The corresponding solution of CYBE has the form
$
\widetilde{r}(y, x) = \dfrac{\gamma}{y-x} + \widetilde{v}(x, y),
$
where $\widetilde{v} \in (\lieg \otimes \lieg)[x, y]$. Next, we have the following equality in the ring $\CC\bigl[x, y, y^{-1}\bigr]/(x^{n+1})$:
$$
\frac{1}{y-x} = \sum\limits_{k= 0}^n y^{-k-1} x^k;
$$
in particular, $y-x \in \CC\bigl[x, y, y^{-1}\bigr]/(x^{n+1})$ is a unit.
Similarly to (\ref{E:expRatSol}), we get an expansion
\begin{equation}\label{E:expRatSol2}
\widetilde{r}_{n}(y, x) =  \sum\limits_{k=0}^n \sum\limits_{l = 1}^d \bigl(y^{-k-1} e_l + \widetilde{p}_{k, l}(y)\bigr) \otimes x^k e_l,
\end{equation}
where $\widetilde{p}_{k, l}$ are some elements of $\lieg[y]$ for  $k \in \NN_0$ and $1 \le l \le d$. In fact, (\ref{E:expRatSol2}) is precisely the image
of the geometric $r$--matrix $\rho$ in the vector space $\bigl(\lieg \otimes \lieg\bigr)\bigl[x, y,y^{-1}\bigr]/(x^{n+1})$. However, the element
$\widetilde{r}_{n}(y, x)$ belongs to the image of some element of the vector space $W \otimes \bigl(\lieg[x]/x^{n+1}\bigr)$. Any element
$\bar{\bar{r}}$ of this space can be written as
$$
\bar{\bar{r}} =  \sum\limits_{k, i=0}^n \sum\limits_{l,j = 1}^d \lambda_{k,l}^{i,j}\bigl(y^{-k-1} e_l + p_{k, l}(y)\bigr) \otimes x^i e_j
$$
for some uniquely determined constants $\lambda_{k,l}^{i,j} \in \CC$. Looking at the Laurent part of the identity
$\tau_n\bigl(\bar{\bar{r}}\bigr) = \widetilde{r}_{n}(y, x)$, we conclude that
$$
\lambda_{k,l}^{i,j} =
\left\{
\begin{array}{cccc}
1 & \mbox{\rm if}\; (i, j) & = &  (k, l)\\
0 & \mbox{\rm if}\; (i, j) &\ne & (k, l)
\end{array}
\right.
$$
and that $p_{k, l} = \widetilde{p}_{k, l}$ for any $0 \le k \le n$ and $1 \le l \le d$. Since $n\in \NN_0$ can be taken arbitrary large, we arrive to the claimed statement.
\end{proof}

\begin{corollary}
Theorem \ref{T:geometrizationRational} implies that any non--degenerate skew--symmetric rational solution of CYBE arises from a geometric $r$--matrix for an appropriate pair $(E, \kA)$, where $E$ is the cuspidal Weierstra\ss{} cubic.
\end{corollary}

\section{Manin triples and the geometric $r$--matrix}
\begin{proposition}\label{P:canManinTriples} Let $E$ be a Weierstra\ss{} cubic, $p \in E$ its ``infinite'' point, $\widehat{Q}_p$ the field of fractions of the local ring $\widehat{O}_p$,
$\omega$ a generator of  $\Gamma(E, \Omega_E)$ and $\kA$ a sheaf of Lie algebras on $E$ satisfying the conditions  of  Theorem \ref{T:CYBEtorsfree}. Then we have a Lagrangian decomposition
\begin{equation}\label{E:ManinTripleCanonical}
Q(\widehat{A}_p) = \widehat{A}_p \dotplus \Gamma\bigl(E_\circ, \kA\bigr),
\end{equation}
where the $\CC$--bilinear pairing on $Q(\widehat{A}_p)$ is given by the rule
$Q(\widehat{A}_p) \times Q(\widehat{A}_p) \stackrel{\kappa_p}\lar \widehat{Q}_p \stackrel{\res_p^\omega}\lar \CC$.
\end{proposition}

\begin{proof} The fact that we have a direct sum decomposition (\ref{E:ManinTripleCanonical}), is a consequence of the Mayer--Vietoris exact sequence
(\ref{E:MVsequence}) and the assumption that $H^0(E, \kA) = 0 = H^1(E, \kA)$. Hence, we have to show that (\ref{E:ManinTripleCanonical}) is a Lagrangian decomposition. To do this, it is sufficient to prove that both Lie algebras $\widehat{A}_p$ and $\Gamma\bigl(E_\circ, \kA\bigr)$ are coisotropic subspaces
of $Q(\widehat{A}_p)$. Obviously, the  differential form $\omega$ is regular at the point $p$. Next, the Killing form
$\widehat{A}_p \times \widehat{A}_p \stackrel{\kappa_p}\lar  \widehat{Q}_p$ actually takes value in the ring $\widehat{O}_p$. Therefore,
the Lie algebra $\widehat{A}_p$ is a coisotropic subspace of $Q(\widehat{A}_p)$, as claimed.

To prove that $\Gamma\bigl(E_\circ, \kA\bigr)$ is a coisotropic Lie subalgebra  of $Q(\widehat{A}_p)$ as well, take any two elements $f, g \in \Gamma\bigl(E_\circ, \kA\bigr)$. Then we get a rational function $\kappa(f,g) \in K$, where here $\kappa$ is the Killing form on the rational envelope $A_K$. The residue
theorem implies that
$$
0 = \sum\limits_{q \in E} \res_q\bigl(\kappa(f, g)\omega\bigr) = \res_s\bigl(\kappa(f, g)\omega\bigr) +
\sum\limits_{q \in \breve{E}_\circ} \res_q\bigl(\kappa(f, g)\omega\bigr) +
\res_p\bigl(\kappa(f, g)\omega\bigr).
$$
We have: $\res_s\bigl(\kappa(f, g)\omega\bigr) = 0$, because $A_s \subset A_K$ is coisotropic (if $E$ is smooth then this  contribution is void). Next, $\res_q\bigl(\kappa(f, g)\omega\bigr) = 0$
for any $q \in \breve{E}_\circ$, since $\kappa(f, g) \in \Gamma\bigl(\breve{E}_\circ, \kO_E\bigr)$ and $\omega$ is regular at $q$. Hence,
$\res_p\bigl(\kappa(f, g)\omega\bigr) = 0$ for any $f, g \in \Gamma\bigl(E_\circ, \kA\bigr)$ as claimed.
\end{proof}

\begin{remark}\label{R:ManinTriples} Let $r$ be an arbitrary non--degenerate skew--symmetric solution of (\ref{E:CYBE}) of the form
$
r(x, y) = \dfrac{\gamma}{y-x} + u(x, y),
$
where $(\CC^2, 0) \stackrel{u}\lar \lieg \otimes \lieg$ is holomorphic (note that any solution of CYBE with one spectral parameter (\ref{E:CYBE1}) automatically has this property by a result of Belavin and Drinfeld \cite{BelavinDrinfeld}). In the domain $|x| <  |y|$, we can write an expansion
\begin{equation}\label{E:expNonDegSol}
r(x, y)  = \sum\limits_{k=0}^\infty \sum\limits_{l = 1}^d \bigl(y^{-k-1} e_l + u_{k, l}(y)\bigr) \otimes x^k e_l,
\end{equation}
where $(\CC, 0) \stackrel{u_{k,l}}\lar \lieg$ are germs of holomorphic functions for $k \in \NN_0$ and $1\le l \le d$. Consider the vector space
\begin{equation}
W:= \bigl\langle z^{-k-1} e_l + u_{k, l}(z) \; \big| \; k \in \NN_0, \; 1 \le l \le d\bigr\rangle_{\CC}  \subseteq \lieg\llbrace z\rrbrace.
\end{equation}
Then we get a Lagrangian decomposition (Manin triple)
\begin{equation}\label{E:basicManinTriple}
\lieg\llbrace z\rrbrace = \lieg\llbracket z\rrbracket \dotplus W
\end{equation}
with respect to the pairing (\ref{E:RatPairing1}); see for example \cite[Proposition 6.2]{EtingofSchiffmann}. Conversely, starting with a Manin triple
(\ref{E:basicManinTriple}), we can take the ``topological basis'' $\bigl(e_l z^k \big| k \in \NN_0, \; 1 \le l \le d\bigr)$ of $\lieg\llbracket z\rrbracket$. Then we get a uniquely determined dual basis $\bigl(f_{k,l} \big| k \in \NN_0, \; 1 \le l \le d\bigr)$ of the Lie algebra $W$. If (\ref{E:basicManinTriple}) is a decomposition arising from an $r$--matrix (\ref{E:expNonDegSol})  then $f_{k,l}(z) = z^{-k-1} e_l + u_{k, l}(z)$.

Manin triples of the form (\ref{E:basicManinTriple}), leading to the elliptic solutions of CYBE, were constructed  by Reyman and Semenov--Tian--Shansky \cite{ReymanST}. In fact, let  $(E, \kA)$ be  as in  Remark \ref{R:ellipticsolutions}. Then the resulting Manin triple (\ref{E:ManinTripleCanonical}) can be identified with the one from \cite{ReymanST}. It follows for example from a straightforward computation \cite[Theorem 5.5]{BH}, that the geometric $r$--matrix attached to the pair $(E, \kA)$ leads to the power series expression  as the one obtained by the approach via Manin triples. As we shall see later in Theorem \ref{T:TwoWaysCYBE}, this result can be deduced without doing any explicit computations. The algebro--geometric origin of   Manin triples (\ref{E:basicManinTriple}) was pointed out  by Drinfeld \cite[Section 3]{Drinfeld}. \qed
\end{remark}

\begin{remark}
Let us now discuss the case of a \emph{nodal} Weierstra\ss{} curve $E$, since some extra steps  to identify the Manin triples (\ref{E:ManinTripleCanonical}) and  (\ref{E:basicManinTriple}) have to be done. Consider the $\CC$--algebra
$$
S:= \bigl\{p \in \CC[w] \; \big| \; p(-1) = p(1)\bigr\} \subset \CC[w].
$$
Then $E_\circ := \Spec(S)$ is isomorphic to an affine nodal Weierstra\ss{} curve. Of course, it admits a one--point compactification by a smooth point $p$.
Algebraically, this compactification can be formalized by the commutative diagram
\begin{equation}\label{E:BLnode}
\begin{array}{c}
\xymatrix{
\Gamma(E_\circ, \kO_E) \ar@{^{(}->}[rr]^-{l_p} & & \widehat{Q}_p & & \widehat{O}_p\ar@{_{(}->}[ll] \\
 S \ar@{^{(}->}[rr] \ar[rr] \ar[u]^-{\cong} &  & \CC\llbrace w^{-1}\rrbrace  \ar[u]^-{\cong}  & & \ar@{_{(}->}[ll] \CC\llbracket w^{-1}\rrbracket.
 \ar[u]_-{\cong}
}
\end{array}
\end{equation}
Consider the differential one--form $\omega:= \dfrac{1}{1-w^2}$ on $E$ and put  $y := \dfrac{1}{w}$. Then $y$ is a local parameter at the point $p$ and
$\omega = \dfrac{1}{1-y^2}$.
It is not difficult to check that $\omega$ generates $\Gamma\bigl(E, \Omega_E\bigr)$.  Consider the algebra isomorphism
\begin{equation*}
\CC\llbracket y\rrbracket \stackrel{\psi}\lar \CC\llbracket z\rrbracket, \quad y \mapsto \mathrm{tanh}(z) = z - \frac{1}{3}z^3 + \frac{2}{15} z^5 + \dots
\end{equation*}
It is easy to see that $\widetilde{\omega}:= \psi(\omega) = dz$. Moreover, the following diagram
$$
\xymatrix{
\CC\llbrace y\rrbrace   \ar[d]_-{\psi} \ar[rr]^-{\res^{\omega}_0} & & \CC \ar[d]^-{=}\\
\CC\llbrace z\rrbrace  \ar[rr]^-{\res^{\widetilde{\omega}}_0} & & \CC
}
$$
is commutative; see for example \cite{Tate}. Note that under these identifications, we get an inclusion
$S \subset \CC\bigl[\mathrm{coth}(z)\bigr] =  \CC\left[\dfrac{\mathrm{ch}(z)}{\mathrm{sh}(z)}\right]$.

\smallskip
\noindent
Let $\kA$ be a sheaf of Lie algebras on a nodal Weierstra\ss{} cubic $E$ and  $\rho \in \Gamma\bigl(\breve{E}\times \breve{E} \setminus \Delta, \kA \boxtimes \kA\bigr)$ be the corresponding geometric $r$--matrix constructed in Theorem \ref{T:CYBEtorsfree}.
Let us choose local coordinates
in a neighbourhood of the point $(p, p) \in \breve{E}\times \breve{E}$, following  (\ref{E:BLnode}). According to (\ref{E:structureRmat}), we can write:
$
\rho(x_1, x_2) = \dfrac{1-x_2^2}{x_1-x_2} \widetilde{\gamma} + \upsilon(x_1, x_2),
$
where $\upsilon \in \Gamma\bigl(\breve{E}\times \breve{E}, \kA \boxtimes \kA\bigr)$. Let us put   $x_j = \mathrm{tanh}(y_j)$ for $j = 1,2$, then we get:
$$
\rho(y_1, y_2) = \frac{\widetilde{\gamma}}{\mathrm{ch}^2(y_2) \cdot\bigl(\mathrm{tanh}(y_1) - \mathrm{tanh}(y_2)\bigr)}  +
\upsilon\bigl(\mathrm{tanh}(y_1), \mathrm{tanh}(y_2)\bigr).
$$
Taking a formal trivialization $\widehat{A}_p \stackrel{\xi}\lar \lieg \llbracket z\rrbracket$, we obtain a skew--symmetric solution of CYBE as in Remark \ref{R:ManinTriples}. \qed
\end{remark}

Summing up, let us start with an arbitrary pair $(E, \kA)$ as in Theorem \ref{T:CYBEtorsfree}. We can always find  an algebra isomorphism $\widehat{O}_p \stackrel{\zeta}\lar \CC\llbracket z\rrbracket$ such that a  generator $\omega \in \Gamma\bigl(E, \Omega_E)$ takes the form:
$\omega = dz$. Choose an isomorphism of Lie algebras $\widehat{A}_p \stackrel{\xi}\lar \lieg\llbracket z\rrbracket$. Then the Lagrangian decomposition
(\ref{E:ManinTripleCanonical}) gets the form (\ref{E:basicManinTriple}), where $W$ is the image of $\Gamma\bigl(E_\circ, \kA\bigr)$ under the map
$Q\bigl(\widehat{A}_p\bigr)\stackrel{\widetilde\xi}\lar \lieg\llbrace z\rrbrace$ induced by $\xi$.

\begin{theorem}\label{T:TwoWaysCYBE} Let $(E, \kA)$ be a pair as in Theorem \ref{T:CYBEtorsfree} and $\bigl(f_{k,l} \big| k \in \NN_0, \; 1 \le l \le d\bigr)$ be the basis of the  Lie algebra $W$, which
 is dual to the ``topological basis'' $\bigl(e_l z^k \big| k \in \NN_0, \; 1 \le l \le d\bigr)$ of $\lieg\llbracket z\rrbracket$.   Then   $f_{k,l} = z^{-k-1} e_l + p_{k, l}$ for some $p_{k,l} \in \lieg\llbracket z\rrbracket$ and the formal power series
 \begin{equation}\label{E:rmatManinTriples}
 r(y, x) = \frac{\gamma}{y-x} + \sum\limits_{k=0}^\infty \sum\limits_{l = 1}^d p_{k, l}(y) \otimes x^k e_l,
 \end{equation}
 coincides with the geometric $r$--matrix $\rho$ trivialized at the formal neighbourhood of $(p, p) \in E \times \breve{E}$ by the map  induced by $\xi$.
\end{theorem}
\begin{proof} Consider the point $(p, p) \in E \times \breve{E}$ (we intentionally underline the asymmetry of these two factors); let $y$ be a local coordinate at the point $p \in E$ and $x$ a local coordinate of $p \in \breve{E}$ chosen in the way explained above.  For any $n \in \NN_0$, consider the scheme $P_n := \Spec\bigl(\CC[z]/z^{n+1}\bigr)$ and the morphism $P_n \stackrel{\jmath_n}\lar \breve{E}$ mapping the closed point of $P_n$ to
$p \in \breve{E}$.
 We have a commutative diagram
$$
\xymatrix{
   & \Gamma\bigl(E \times \breve{E}, \kA \boxtimes \kA(\Delta)\bigr) \ar@{_{(}->}[ld] \ar@{^{(}->}[rd] &  \\
\Gamma\bigl(E_\circ \times \breve{E}\setminus \Delta, \kA \boxtimes \kA\bigr) \ar[d] &  &  \Gamma\bigl(\breve{E}  \times \breve{E}\setminus \Delta, \kA \boxtimes \kA\bigr) \ar[d] \\
\Gamma\bigl(E_\circ, \kA\bigr)\otimes \Gamma\bigl(P_n, \jmath_n^*\kA\bigr) \ar@{^{(}->}[rr] \ar[d]_-{\widetilde{\xi}_y \otimes \bar{\xi}_x} & &  \Gamma\bigl(\breve{E}_\circ, \kA\bigr)\otimes \Gamma\bigl(P_n, \jmath_n^*\kA\bigr) \ar@{_{(}->}[d] \\
W \otimes \bigl(\lieg[x]/(x^{n+1})\bigr) \ar@{^{(}->}[rr]  & &  (\lieg \otimes \lieg) \otimes \bigl(\CC\llbrace y\rrbrace[x]/(x^{n+1})\bigr).
}
$$
By the same argument as in Theorem \ref{T:geometrizationRational}
one can show that the image
of the distinguished element $\varrho \in \Gamma\bigl(E \times \breve{E}, \kA \boxtimes \kA(\Delta)\bigr)$ in the vector space $(\lieg \otimes \lieg) \otimes \bigl(\CC\llbrace y\rrbrace[x]/(x^{n+1})\bigr)$ is equal to
$
 r_n(y, x) :=  \sum\limits_{k=0}^n \sum\limits_{l = 1}^d f_{k,l}(y) \otimes x^k e_l,
$
implying the result.
\end{proof}

\section{Some explicit computations}
In this section, we compute the solutions of CYBE for the Lie algebra $\mathfrak{sl}_n(\CC)$ arising from pairs $(E, \kA)$ by the recipe of Theorem
\ref{T:CYBEtorsfree}, where $E$ is a singular Weierstra\ss{} curve  and $\kA$ a certain sheaf of Lie algebras, which is not locally free at the singular point $s \in E$.

\subsection{Category of triples}
\label{SS:catoftriples}

We first recall some general techniques to describe torsion free sheaves  on singular  curves, following \cite{DrozdGreuel, Thesis, Survey}.
Let $X$ be a reduced singular  curve, $\widetilde{X} \stackrel{\nu}\lar   X$
its normalisation and $\kI :=
{\mbox{\textit{Hom}}}_\kO\bigl(\pi_*(\kO_{\widetilde{X}}), \kO\bigr)$
the conductor ideal sheaf.
Denote  by $Z = V(\kI) \stackrel{\eta}\lar X$ the
closed subscheme   defined by $\kI$
(whose  topological support is precisely the singular locus of $X$)  and by
$\widetilde{Z} \stackrel{\tilde\eta}\lar \widetilde{X}$ its preimage in
$\widetilde{X}$, defined by the Cartesian  diagram
\begin{equation}\label{E:diagtriples}
\begin{array}{c}
\xymatrix
{\widetilde{Z} \ar[r]^{\tilde{\eta}} \ar[d]_{\tilde{\nu}}
& \widetilde{X} \ar[d]^\nu \\
Z \ar[r]^\eta & X.
}
\end{array}
\end{equation}

\begin{definition}\label{D:TFTriples}
  The \emph{category of triples} $\Tri(X)$ is defined as follows.
  \begin{itemize}
  \item Its objects are triples
    $\bigl(\widetilde\kF, \kV, \theta\bigr)$, where
    $\widetilde\kF \in \VB(\widetilde{X})$, $\kV \in \Coh(Z)$ and
    $\tilde{\nu}^*\kV \stackrel{\theta}\lar \tilde{\eta}^*\widetilde\kF$ is an
    epimorphism in $\Coh(\widetilde{Z})$  such that the adjoint morphism in $\Coh(Z)$
    $$
    \kV \lar \tilde{\nu}_*\bigl(\tilde{\nu}^*\kV\bigr) \stackrel{\tilde{\nu}_*(\theta)}\lar \tilde{\nu}_*\bigl(\tilde{\eta}^*\widetilde\kF\bigr)
    $$
is a monomorphism.
  \item The set of morphisms
    $\Hom_{\Tri(X)}\bigl((\widetilde\kF_1, \kV_1, \theta_1),
    (\widetilde\kF_2, \kV_2, \theta_2)\bigr)$ consists of all pairs
    $(f, g)$, where $\widetilde\kF_1 \stackrel{f}\lar  \widetilde\kF_2$ and
    $\kV_1 \stackrel{g}\lar \kV_2$
    are morphisms of coherent sheaves  such that the following
    diagram
    $$
    \xymatrix
    {\tilde{\nu}^*\kV_1 \ar[rr]^-{\theta_1}\ar[d]_{\tilde{\nu}^*(g)} & &
      \tilde{\eta}^*\widetilde\kF_1
      \ar[d]^{\tilde{\eta}^*(f)} \\
      \tilde{\nu}^*\kV_2 \ar[rr]^-{\theta_2}
      & & \tilde{\eta}^*\widetilde\kF_2
    }
    $$
    is commutative.
  \end{itemize}
\end{definition}

\begin{theorem}\label{T:keyonTF}
  Let $X$ be a reduced curve. For any torsion free sheaf $\kF$ on $X$, consider the canonical morphism $\theta_{\kF}:
    \tilde{\pi}^*(\eta^*\kF) \lar \tilde\eta^*(\pi^*\kF) \lar \tilde\eta^*(\nu^*\kF/\mathsf{tor}(\nu^*\kF))$.
  Then the functor
  \begin{equation}
 \TF(X) \stackrel{\mathbb{F}}\lar \Tri(X), \quad \kF \mapsto  \bigl(\nu^*\kF/\mathsf{tor}(\nu^*\kF), \eta^*\kF,
    \theta_{\kF}\bigr)
  \end{equation}
  is an equivalence of categories. Moreover, let  $\Tri^{\mathsf{lf}}(X)$ be  the full subcategory of $\Tri(X)$ consisting of those objects $\bigl(\widetilde\kF, \kV, \theta\bigr)$, for which $\kV$ is a free $\kO_Z$--module and $\theta$ is an isomorphism. Then $\FF$ restricts on an equivalence of categories
  $\VB(X) \lar \Tri^{\mathsf{lf}}(X)$. Finally, for any torsion free sheaf $\kF$ and locally free sheaf $\kG$ we have an isomorphism
  $$
  \mathbb{F}(\kF \otimes \kG) \cong \mathbb{F}(\kF)  \otimes \mathbb{F}(\kG),
  $$
  where the tensor product in $\Tri(X)$ is defined in a straightforward way.
\end{theorem}

\noindent
A proof of this Theorem can be found in  \cite[Theorem 1.3]{Thesis} or \cite[Theorem 16]{Survey}. \qed

\subsection{Torsion free sheaves on singular Weierstra\ss{} curves} Let $E$ be a singular Weierstra\ss{} cubic, $\PP^{1} \stackrel{\nu}\lar E$ its normalization. In both cases we have: $Z = \Spec(\CC)$, whereas
$$
\widetilde{Z} =
\left\{
\begin{array}{ll}
\Spec\bigl(\CC \times \CC\bigr)  & \mbox{\rm nodal case}\\
\Spec\bigl(\CC[\varepsilon]/(\varepsilon^2)\bigr)  & \mbox{\rm cuspidal case}.
\end{array}
\right.
$$
 Choose homogeneous coordinates $(z_0: z_1)$ on $\PP^1$ so that the condition (\ref{E:PreimageSing}) is satisfied. We shall identify  $y \in \CC$ we the point $(1: y)$ of $\PP^1$.
The choice of homogeneous coordinates on $\PP^1$  defines two  distinguished sections
$z_0, z_1 \in H^0\bigl(\PP^1, \kO_{\PP^1}(1)\bigr)$. As a consequence, for any $c \in \mathbbm{N}$, we get
a distinguished basis of the vector space $\Hom_{\PP^1}\bigl(\kO_{\PP^1}, \kO_{\PP^1}(c)\bigr)$ given by
the monomials $z_0^c, z_0^{c-1} z_1, \dots, z_1^c$.
It is convenient to choose  the following trivialization:
$$
\kO_{\PP^1}(1)\Big|_{Z} \stackrel{\xi}\lar \kO_Z,  \quad f \mapsto
\left\{
\begin{array}{lc}
\left(\dfrac{f}{z_0\big|_{V_0}}, \dfrac{f}{z_1\big|_{V_\infty}}\right) & \mbox{\rm nodal case}\\
\left(\dfrac{f}{z_1\big|_{V_\infty}}\right) & \mbox{\rm cuspidal case},
\end{array}
\right.
$$
where $f$ is a local section of $\kO_{\PP^1}(1)$, whereas $V_0$ and $V_\infty$ are open neighbourhoods of $0$ and $\infty$ respectively. This choice then defines trivializations $\kO_{\PP^1}(c)\Big|_{\widetilde{Z}} \stackrel{\xi_c}\lar \kO_Z$ for any $c \in \mathbb{Z}$, which are compatible with the  isomorphisms
$\kO_{\PP^1}(c_1) \otimes \kO_{\PP^1}(c_2) \cong \kO_{\PP^1}(c_1+c_2)$.
By Theorem \ref{T:keyonTF}, any torsion free sheaf $\kF$ of rank $n$ on $E$ is determined by the corresponding triple
$\bigl(\widetilde\kF, \CC^m, \Theta \bigr) \cong \mathbb{F}(\kF)$, where
\begin{equation}\label{E:triplesconcrete}
\Theta =
\left\{
\begin{array}{cl}
\bigl(\Theta_0, \Theta_\infty\bigr)  & \mbox{\rm nodal case}\\
\Theta_\circ + \varepsilon \Theta_\varepsilon & \mbox{\rm cuspidal case}
\end{array}
\right.
\quad \mbox{\rm with} \quad \Theta_0, \Theta_\infty, \Theta_\circ, \Theta_\varepsilon \in \mathsf{Mat}_{n \times m}(\CC).
\end{equation}
It is not difficult to show that
\begin{equation}\label{E:EulerCharact}
\chi(\kF)  = \mathsf{deg}\bigl(\widetilde{\kF}\bigr) + (m-n);
\end{equation}
see e.g.~\cite[Lemma 5.1.6]{BK4}.
Next, by the theorem of Birkhoff--Grothendieck, any vector bundle
$\widetilde\kF$ on $\PP^1$ splits into a direct sum of line bundles:
$
\widetilde\kF \cong \bigoplus\limits_{c \in \mathbb{Z}} \bigl(\kO_{\PP^1}(c)\bigr)^{\oplus m_c}.
$
Therefore, a description of the isomorphism  classes of objects in the category $\Tri(E)$ (and, as a consequence, of $\TF(E)$), reduces to a certain matrix problem \cite{DrozdGreuel, Thesis, Burban1, BodnarchukDrozd, Survey, BodnarchukDrozd2}.

\begin{lemma}
Let $E$ be a singular Weierstra\ss{} curve and $y \in \breve{E}$ a smooth point. In terms ob the identifications made above, we have:
\begin{equation}
\mathbb{F}\bigl(\kO_E(y)\bigr) \cong \bigl(\kO_{\PP^1}(1), \CC, \theta_y\bigr),
\end{equation}
where $\theta_y = \bigl((1), (-y)\bigr)$ in the nodal case and $\theta_y = \bigl((1) - \varepsilon (y)\bigr)$ in the cuspidal case.
\end{lemma}

\begin{proof} We refer to \cite[Lemma 5.1.2  and Lemma 5.1.27]{BK4} for a detailed treatment, both of the nodal and the cuspidal cases. Note that
in the nodal case, slightly different conventions to trivialize $\kO_{\PP^1}(1)\Big|_{\widetilde{Z}}$ were made in \cite[Lemma 5.1.2]{BK4}, leading to a deviation in the final description of the corresponding triples.
\end{proof}

\begin{proposition}\label{P:simpleTF}
Let $\kQ$ be a torsion free sheaf on a singular Weierstra\ss{} curve $E$, $n = \mathsf{rk}(\kQ)$ and $d = \chi(\kQ)$. Then the following results are true.
\begin{itemize}
\item If $\kQ$ is simple (i.e.~$\End_E(\kQ) \cong \CC$) then $\mathsf{gcd}(n, d) = 1$.
\item Other way around, for any $(n, d) \in \mathbb{N} \times \mathbb{Z}$ such that $\mathsf{gcd}(n, d) = 1$, there exists a unique (up a an isomorphism) simple torsion free and not locally free sheaf $\kQ$  of rank $n$ and Euler characteristic $d$.
\item For $d = 1$, this torsion free sheaf is given by the triple $\bigl(\kO_{\PP^1}^{\oplus n}, \CC^{n+1}, \Theta\bigr)$, where
\begin{equation}\label{E:canformnodal}
\Theta = \bigl(\Theta_0, \Theta_\infty\bigr) =
\left(
\left(
\begin{array}{cccc}
0 & 1 & \dots &  0 \\
\vdots & \vdots & \ddots & \vdots  \\
0 & 0 & \dots &  1
\end{array}
\right),
\;
\left(
\begin{array}{cccc}
1 &  \dots & 0 & 0  \\
\vdots & \ddots & \vdots & \vdots  \\
0 &  \dots  & 1 & 0
\end{array}
\right)
\right).
\end{equation}
in the nodal case, and by
\begin{equation}\label{E:canformcusp}
\Theta = \Theta_\circ + \varepsilon \Theta_\varepsilon =
\left(
\begin{array}{cccc}
1 &  \dots & 0 & 0  \\
\vdots & \ddots & \vdots & \vdots  \\
0 &  \dots  & 1 & 0
\end{array}
\right)
+ \varepsilon
\left(
\begin{array}{cccc}
0 & 1 & \dots &  0 \\
\vdots & \vdots & \ddots & \vdots  \\
0 & 0 & \dots &  1
\end{array}
\right),
\end{equation}
in the cuspidal case (here, we follow the notation (\ref{E:triplesconcrete})).
\item Let $O$ be the completion of the local ring $\kO_{s}$ and $\widetilde{O}$ be its normalization. Then we have:
$\widehat{\kQ}_s \cong O^{(n-1)} \oplus \widetilde{O}$.
\end{itemize}
\end{proposition}
\begin{proof}
The first two statements are consequences of  \cite[Corollary 4.3 and Corollary 4.5]{BK3}. It follows from the definition
of the category $\Tri(E)$ that triples defined by (\ref{E:canformnodal}) and (\ref{E:canformcusp}) are actually simple. Hence, the third statement follows from the formula (\ref{E:EulerCharact}). Finally, the last statement follows from the fact that the functor $\mathbb{F}$ from Theorem \ref{T:keyonTF}
commutes with the functors of taking localizations and completions.
\end{proof}

\begin{proposition}\label{P:LieAlgTorsionFree} Let $\kQ$ be a simple torsion free and not locally free sheaf on $E$ of rank $n$ and Euler characteristic one. Consider the sheaf of Lie algebras $\kA = \overline{\mbox{\it Ad}}(\kQ)$ on $E$ defined by the short exact sequence
\begin{equation}\label{E:AdTorsionFree}
0 \lar \kA \lar \mbox{\it End}_E(\kQ) \stackrel{\mathsf{tr}}\lar \widetilde{\kO} \lar 0,
\end{equation}
where  $\mathsf{tr}$ is the composition $\mbox{\it End}_E(\kQ) \stackrel{\mathsf{can}}\lar \nu_*\bigl(\mbox{\it End}_{\PP^1}(\widetilde{\kQ})\bigr) \stackrel{\widetilde{\mathsf{tr}}}\lar \widetilde\kO$ for $\widetilde{\kO} := \nu_*\bigl(\kO_{\PP^1}\bigr)$ and $\widetilde\kQ:= \nu_*\bigl(\nu^*\kF/\mathsf{tor}(\nu^*\kF)\bigr) \cong
\widetilde\kO^{\oplus n}$. Then the following results are true.
\begin{itemize}
\item $H^0(E, \kA) = 0 = H^1(E, \kA)$ and $\Gamma\bigl(\breve{E}, \kA\bigr) \cong \lieg \otimes \Gamma\bigl(\breve{E}, \kO_E\bigr)$.
\item $A_s$ is a coisotropic Lie subalgebra of the rational hull $A_K$.
\end{itemize}
In other words, the pair $(E, \kA)$ satisfies all conditions of Theorem \ref{T:CYBEtorsfree}.
\end{proposition}
\begin{proof}
Let $Q := \widehat{\kQ}_s$. It follows from the last part of Proposition \ref{P:simpleTF} that
\begin{equation}\label{E:localEnd}
\End_{O}(Q) \cong
\left(
\begin{array}{cccc}
\widetilde{O} & \widetilde{O} & \dots & \widetilde{O} \\
I & O & \dots & O \\
\vdots & \vdots & \ddots & \vdots \\
I & O & \dots & O
\end{array}
\right),
\end{equation}
where $I := \Hom_O\bigl(\widetilde{O}, O\bigr)$ is the conductor ideal. Therefore, the morphism $\mathsf{tr}$ in the exact sequence (\ref{E:AdTorsionFree}) is an epimorphism, as claimed. Consider the long exact cohomology sequence of (\ref{E:AdTorsionFree}). Taking into account that
$H^1(E, \widetilde{\kO}) \cong H^0\bigl(\PP^1, \kO_{\PP^1}\bigr) =  0$, we get:
$$
0 \lar H^0(E, \kA) \lar \End_E(\kQ) \stackrel{\mathsf{tr}}\lar H^0(E, \widetilde\kO) \lar H^1(E, \kA) \lar H^1(E, \mbox{\it End}_E(\kQ)\bigr) \lar 0.
$$
Since $H^0(E, \widetilde\kO) \cong \CC$ and $\mathsf{tr}\bigl(\mathsf{id}_{\kQ}\bigr) = n$, we conclude that $H^0(E, \kA) = 0$ and
$H^1(E, \kA) \cong  H^1(E, \mbox{\it End}_E(\kQ)\bigr)$. According to \cite[Corollary 4.3]{BK3}, there exists an auto--equivalence of the derived category
$D^b\bigl(\Coh(E)\bigr)$ mapping the torsion free sheaf $\kQ$ to the structure sheaf of the singular point $\CC_s$. Therefore,
$$
\Ext^1_E(\kQ, \kQ) \cong \Ext^1_E(\CC_s, \CC_s) \cong \Gamma\bigl(E, \mbox{\it Ext}^1_E(\CC_s, \CC_s)\bigr) \cong \CC^2,
$$
where the last isomorphism follows from a local computation. Next, we have the following local--to--global exact sequence:
\begin{equation*}
0 \lar H^1\bigl(E, \mbox{\it End}_E(\kQ)\bigr) \lar \Ext^1_E(\kQ, \kQ) \lar \Gamma\bigl(E, \mbox{\it Ext}^1_E(\kQ, \kQ)\bigr) \lar
H^2\bigl(E, \mbox{\it End}_E(\kQ)\bigr).
\end{equation*}
From the dimension reasons we have: $H^2\bigl(E, \mbox{\it End}_E(\kQ)\bigr) = 0$. The local structure of the $O$--module $Q$ is known:
$Q \cong O^{n-1} \oplus \widetilde{O}$. We have the following vanishings:
$$
\Ext^1_{O}(O, O) = \Ext^1_{O}(O, \widetilde{O}) = \Ext^1_{O}(\widetilde{O}, O) = 0
$$
(the last one follows for example from the fact that the local ring $O$ is Gorenstein). A local computation shows that both in nodal and cuspidal cases we have: $\Ext^1_{O}(\widetilde{O}, \widetilde{O}) \cong \CC^2$. This  implies that $H^1(E, \mbox{\it End}_E(\kQ)\bigr) = 0$.
The isomorphism  $\Gamma(\breve{E}, \kA) \cong \lieg \otimes \Gamma\bigl(\breve{E}, \kO_E\bigr)$ obviously follows from the fact that $\kQ\big|_{\breve{E}} \cong \kO_{\breve{E}}^{\oplus n}$. The first part of the proposition is proven.

\smallskip
\noindent
Consider  the following short exact sequence of $O$--modules:
$$
0 \lar A_s \lar
\left(
\begin{array}{cccc}
\widetilde{O} & \widetilde{O} & \dots & \widetilde{O} \\
I & O & \dots & O \\
\vdots & \vdots & \ddots & \vdots \\
I & O & \dots & O
\end{array}
\right)
  \stackrel{\mathsf{tr}}\lar \widetilde{O} \lar 0.
$$
It follows, that
$
A_s \subseteq
\left(
\begin{array}{cccc}
O & \widetilde{O} & \dots & \widetilde{O} \\
I & O & \dots & O \\
\vdots & \vdots & \ddots & \vdots \\
I & O & \dots & O
\end{array}
\right).
$
Since $I \cdot \widetilde{O} = I \subseteq O$, we may conclude  that  the trace form $A_s \times A_s \stackrel{\mathsf{tr}}\lar \widetilde{O}$ actually takes values in the ring $O$. It follows from the definition of Rosenlicht differential forms that  $A_s$ is a coisotropic Lie subalgebra of  $A_K$ with respect to the pairing (\ref{E:canpairing}).
\end{proof}

\begin{remark}
Using the technique of \cite{BodnarchukDrozd2} on the description of  Schurian objects in brick--tame  matrix problems, one can construct canonical forms (\ref{E:canformnodal}) and (\ref{E:canformcusp}) for arbitrary
mutually prime $(n, d) \in \mathbb{N} \times \mathbb{Z}$ and prove that the last statement of Proposition \ref{P:simpleTF} holds in the general case as well.
Therefore, Proposition \ref{R:AdP} is true in a greater generality, too.
\end{remark}

\subsection{Examples of solutions of CYBE, arising from torsion free sheaves}
Let $(E, \kA)$ be a pair from Proposition \ref{P:LieAlgTorsionFree}. The main result of this section are concrete formulae for the corresponding geometric $r$--matrix. We denote by
$\mathfrak{sl}_n(\CC) := \lieg = \lieg_+ \oplus \lieh \oplus \lieg_-$ the conventional triangular
decomposition of $\lieg$  into the direct sum of the Lie algebras of strictly upper triangular, diagonal and strictly lower triangular matrices. Let $\Phi_{\pm}$ be  the set of postive/negative roots of $\lieg$. Then we have: $\Phi_+ = \bigl\{(i, j) \in \mathbb{N}^2 \big| 1 \le i < j \le n\bigr\}$; for $\alpha = (i, j)  \in \Phi_\pm$ we write
$e_\alpha = e_{i,j}$.
Let $\xi = \exp\Bigl(\dfrac{2\pi i}{n}\Bigr)$ be a primitive $n$--th root of $1$ and $\zeta_j = \xi^j$. Then we  have the following  ``natural'' bases of the Cartan subalgebra  $\lieh$:
\begin{itemize}
\item $\bigl(g_1, \dots, g_{n-1}\bigr)$ with  $g_j = \mathsf{diag}\bigl(1, \zeta_j, \dots, \zeta_j^{n-1}\bigr)$ for
$1 \le j \le n-1$.
\item $\bigl(h_1, \dots, h_{n-1}\bigr)$ with  $h_j = \mathsf{diag}\bigl(0, \dots, 0, 1, -1, 0, \dots, 0\bigr)$, where $1$ stands at the $j$-th entry for $1 \le j \le n-1$.
\end{itemize}
In both cases, $\bigl(g_1^\ast, \dots, g_{n-1}^\ast\bigr)$ and $\bigl(h_1^\ast, \dots, h_{n-1}^\ast\bigr)$ denote the dual bases of $\lieh$.

\begin{theorem}\label{T:ExplicitSolFromTF} Let $\rho \in \Gamma(\breve{E}\times \breve{E}\setminus \Delta, \kA \boxtimes \kA)$ be the geometric $r$--matrix. Then there exists
a trivialization $\Gamma(\breve{E}, \kA) \stackrel{\xi}\lar \lieg \otimes \Gamma\bigl(\breve{E}, \kO_E\bigr)$ such the corresponding solution
$r = \rho^\xi$ of the classical Yang--Baxter equation is the following.
\begin{enumerate}
\item
If  $E$ is nodal then
$
r(x, y) = r_{\mathsf{st}}(x,y) + r_{\lieh} + r_{\mathsf{sp}},
$
where
$$
r_{\mathsf{st}}(x,y) = \frac{x}{y-x}\gamma + \Bigl(\frac{1}{2}\sum\limits_{j=1}^{n-1} g_j^* \otimes g_j + \sum\limits_{\alpha \in \Phi_+} e_{-\alpha} \otimes e_{\alpha}\Bigr)
$$
is the standard quasi--trigonometric $r$--matrix \cite{BelavinDrinfeld, KPSST},
$$
r_{\lieh} = \frac{1}{2n}\sum\limits_{j = 1}^{n-1} \Bigl(\frac{1 + \zeta_j}{1 - \zeta_j}\Bigr) g_{n-j} \otimes g_j \quad \mbox{\rm and} \quad
 r_{\mathsf{sp}}= \sum\limits_{\alpha \in \Phi_+} e_{-\alpha} \wedge \Bigl(\sum\limits_{k = 1}^{p(\alpha)} e_{\tau^k(\alpha)}\Bigr).
$$
Here,  for $\alpha = (i, j) \in \Phi_+$ we put: $p(\alpha) := i-1$ and $\tau(\alpha) := (i-1, j-1) \in \Phi_+$ provided  $i \ge 2$, whereas
$a \wedge b := a\otimes b - b \otimes a$ for $a, b \in \lieg$. In fact, $r_{\mathsf{sp}} \in \lieg \otimes \lieg$ is the constant solution of the classical Yang--Baxter equation, satisfying the constraint $r_{\mathsf{sp}} + r_{\mathsf{sp}}^{21} = \gamma$, which is given by the Belavin--Drinfeld triple
$(\Gamma_1, \Gamma_2, \tau)$, where 
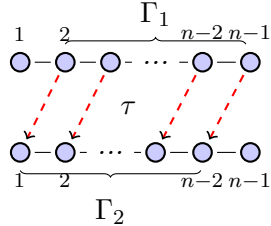
\begin{center}
 \begin{tikzpicture}[scale=0.12]
\node[ve](a1) at(0,10) {} ;
\node[ve](a2) at (5,10){} ;
\node[ve](a3) at (10,10){};

\node (a0) at (15,10){...};

\node[ve](a4) at (20.2,10){};
\node[ve](a5) at (25.5,10){};

\draw (a1)node[above=4pt]{$_1$}  --(a2)node[above = 4pt]{$_2$} --(a3)--(a0)--(a4)node[above = 4pt]{$_{n-2}$} -- (a5) node[above = 4pt]{$_{n-1}$}   ;

\node[ve](b1) at(0,0) {};
\node[ve](b2) at (5,0){};
\node (b3) at (10,0){...};

\node [ve] (b31) at (15,0){};

\node[ve](b4) at (20.2,0){};
\node[ve](b5) at (25.5,0){};

\draw (b1)node[below=4pt]{$_1$}  --(b2)node[below = 4pt]{$_2$} --(b3)--(b31)--(b4)node[below = 4pt]{$_{n-2}$} -- (b5) node[below = 4pt]{$_{n-1}$}   ;

\node at (12,5){$\tau$};

 \draw[conj,->](a2) to (b1);
 \draw[conj,->](a3) to (b2);
 \draw[conj,->](a4) to (b31);

 \draw[conj,->](a5) to (b4);

 \draw[decoration={brace},decorate]
  (5,12) -- node[above=3pt] {$\Gamma_1$} (25,12);

\draw[decoration={brace,mirror},decorate]
  (0, -2) -- node[below=7pt] {$\Gamma_2$} (20,-2);
 \end{tikzpicture}
 \end{center}
\item If  $E$ is cuspidal  then
$$
r(x, y) = \frac{\gamma}{y-x} + \sum\limits_{k = 1}^{n-1} h_k^* \wedge e_{k+1, k} + \sum \limits_{k \ge l+2} \bigl(\sum\limits_{j= 0}^{l-1} e_{l-j, k - j-1}\bigr)\wedge e_{k,l}.
$$
\end{enumerate}
\end{theorem}

\begin{proof} It was explained in \cite[Subsection 5.1.4]{BK4} and \cite[Corollary 6.5]{BH} that a choice of homogeneous coordinates on $\PP^1$ together with  a choice of trivializations $\kO_{\PP^1}(c)\Big|_{\widetilde{Z}} \stackrel{\xi_c}\lar \kO_{\widetilde{Z}}$ specify   a trivialization $\Gamma(\breve{E}, \kA) \stackrel{\xi}\lar \lieg \otimes \Gamma\bigl(\breve{E}, \kO_E\bigr)$ and an embedding
$\Gamma\bigl(E, \kA(x)\bigr) \stackrel{\bar\xi}\lar \lieg[z]$
such that the following diagram of vector spaces
\begin{equation}
\begin{array}{c}
\xymatrix{
\kA\big|_{x} \ar[d]_-{\xi_x} & \Gamma\bigl(E, \kA(x)\bigr) \ar[r]^-{\mathsf{ev}_y} \ar[l]_-{\mathsf{res}^\omega_x} \ar[d]^-{\bar\xi} & \kA\big|_{y} \ar[d]^-{\xi_y} \\
\lieg  & \mathsf{Sol} \ar[r]^{\overline{\mathsf{ev}}_y} \ar[l]_{\overline{\mathsf{res}}_x} & \lieg
}
\end{array}
\end{equation}
is commutative, where
\begin{itemize}
\item The vector space $\mathsf{Sol} := \mathsf{Im}(\bar\xi) \subset \lieg[z]$ has the following description:
\begin{equation}\label{E:SolNode}
\mathsf{Sol} = \left\{A + zB \left|
\begin{array}{lcl}
A \Theta_0 & = & - x \Theta_0 C\\
B \Theta_\infty & = & \Theta_\infty C
\end{array}
\mbox{\rm for some}\;  C \in \mathfrak{gl}_{n+1}(\CC)
\right.
\right\}
\end{equation}
in the nodal case, and
\begin{equation}\label{E:SolCusp}
\mathsf{Sol} = \left\{A + zB \left|
\begin{array}{lcl}
B \Theta_\circ & = & \Theta_\circ  D\\
A \Theta_\circ + B \Theta_\varepsilon & = & (\Theta_\varepsilon - x\Theta_\circ)\cdot  D
\end{array}
\;\mbox{\rm for some}\;  D \in \mathfrak{gl}_{n+1}(\CC)
\right.
\right\}
\end{equation}
in the cuspidal case.
\item For $\Phi = A + z B \in \mathsf{Sol}$ we have:
$$
\overline{\mathsf{ev}}_y(\Phi) = \frac{1}{y-x}(A + y B) \quad \mbox{and} \quad
\overline{\mathsf{res}}_x(\Phi) =
\left\{
\begin{array}{cc}
\dfrac{1}{x} \bigl(A + xB\bigr) & \mbox{nodal case} \\
A + x B & \mbox{cuspidal case}.
\end{array}
\right.
$$
\end{itemize}
We denote by $r_{x,y}^\sharp:= \overline{\ev_y} \circ \overline{\res}_x^{-1} \in \Hom_{\CC}(\lieg, \lieg)$, which is the image of the tensor $r(x,y) \in \lieg \otimes \lieg$ under the
isomorphism $\lieg \otimes \lieg \lar \Hom_{\CC}(\lieg, \lieg)$ induced by the trace form.

\smallskip
\noindent
\underline{1.~The nodal case}. To begin with, let us observe that
$$
\mathsf{Sol} = \left\{
- x\left(
\begin{array}{c|c}
D & 0 \\
\hline
a & \beta
\end{array}
\right) +
z \left(
\begin{array}{c|c}
\beta & b \\
\hline
0 & D
\end{array}
\right)
\left|
\begin{array}{l}
D \in \mathfrak{gl}_{n-1}(\CC), \; \beta = - \mathsf{tr}(D) \\
a, b \in \mathsf{Mat}_{1 \times (n-1)}(\CC)
\end{array}
\right.
\right\}
$$
From this description of $\mathsf{Sol}$ it is easy to conclude that for any $1 \le i < j \le n$ we have:
$$
\overline{\res}_x^{-1}\bigl(e_{i,j}\bigr) =
\left\{
\begin{array}{lc}
z e_{i,j} + (z-x)\bigl(e_{i-1, j-1}+ \dots + e_{1, j-i+1}\bigr) & i \ge 2 \\
z e_{1, j}                                                      & i = 1
\end{array}
\right.
$$
Consider the linear automorphism $\CC^n \stackrel{\tau}\lar \CC^n$ given by the formula
$\tau(\beta_1, \beta_2, \dots, \beta_n) = (\beta_2, \dots, \beta_n, \beta_1)$. Clearly, the spectrum of
$\tau$ is $\bigl\{1, \xi, \dots, \xi^{n-1}\bigr\} = \bigl\{1, \zeta_1, \dots, \zeta_{n-1}\bigr\}$.
We view $\CC^n$ as the space of the diagonal matrices of size $(n \times n)$. Then  $(g_1, \dots, g_{n-1})$ is a basis of $\lieh$ and moreover:
$
\tau(g_j) = \zeta_j g_j$ and $g_j^\ast = \dfrac{1}{n} g_{n-j}$ for all
$1 \le j \le n-1$.
For any $h \in \lieh$ we have:
$
h - \tau(h) \xrightarrow{\overline{\res}_x^{-1}} z h - x \tau(h) \xrightarrow{\overline{\ev}_y}
\dfrac{y}{y-x}\bigl(h - \tau(h)\bigr) + \tau(h).
$
Therefore, we get:
$$
\left\{
\begin{array}{ll}
r_{x,y}^{\sharp}(g_j) = \dfrac{y}{y-x} g_j + \dfrac{\zeta_j}{1-\zeta_j} g_j & \;  \mbox{for all} \quad
1 \le j \le n-1,\\
r_{x,y}^{\sharp}(e_{\alpha}) = \dfrac{y}{y-x} e_{\alpha} + \sum\limits_{\alpha \in \Phi_+} e_{-\alpha} \wedge \Bigl(\sum\limits_{k = 1}^{p(\alpha)} e_{\tau^k(\alpha)}\Bigr) & \; \mbox{for all} \quad
\alpha \in \Phi_+.
\end{array}
\right.
$$
From these formulae we obtain: $r(x,y) = r_{\mathsf{st}}(x,y) + r_{\lieh} + r_{\mathsf{sp}} = $
$$
\left(\frac{x}{y-x}\gamma +  \widehat{\gamma}\right) +
 \frac{1}{2n}\sum\limits_{j = 1}^{n-1} \Bigl(\frac{1 + \zeta_j}{1 - \zeta_j}\Bigr) g_{n-j} \otimes g_j + \sum\limits_{\alpha \in \Phi_+} e_{-\alpha} \wedge \Bigl(\sum\limits_{k = 1}^{p(\alpha)} e_{\tau^k(\alpha)}\Bigr),
$$
where $\widehat{\gamma} = \Bigl(\dfrac{1}{2}\sum\limits_{j=1}^{n-1} g_j^* \otimes g_j + \sum\limits_{\alpha \in \Phi_+} e_{-\alpha} \otimes e_{\alpha}\Bigr) = \Bigl(\dfrac{1}{2n}\sum\limits_{j=1}^{n-1} g_{n-j} \otimes g_j + \sum\limits_{\alpha \in \Phi_+} e_{-\alpha} \otimes e_{\alpha}\Bigr)$.

\smallskip
\noindent
\underline{2.~The  cuspidal case}.
Let $C \in \lieg$ and $\overline{\res}_x^{-1}(C) =: A + z B
\in \mathsf{Sol}$. Then we have:
$$
C \xrightarrow{\overline{\res}_x^{-1}} A + zB \xrightarrow{\overline{\ev}_y}
\frac{1}{y-x}\bigl(A + x B\bigr) + B = \frac{1}{y-x} C + B.
$$
It follows from the description (\ref{E:SolCusp}) that the matrices $A, B \in \mathfrak{sl}_{n}(\CC)$ and
$C \in \mathfrak{gl}_{n+1}(\CC)$ obey the following relation:
$$
\left\{
\begin{array}{lcl}
(C | 0) & = & \Theta_\varepsilon  D - (0 | B) \\
(B | 0) & = & \Theta_\circ D. \\
\end{array}
\right.
$$
Therefore, we get the following formulae:
$$
\left\{
\begin{array}{lcll}
r_{x,y}^{\sharp}(h_j) & = & \dfrac{1}{y-x} h_j + e_{j+1, j} & \;  \mbox{for all} \quad
1 \le j \le n-1,\\
r_{x,y}^{\sharp}(e_{j, j+1}) & = &  \dfrac{1}{y-x} e_{j, j+1} - h_j^\ast & \;  \mbox{for all} \quad
1 \le j \le n-1,\\
r_{x,y}^{\sharp}(e_{k, l}) & = &  \dfrac{1}{y-x} e_{k, l} - \bigl(e_{k, l-1} + e_{k-1, l-2} + \dots +
 e_{1, k-l}\bigr) & \; \mbox{for all} \quad
k - l \ge 2,
\end{array}
\right.
$$
which imply the result.
\end{proof}

\subsection{Geometrization of the solutions of CYBE for the Lie algebra $\mathfrak{sl}_2(\CC)$}
 According to results of Belavin and Drinfeld \cite{BelavinDrinfeld} and Stolin  \cite{Stolin},
there are precisely six non--equivalent solutions of the classical Yang--Baxter equation (say, with
one spectral parameter (\ref{E:CYBE1})):
one elliptic, two trigonometric and three rational for the Lie algebra $\mathfrak{sl}_2(\CC)$.
Fix the following basis of $\mathfrak{sl}_2(\CC)$:
$
h =
\left(
\begin{array}{cc}
1 & 0 \\
0 & -1
\end{array}
\right),$
$
e =
\left(
\begin{array}{cc}
0 & 1 \\
0 & 0
\end{array}
\right)$ and $
f =
\left(
\begin{array}{cc}
0 & 0 \\
1 & 0
\end{array}
\right).
$

\smallskip
\noindent
1.~The elliptic solution of Baxter
$$
r(z) =   \frac{\mathrm{cn}(z)}{\mathrm{sn}(z)} h \otimes h +
\frac{1+ \mathrm{dn}(z)}{\mathrm{sn}(z)}(e \otimes f  +
f \otimes e)  +
\frac{1 -  \mathrm{dn}(z)}
{\mathrm{sn}(z)}(e \otimes e + f \otimes f)
$$
corresponds to the geometric $r$--matrix associated with the pair $(E, \kA)$, where
$E$ is an elliptic curve and $\kA = \Ad(\kP)$ for a simple vector bundle $\kP$ of rank two and degree
one; see for example \cite[Theorem 4.3.6]{BK4}.

\smallskip
\noindent
2.~The trigonometric solution of Baxter
$$
r(z) =   \frac{\mathrm{cot}(z)}{2} h \otimes h +
\frac{1}{\mathrm{sin}(z)}(e \otimes f  +
f \otimes e)
$$
corresponds to the geometric $r$--matrix associated with  the pair $(E, \kA)$, where
$E$ is a nodal Weierstra\ss{} cubic  and  $\kA = \overline{\Ad}(\kQ)$ for the  simple torsion free sheaf
 $\kQ$ of rank two and  Euler characteristic one, which is not locally free; see Theorem \ref{T:ExplicitSolFromTF}.

\smallskip
\noindent
3.~The trigonometric solution of Cherednik \cite{Cherednik2}
$$
r(z) =   \frac{\mathrm{cot}(z)}{2} h \otimes h +
\frac{1}{\mathrm{sin}(z)}(e \otimes f  +
f \otimes e) + \mathrm{sin}(z) f \otimes f
$$
corresponds to the geometric $r$--matrix associated with  the pair $(E, \kA)$, where
$E$ is a nodal Weierstra\ss{} cubic  and  $\kA = {\Ad}(\kP)$ for a   simple locally  free sheaf
 $\kP$ of rank two and  degree one; see \cite[Subsection 5.2.3]{BK4}.

\smallskip
\noindent
4.~The rational solution of Yang
$
r(z) =   \dfrac{1}{z} \left(\dfrac{1}{2} h \otimes h +
e \otimes f  +
f \otimes e\right)
$
corresponds to the geometric $r$--matrix associated with  the pair $(E, \kA)$, where
$E$ is the cuspidal  Weierstra\ss{} cubic  and
$
\kA = \nu_\ast
\left(\mbox{\it sl}\left(
\begin{array}{cc}
\kI & \kI \\
\kI & \kI
\end{array}
\right)
\right),
$
where $\PP^1 \stackrel{\nu}\lar E$ is the normalization map and $\kI$ is the ideal sheaf
of the point $\nu^{-1}(s)$; see Remark \ref{R:Yangian}.

\smallskip
\noindent
5.~The rational solution of Stolin \cite{Stolin}
$$
r(z) =   \frac{1}{z} \left(\dfrac{1}{2} h \otimes h +
e \otimes f  +
f \otimes e\right) + z(f \otimes h + h \otimes f) - z^3 f  \otimes f
$$
corresponds to the geometric $r$--matrix associated with  the pair $(E, \kA)$, where
$E$ is the cuspidal  Weierstra\ss{} cubic  and  $\kA = {\Ad}(\kP)$ for a   simple locally  free sheaf $\kP$ of rank two and  degree one; see \cite[Subsection 5.2.5]{BK4}.

\smallskip
\noindent
6.~The rational solution
$$
r(z) =   \frac{1}{z} \left(\dfrac{1}{2} h \otimes h +
e \otimes f  +
f \otimes e\right) + \dfrac{1}{2}\bigl(h \otimes f - f \otimes h\bigr)
$$
corresponds to the geometric $r$--matrix associated with  the pair $(E, \kA)$, where
$E$ is the cuspidal  Weierstra\ss{} cubic  and  $\kA = \overline{\Ad}(\kQ)$ for  the  simple torsion free sheaf
 $\kQ$ of rank two and  Euler characteristic one, which is not locally free;  see Theorem \ref{T:ExplicitSolFromTF}.

\end{document}